\documentclass[english]{amsart}

\usepackage{a4wide}

\newtheorem{theorem}{Theorem}[section]
\newtheorem{lemma}[theorem]{Lemma}
\newtheorem{proposition}[theorem]{Proposition}
\newtheorem{corollary}[theorem]{Corollary}

\theoremstyle{definition}
\newtheorem{definition}[theorem]{Definition}

\theoremstyle{remark}
\newtheorem*{remark}{Remark}

\numberwithin{equation}{section}

\usepackage[indent = 0.3in]{}

\usepackage{graphicx}

\usepackage{xcolor}

\usepackage[fixlanguage]{babelbib}
\selectbiblanguage{english}

\usepackage{comment}

\usepackage{amsfonts,amssymb,amsbsy,amsmath,mathrsfs,amsthm}
\usepackage{bbm}

\usepackage{commath}

\usepackage{enumitem}

\usepackage{hyperref}
\usepackage{url}

\hypersetup{pdfpagemode=UseNone, pdfstartview=FitH,
  colorlinks=true,urlcolor=red,linkcolor=blue,citecolor=black,
  pdftitle={Default Title, Modify},pdfauthor={Your Name},
  pdfkeywords={Modify keywords}}

\RequirePackage{babel}

\newcommand{\dx}{\, d x}
\newcommand{\dy}{\, d y}
\newcommand{\dt}{\, d t}
\newcommand{\ds}{\, d s}

\newcommand{\R}{\mathbb{R}}
\newcommand{\Rn}{\mathbb{R}^{n+1}}

\newcommand{\Z}{\mathbb{Z}}
\newcommand{\N}{\mathbb{N}}

\def\Xint#1{\mathchoice
   {\XXint\displaystyle\textstyle{#1}}%
   {\XXint\textstyle\scriptstyle{#1}}%
   {\XXint\scriptstyle\scriptscriptstyle{#1}}%
   {\XXint\scriptscriptstyle\scriptscriptstyle{#1}}%
   \!\int}
\def\XXint#1#2#3{{\setbox0=\hbox{$#1{#2#3}{\int}$}
     \vcenter{\hbox{$#2#3$}}\kern-.5\wd0}}

\def\dashint{\Xint-}

\DeclareMathOperator*{\esssup}{ess\,sup}
\DeclareMathOperator*{\essinf}{ess\,inf}

\newcommand{\dist}[0]{\operatorname{dist}}

\usepackage{cite}
\makeatletter
\newcommand{\citecomment}[2][]{\citen{#2}#1\citevar}
\newcommand{\citeone}[1]{\citecomment{#1}}
\newcommand{\citetwo}[2][]{\citecomment[,~#1]{#2}}
\newcommand{\citevar}{\@ifnextchar\bgroup{;~\citeone}{\@ifnextchar[{;~\citetwo}{]}}}
\newcommand{\citefirst}{\@ifnextchar\bgroup{\citeone}{\@ifnextchar[{\citetwo}{]}}}

\makeatother

\begin{document}

\title{Parabolic weak porosity and parabolic Muckenhoupt distance functions}

\author{Henri Lahdelma}
\address[H.L.]{Department of Mathematics, School of Science, Aalto University, P.O. Box 11100, 00076 Aalto, Finland}
\email{henri.lahdelma@aalto.fi}
\thanks{The first author would like to thank Juha Kinnunen for discussions and suggestions that have been helpful in the research project. The first author was supported by the Magnus Ehrnrooth Foundation.}

\author{Kim Myyryl\"ainen}
\address[K.M.]{University of Jyvaskyla, Department of Mathematics and Statistics, P.O. Box 35, FI-40014 University of Jyvaskyla, Finland}
\email{kim.k.myyrylainen@jyu.fi}
\thanks{The second author was supported by the Research Council of Finland project 360185, the Emil Aaltonen Foundation, Charles University PRIMUS/24/SCI/020 and Research Centre program No.~UNCE/24/SCI/005.
}

\author{Antti V. V\"ah\"akangas}
\address[A.V.V.]{University of Jyvaskyla, Department of Mathematics and Statistics, P.O. Box 35, FI-40014 University of Jyvaskyla, Finland} 
\email{antti.vahakangas@iki.fi}



\begin{abstract}
We develop the parabolic weak porosity to characterize the parabolic Muckenhoupt $A_1$ weights with time-lag. Our main result shows that a nonempty closed set is parabolic weakly porous if and only if the parabolic distance function of the set to a negative power is in the parabolic Muckenhoupt $A_1$ class. We apply a novel stopping time argument in combination with the translation and doubling results for the parabolic weakly porous sets.
\end{abstract}

\maketitle

\section{Introduction}

The regularity of distance weights has proven to be an interesting topic in harmonic analysis with applications to PDE theory. 
Given a nonempty set $E$, a distance weight $w \in L_\textnormal{loc}^1$ is of the form
\begin{align*}
    w(\cdot) = \dist(\cdot,E)^{-\alpha}
\end{align*}
for some $\alpha >0$. The main interest has been to characterize all possible sets $E$ and exponents $\alpha$ such that a distance weight $w$ exhibits regular behaviour, in particular, the Muckenhoupt properties. For earlier results concerning distance weights in the elliptic setting, see \cite{ACDT2014, DILLTV2019, Aikawa1991, DuranGarcia2010}. Our goal is to advance the higher dimensional one-sided theory of Muckenhoupt distance weights, that is, the forward-in-time parabolic Muckenhoupt distance weights by introducing the concept of parabolic weakly porosity. 

For the set $E$ to induce Muckenhoupt $A_q$ distance weights in the elliptic setting it is enough that $E$ is a porous set, a condition analyzed in \cite{Luukkainen1998, DILLTV2019}. The concept of porosity was expanded into the weak porosity by Vasin \cite{Vasin05}, characterizing the $A_1$ distance weights on a unit circle. These results were later generalized to $\R^n$ by Anderson et al. \cite{ALMV24}, while also analyzing the exponent $\alpha$ by refining the Assouad codimension into the Muckenhoupt exponent. The theory has quickly expanded thereafter. Similar characterizations for $A_q$ distance weights for $q>1$ were independently studied by Gómez \cite{Vargas2026} and Pasquariello and Uriarte-Tuero \cite{PasquarielloUriarte-Tuero2025}, considering also a more general median porosity. The weak porosity has also inspired similar characterizations for $A_1$ distance weights in metric measure spaces \cite{Mudarra25, AimarGomezVargas24}, and has applications, for example, in Carleson embeddings \cite{Vasin2025}.

The one-sided theory of the weak porosity has already been studied on a real line by Aimar et al. \cite{AGIM25} to characterize the one-sided $A_1$ distance weights. Their left- and right-sided versions of weak porosity adapt to the notion of time by limiting the holes or pours of weakly porous sets. This results in information passing only in one direction along the real line, which leads to interesting forward-in-time versions of the elliptic results, such as the doubling of the maximal hole. However, many of the methods on the real line are specialized on one dimension, which makes it challenging to generalize the one-sided theory to higher dimensions.

In our paper, we show a full characterization of the parabolic $A_1$ distance weights via parabolic weakly porous sets. The parabolic Muckenhoupt theory was introduced by Kinnunen and Saari \cite{kinnunenSaariMuckenhoupt, kinnunenSaariParabolicWeighted} as an $n+1$-dimensional generalization of the one-sided Muckenhoupt theory, with later research of the parabolic $A_q$ and parabolic BMO having followed \cite{KinnunenMyyry2024, KinnunenMyyryYang2024}. The parabolic $A_1$ classes, denoted by $A_1^+(\gamma)$, consist of weights $w$ satisfying
\begin{align*}
    \dashint_{R^-(\gamma)} w(x,t) \dx \dt \leq C \essinf_{(x,t)\in R^+(\gamma)} w(x,t) \quad \textnormal{for every} \quad R^\pm(\gamma) \subseteq \Rn.
\end{align*}
Here $R^\pm(\gamma)$ are space-time rectangles with a fixed time-lag $0< \gamma<1$, see Section \ref{subsec: parabolic A1}. The time-lag between $R^-(\gamma)$ and $R^+(\gamma)$ is a profound feature of the parabolic theory, since the underlying PDE, that is, the doubly nonlinear equation, does not support Harnack's inequality without the time-lag, see \cite{GianazzaVespri2006, KinnunenKuusi07}. This reflects to the parabolic Muckenhoupt theory by the time-lag being a crucial part of chaining arguments. Moreover, the time-lag is unique to the higher dimensions, since it is shown in \cite{KongYangYuan26} that the time-lag can be eliminated when working on the real line. For more results of the parabolic theory, such as two-weight versions with estimates for the fractional maximal functions, see \cite{Cruz-Uribe2025Two-weightFunctions, KongYangYuanZhu2025, MaHeYan23, CKYY2026}.

Our main theorem, Theorem \ref{theo: grand theorem}, characterizes the parabolic $A_1$ distance weights, that is, $A_1^+(\gamma)$ for $0<\gamma<1$ via parabolic weakly porous sets, where the distance functions use the parabolic distance metric, see Section \ref{sec: preliminaries}.
\begin{theorem}
    A parabolic distance weight of negative power with respect to a nonempty closed set $E\subseteq \Rn$ belongs to $A_1^+(\gamma)$ if and only if $E$ is parabolic weakly porous.
\end{theorem}
This problem has already been presented by Kong et al. \cite{KongYangYuan26} in their work of parabolic BLO with nonzero time-lag. They showed that $A_1^+(\gamma)$ distance weights induce forward-in-time weak porosity, however, the reverse direction was left as an open problem. The reverse direction, appearing more difficult, will be answered in our work. 

The elliptic weak porosity of \cite{ALMV24} is based on dyadic methods. Since dyadic methods have also worked well in the parabolic setting using parabolic rectangles, for instance, in Calderón--Zygmund type decompositions \cite{Cruz-Uribe2025Two-weightFunctions, KongYangYuanZhu2025, KinnunenMyyryYang2024, KinnunenMyyry2025}, this has motivated to develop the parabolic weak porosity using the dyadic division of a parabolic rectangle. The parabolic rectangles $R^\pm(\gamma)$ are space-time cylinders that scale to the power of some fixed $p>1$ along the temporal axis. This different scaling with respect to spatial and temporal directions is natural to assume due to the underlying PDE. Moreover, it plays an important role in the chaining arguments and motivates the use of the parabolic metric. Especially the case $p=2$ is the most fundamental, as then the underlying doubly nonlinear equation reduces to the heat equation. Regardless, it is possible to obtain a well-defined dyadic system for any $p>1$, see Section \ref{sec: preliminaries}.

Theorem \ref{theo: doubling porosity} and Corollary \ref{cor: translation} show the forward-in-time doubling features and the time-lag invariance of parabolic weakly porous sets. The forward-in-time aspect is necessary, and standard doubling of \cite{ALMV24} cannot be expected in the parabolic setting. This is apparent since the aforementioned results take advantage of the nonzero time-lag $\gamma>0$ and the parabolic geometry $p>1$ in a typical chaining argument appearing often in the parabolic theory, for instance in \cite{kinnunenSaariMuckenhoupt, KinnunenMyyryYang2024}. As a benefit of these results, we can restrict the parabolic weak porosity under integer translations further motivating the dyadic approach.

Our strategy of proving the main theorem, Theorem \ref{theo: grand theorem}, is based on the relationship between the porosity constants and the exponent $\alpha$ of a parabolic distance weight. We have decided to call this relationship the $\alpha$-improvement of a parabolic weakly porous set, see Definition \ref{def: alpha improvement}. The $\alpha$-improvement seems to offer a robust machinery to obtain $A_1$-type inequalities. As a matter of fact, $\alpha$-improvement itself already establishes a full characterization of $A_1^+(\gamma)$ distance weights, see Theorem \ref{theo: A1 => wp} and Corollary \ref{cor: alpha improvement => A1}. This approach shares similarities with \cite[Theorem 4.1]{Vargas2026}, where the exponent $\alpha$ plays a similar role for $A_q$ weights with $q>1$. In their and our works the estimates are quantitative and preserve information of the exponent $\alpha$. Thus, it seems that this type of approach is effective to fully understand the theory of weak porosity.

The greatest difficulties of our work arise in proving the $\alpha$-improvement of parabolic weakly porous sets, see Lemma \ref{lem: wp => polynomial decay}. The approaches of the one-sided theory on the real line in \cite{AGIM25} do not seem to generalize to the parabolic setting. Moreover, the direct approaches of \cite{ALMV24, Vargas2026} in the elliptic case are also difficult to implement in the forward-in-time context as such. However, in Section~\ref{sec: stopping conditions} we present a novel stopping time argument in the parabolic setting, where we recursively apply the elliptic strategy. With this stopping time argument we are able to derive an exponential decay estimate, see Lemma~\ref{lem: exponential decay}. This estimate is the key to the $\alpha$-improvement of parabolic weakly porous sets, giving the full characterization of parabolic distance weights.

\section{Preliminaries} \label{sec: preliminaries}
Our preliminaries mainly consist of notational aspects, however, we will also have discussion of our methods, namely, the dyadic approach in the parabolic setting.

\subsection{Basic assumptions}
The space that we focus on is 
\begin{align*}
    \Rn = \big\{(x,t) \, : \, x = (x_1,x_2,\cdots,x_n) \in \R^n, \, t\in \R \big\},
\end{align*}
where $n \in \N$, and where the extra dimension is given to a time variable. We use the functions $\textnormal{pr}_x(\cdot)$ and $\textnormal{pr}_t(\cdot)$ to describe the spatial and temporal projections of sets respectively. For any set $E \subseteq \Rn$ we define a translation by $(y,s) \in \Rn$ as
\begin{align*}
    E + (y,s) = \big\{(x,t) + (y,s) \in \Rn \,:\, (x,t) \in E \big\}.
\end{align*}
We denote by $|\cdot|$ the Lebesgue measure in $\Rn$ and the temporal distances between any two points $(x,t),(y,s) \in \Rn$ by
\begin{align*}
    \quad \dist_t\big((x,t),(y,s)\big) = |t-s|.
\end{align*}

Additionally, we define a parabolic metric in this paper to simplify certain expressions. Given a parabolic constant $1< p<\infty$, the parabolic distance metric is
\[
\dist_p\big((x,t),(y,s)\big)=\max\big\{\|x-y\|_\infty,\lvert t-s\rvert^{\frac1p}\big\}.
\]
One can verify that $\dist_p(\cdot,\cdot)$ is indeed a metric,
and for any two sets $E,F \subseteq \Rn$ we define naturally
\[
\dist_p(E,F) = \inf \big\{ \dist_p\big((x,t),(y,s)\big) : (x,t)\in E, (y,s)\in F \big\}.
\]
Unless specifically mentioned, we always assume $p>1$. We also define for convenience the parabolic diameter of sets as
\begin{align*}
    \text{diam}_p(E) = \sup_{(x,t),(y,s) \in E} \dist_p\big((x,t),(y,s)\big).
\end{align*}

We will replace the basic sets, such as cubes, with parabolic rectangles. A parabolic rectangle $R \subseteq \Rn$ at a point $(x,t)$ and with a side length $L>0$ is defined by
\[
R = R(x,t,L) = Q(x,L) \times [t - L^p, t),
\]
where $Q(x,L) \in \R^n$ is a half open cube centered at $x$ with a side length $L$. These rectangles follow the parabolic geometry, which makes them better suited for the geometric arguments. The spatial side length of the parabolic rectangle is denoted by $l_x(\cdot)$ and the temporal side length is denoted by $l_t(\cdot)$. Note that if the side length of a parabolic rectangle $R$ is $L>0$, then $l_x(R) = L$ and $l_t(R) = L^p$.

We also need certain type of truncations of parabolic rectangles. For a truncation parameter $0\leq \gamma \leq 1/2$ we define a truncation of a parabolic rectangle
\begin{align*}
    R(x,t,L,\gamma) = R(\gamma) = Q(x,L) \times [t - L^p, t - \gamma L^p).
\end{align*}
The truncation keeps the spatial side length the same as $l_x\big(R(\gamma)\big) = L$, however, the temporal side length is now $l_t\big(R(\gamma)\big) = (1-\gamma)L^p$. For simplicity, we call the truncated parabolic rectangles also parabolic rectangles or just rectangles.

Temporal translations are particularly important in the parabolic theory. Hence, for any parabolic rectangle $R(\gamma) \subseteq \Rn$ with $0\leq \gamma \leq 1/2$ we define the notation
\begin{align*}
    R^\theta(\gamma) = R(\gamma) + \big(0,  \theta l_t\big(R(\gamma)\big)\big),
\end{align*}
where $\big(0,  \theta l_t\big(R(\gamma)\big)\big) \in \Rn$ and $\theta \in \R$ is a translation parameter. To keep track of translations, we often study the distance between the lower faces of $R(\gamma)$, which we denote by $\partial_{\textnormal{low}} R \subseteq \partial R$. Consequently, we have the relationship
\begin{align*}
    \dist_t\big(\partial_{\textnormal{low}}R(\gamma),\partial_{\textnormal{low}}R^\theta(\gamma)\big) = |\theta| l_t\big(R(\gamma)\big).
\end{align*}

Finally, for any $f\in L_{\textnormal{loc}}^1$ and measurable $A \subseteq \Rn$ with $0<|A| <\infty$, we denote
\begin{align*}
    \dashint_A f \dx \dt = \frac{1}{|A|} \int_A f \dx \dt.
\end{align*}

\subsection{Parabolic dyadic division}
Our approaches are heavily based on dyadic methods. We begin by introducing a dyadic division for a parabolic rectangle. There are two main difficulties in the definition of the dyadic lattice. Firstly, while each spatial edge of any parabolic rectangle $R \subseteq \Rn$ can be divided dyadically, the same does not apply for the temporal edges.

Ideally, we would like to divide the temporal edges into $2^p$ parts. However, if $2^p$ is not an integer, then necessarily the subrectangles have to overlap or they cannot be the same shape as $R$. We have opted to sacrifice the similarity of the shape in our definition by using the truncated subrectangles.

The second difficulty arises from much deeper geometrical arguments of Section \ref{sec: stopping conditions}. In short, the division rate of each dyadic layer may need to be increased to extend the results to the case when $p$ is small. We choose an integer $d \geq 1$ as the division rate such that
\begin{align} 
    dp \geq \log_2 (9). \label{eq: division rate}
\end{align}
The simplest choice would be $d = 4$, which works for every $p>1$.

We obtain the first dyadic layer $\mathcal{D}_1(R)$ by dividing the spatial edges into $2^d$ parts and the temporal edges into $ \big\lceil 2^{dp} \big\rceil$ parts. Now $\mathcal{D}_1(R)$ consists of truncated parabolic rectangles $P(\gamma)$ for some $0\leq \gamma <1$ such that
\begin{align*}
    l_x\big(P(\gamma)\big) = \frac{1}{2^d} l_x(R) \quad \text{and} \quad l_t\big(P(\gamma)\big) = \frac{1}{\lceil2^{dp}\rceil} l_t(R) =  \frac{1}{\lceil2^{dp}\rceil} l_x(R)^p.
\end{align*}
On the other hand, $\gamma$ fixes the temporal side length of $P(\gamma)$ such that 
\begin{align*}
    l_t\big(P(\gamma)\big) = (1-\gamma) l_x\big(P(\gamma)\big)^p = (1-\gamma) \frac{1}{2^{dp}}l_x(R)^p.
\end{align*}
This means that $\gamma$ satisfies
\begin{align*}
    (1-\gamma) = \frac{2^{dp}}{\lceil2^{dp}\rceil}.
\end{align*}
Observe that $ \big\lceil 2^{dp} \big\rceil$ approximates $2^{dp}$ with upper and lower bound
\begin{align*}
    \frac{1}{2} \leq \frac{2^{dp}}{\big\lceil 2^{dp} \big\rceil} \leq 1,
\end{align*}
implying $0\leq \gamma \leq 1/2$.

To generate the higher order dyadic layers while preserving the nestedness of the dyadic layers, we have to generalize the dyadic division to truncated parabolic rectangles $R(\gamma)$ for any $0\leq \gamma \leq 1/2$.  We define the general first order dyadic layer $\mathcal{D}_1\big(R(\gamma)\big)$ in the following way. Divide the spatial sides into $2^d$ equally long intervals so that the spatial side length of each $P(\alpha)\in \mathcal{D}_1\big(R(\gamma)\big)$ satisfies
\begin{align}
    l_x\big(P(\alpha)\big) = \frac{1}{2^d} l_x\big(R(\gamma)\big) = \frac{1}{2^d} l_x(R). \label{eq: lx}
\end{align}

To ensure that $0 \leq \alpha \leq 1/2$ stays within the desired range, we alternate dividing the temporal edges into $ \big\lceil 2^{dp} \big\rceil$ or $\big\lfloor 2^{dp} \big\rfloor$ intervals. We set number of temporal divisions
\begin{align}
    k = 
    \begin{cases}
        \big\lceil 2^{dp} \big\rceil, \quad 0 \leq \gamma \leq 1- \frac{2^{dp}+1}{2^{dp+1}}, \\
        \big\lfloor 2^{dp} \big\rfloor, \quad 1- \frac{2^{dp}+1}{2^{dp+1}} < \gamma \leq \frac{1}{2} ,
    \end{cases} \label{eq: k}
\end{align}
and divide the temporal edges into $k$ equally long intervals. We will show that the rectangles $P(\alpha)$ are well-defined under this definition. 

\begin{proposition} \label{prop: dyadic division}
    Let $R(\gamma) \subseteq \Rn$ be a parabolic rectangle with $0\leq \gamma \leq  1/2$. Then, the dyadic layer $\mathcal{D}_1\big(R(\gamma)\big)$ consists of parabolic rectangles $P(\alpha) \subseteq R(\gamma)$ with $0 \leq \alpha \leq 1/2$. In particular,
    \begin{align*}
        l_x\big(P(\alpha)\big) = \frac{1}{2^d} l_x\big(R(\gamma)\big)
    \end{align*}
    and
    \begin{align*}
        \frac{1}{2} \cdot \frac{1}{2^{dp}}  l_t\big(R(\gamma)\big) \leq l_t\big(P(\alpha)\big)\leq 2\cdot \frac{1}{2^{dp}} l_t\big(R(\gamma)\big).
    \end{align*}
\end{proposition}

\begin{proof}
    Let $R(\gamma) \subseteq \Rn$ be a parabolic rectangle with $0\leq \gamma \leq  1/2$. We already showed in (\ref{eq: lx}) that the spatial side lengths of each $P(\alpha) \in \mathcal{D}_1\big(R(\gamma)\big)$ follow the dyadic division. Observe that the temporal side length is also uniquely determined by the spatial side length up to the parameter $0\leq \alpha <1$. Therefore, we obtain
    \begin{align}
        \begin{split}l_t\big(P(\alpha)\big) &= (1-\alpha) l_x\big(P(\alpha)\big)^p \\
        &= (1-\alpha) \frac{1-\gamma}{1-\gamma} \frac{1}{2^{dp}}l_x\big(R(\gamma)\big)^p \\
        &= \frac{1-\alpha}{1-\gamma}\frac{1}{2^{dp}} l_t\big(R(\gamma)\big).
        \end{split}\label{eq: dyadic division 1}
    \end{align}

    On the other hand, recalling $k$ from (\ref{eq: k}), dividing the temporal edges of $R(\gamma)$ into $k$ intervals implies that
    \begin{align*}
        l_t\big(P(\alpha)\big) = \frac{1}{k} l_t\big(R(\gamma)\big).
    \end{align*}
    Combining both expressions of the temporal side length fixes $\alpha$ as
    \begin{align*}
        \alpha = 1 - (1-\gamma)\frac{2^{dp}}{k}.
    \end{align*}

    We show now that $0\leq \alpha \leq 1/2$. We check first the case $0\leq \gamma \leq 1 - 2^{-(dp+1)}(2^{dp}+1)$, that is, $k =\big\lceil 2^{dp} \big\rceil $ by (\ref{eq: k}). We have the lower bound
    \begin{align*}
        \alpha = 1 - (1-\gamma)\frac{2^{dp}}{\big\lceil 2^{dp} \big\rceil} \geq 1 - \frac{2^{dp}}{2^{dp}} = 0,
    \end{align*}
    and the upper bound also follows as
    \begin{align*}
        \alpha = 1 - (1-\gamma)\frac{2^{dp}}{\big\lceil 2^{dp} \big\rceil} \leq 1 - \frac{2^{dp}+1}{2^{dp+1}}\frac{2^{dp}}{2^{dp} +1} = \frac{1}{2}. 
    \end{align*}
    For the case $ 1 - 2^{-(dp+1)}(2^{dp}+1) < \gamma \leq 1/2$, that is, $k =\big\lfloor 2^{dp} \big\rfloor $ by (\ref{eq: k}), we have the upper bound
    \begin{align*}
        \alpha = 1 - (1-\gamma)\frac{2^{dp}}{\big\lfloor 2^{dp} \big\rfloor} \leq 1 - \frac{1}{2}\frac{2^{dp}}{2^{dp}}  = \frac{1}{2},
    \end{align*}
    and the lower bound follows similarly as
    \begin{align*}
        \alpha = 1 - (1-\gamma)\frac{2^{dp}}{\big\lfloor 2^{dp} \big\rfloor} \geq 1 - \frac{2^{dp}+1}{2^{dp+1}} \frac{2^{dp}}{\big\lfloor 2^{dp} \big\rfloor} \geq  1 - \frac{2^{dp}+1}{2\big\lfloor 2^{dp} \big\rfloor}\geq 0.
    \end{align*}
    Since $\gamma,\alpha \in [0,1/2]$, by (\ref{eq: dyadic division 1}) we get 
        \begin{align*}
            \frac{1}{2} \cdot \frac{1}{2^{dp}}  l_t\big(R(\gamma)\big) \leq l_t\big(P(\alpha)\big)\leq 2\cdot \frac{1}{2^{dp}} l_t\big(R(\gamma)\big).
        \end{align*}
    \end{proof}

\subsection{Parabolic dyadic lattice}
In the rest of the section, we introduce the main tools of the paper. For any $R(\gamma_0) \in \Rn$ with $0\leq \gamma_0 \leq 1/2$, we use the first order dyadic layer $\mathcal{D}_1\big(R(\gamma_0)\big)$ to generate the whole dyadic lattice recursively. We set $\mathcal{D}_0\big(R(\gamma_0)\big) = \big\{ R(\gamma_0)\big\}$. Then, we recursively construct the dyadic lattice of any order by
\begin{align*}
    \mathcal{D}_{i+1}\big(R(\gamma_0)\big) = \bigcup_{P(\gamma)\in \mathcal{D}_i(R(\gamma_0))}\mathcal{D}_1\big(P(\gamma)\big),
\end{align*}
where $i \in \N$ and $0\leq \gamma \leq 1/2$. The different order dyadic subrectangles together define the whole dyadic lattice $\mathcal{D}(R(\gamma_0))$ in the parabolic geometry as
\begin{align*}
    \mathcal{D}\big(R(\gamma_0)\big) = \bigcup_{i=0}^\infty \mathcal{D}_i\big(R(\gamma_0)\big).
\end{align*} Thanks to Proposition \ref{prop: dyadic division} the definitions are valid. In particular, the side length of any $P(\gamma) \in \mathcal{D}_m\big(R(\gamma_0)\big)$ follows the dyadic scale.

\begin{corollary} \label{cor: dyadic scale}
    Let $R(\gamma_0) \subseteq \Rn$ be a parabolic rectangle with $0\leq \gamma_0 \leq  1/2$. Then, for any $i\in \N$ the dyadic layer $\mathcal{D}_i\big(R(\gamma_0)\big)$ consists of parabolic rectangles $P(\gamma) \subseteq R(\gamma_0)$ with $0 \leq \gamma \leq 1/2$. In particular,
    \begin{align*}
        l_x\big(P(\gamma)\big) = 2^{-id} l_x\big(R(\gamma_0)\big)
    \end{align*}
    and
    \begin{align*}
        \frac{1}{2} \cdot 2^{-idp}  l_t\big(R(\gamma_0)\big) \leq l_t\big(P(\gamma)\big)\leq 2\cdot 2^{-idp} l_t\big(R(\gamma_0)\big).
    \end{align*}
\end{corollary}

\begin{proof}
    Let $R(\gamma_0)$ be a parabolic rectangle with $0 \leq \gamma_0 \leq 1/2$. We will prove the statement via induction. We start with the base case, that is, $i = 0$. Now, $\mathcal{D}_0\big(R(\gamma_0)\big) = \big\{R(\gamma_0)\big\}$, and thus the statement is clearly true.

    Then, fix $i \in \N$. Assume inductively that $\mathcal{D}_i\big(R(\gamma_0)\big)$ consists of parabolic rectangles $P(\gamma)$ with $0\leq \gamma \leq 1/2$ such that
    \begin{align*}
        l_x\big(P(\gamma)\big) = 2^{-id} l_x\big(R(\gamma_0)\big)
    \end{align*}
    for some $i \in \N$. Take any $Q(\alpha) \in \mathcal{D}_{i+1}\big(R(\gamma_0)\big)$ with $0 \leq \alpha < 1$. The definition of the dyadic lattice implies that there exists $P(\gamma) \in \mathcal{D}_i\big(R(\gamma_0)\big)$ such that $Q(\alpha) \in \mathcal{D}_1\big(P(\gamma)\big)$. By Proposition \ref{prop: dyadic division} $Q(\alpha)$ is a parabolic rectangle with the truncation parameter satisfying $0\leq \alpha \leq 1/2$, and the spatial side length follows
    \begin{align*}
        l_x\big(Q(\alpha)\big) = 2^{-d} l_x\big(P(\gamma)\big) = 2^{-(i+1)d} l_x\big(R(\gamma_0)\big).
    \end{align*}
    The spatial side length necessarily fixes the temporal side length. We get
    \begin{align*}
        l_t\big(Q(\alpha)\big) &= (1-\alpha) l_x\big(Q(\alpha)\big)^p \\
        &=(1-\alpha) 2^{-(i+1)dp} l_x\big(R(\gamma_0)\big)^p \\
        &= \frac{1-\alpha}{1-\gamma_0} \cdot 2^{-(i+1)dp} l_t\big(R(\gamma_0)\big).
    \end{align*}
    Since $\gamma_0,\alpha \in [0,1/2]$, the above simplifies into
    \begin{align*}
        \frac{1}{2} \cdot 2^{-(i+1)dp}  l_t\big(R(\gamma_0)\big) \leq l_t\big(Q(\alpha)\big)\leq 2\cdot 2^{-(i+1)dp} l_t\big(R(\gamma_0)\big),
    \end{align*}
    finishing the proof.
\end{proof}

\subsection{Parabolic dyadic lattice for time-strips}
To study the effects evolving in time, we also introduce an extension of the dyadic lattice of a parabolic rectangle. We extend the dyadic lattice of $R(\gamma_0)$ of any order $i \in \N$ into the time-strip
$\text{pr}_x\big(R(\gamma_0)\big)\times \R \subseteq \Rn$ by
\begin{align*}
    \mathcal{D}_i^{\text{ext}}\big(R(\gamma_0)\big) = \bigcup_{j\in \Z}\mathcal{D}_i\big(R(\gamma_0) + (0,jl_t\big(R(\gamma_0)\big)\big).
\end{align*}
Similarly, we define the whole extended dyadic lattice as
\begin{align*}
    \mathcal{D}^{\text{ext}}\big(R(\gamma_0)\big) = \bigcup_{i=0}^\infty\mathcal{D}_i^{\text{ext}}\big(R(\gamma_0)\big).
\end{align*}
The extended dyadic lattice is thus closed under integer translations. In other words, if $P(\gamma) \in \mathcal{D}^{\textnormal{ext}}\big(R(\gamma_0)\big)$, then $P^\theta(\gamma) \in \mathcal{D}^{\textnormal{ext}}\big(R(\gamma_0)\big)$ for every $\theta \in \Z$.

On the other hand, the dyadic structure motivates the concept of dyadic parent rectangles. By the definition, for every $P(\gamma) \in \mathcal{D}^{\textnormal{ext}}\big(R(\gamma_0)\big)$ such that $l_x\big(P(\gamma)\big) < l_x\big(R(\gamma_0)\big)$, there exists the unique dyadic parent $Q(\alpha) \in \mathcal{D}^{\textnormal{ext}}\big(R(\gamma_0)\big)$ with $0\leq \alpha \leq 1/2$ such that $P(\gamma) \in \mathcal{D}_1\big(Q(\alpha)\big)$. To denote the dyadic parent, we use the convention
\begin{align*}
    \pi P(\gamma) = Q(\alpha).
\end{align*}
Moreover, if $P(\gamma) \in \mathcal{D}_m^{\textnormal{ext}}\big(R(\gamma_0)\big)$ for some $m \in \N$, we define recursively the higher order parents as
\begin{align*}
    \pi_0 P(\gamma) = P(\gamma) \quad \text{and} \quad \pi_{i+1} P(\gamma) = \pi \big(\pi_i P(\gamma)\big)
\end{align*}
for every $i = 0, 1,\dots,m-1$. Note that $\pi P(\gamma) = \pi_1P(\gamma)$. The parent operator is applied before any translation $\theta \in \R$, however, we will write
\begin{align*}
    \pi_i P^\theta(\gamma) = \pi_i P(\gamma) + \big(0,\theta l_t\big(\pi_i P(\gamma)\big)\big).
\end{align*}
for any $i \in \N$ such that the higher order parent is well-defined.

The concept of taking the dyadic parent and then applying a translation appears frequently in the paper. Thus, it makes sense to define the forward-in-time parent as  
\begin{align}
    \pi^+P(\gamma) = \pi P^{\theta_0}(\gamma), \label{eq: theta0}
\end{align}
where integer $\theta_0 \geq 2 $ is some fixed default translation or time-lag. Observe that if $P(\gamma) \in \mathcal{D}^{\textnormal{ext}}\big(R(\gamma_0)\big)$ with  $l_x\big(P(\gamma)\big) < l_x\big(R(\gamma_0)\big)$, then there does exist $\pi^+P(\gamma) \in \mathcal{D}^{\textnormal{ext}}\big(R(\gamma_0)\big)$. While there is some room how to choose this parameter $\theta_0$, it has to be chosen appropriately for our methods. We will fix the parameter in Section \ref{sec: stopping conditions}. The choice is based on various geometric aspects appearing later in the paper, see Lemma \ref{lem: parameters}. The higher order forward-in-time parents are also defined recursively. If $P(\gamma) \in \mathcal{D}_m^{\textnormal{ext}}\big(R(\gamma_0)\big)$ for some $m \in \N$, then
\begin{align*}
    \pi_0^+P(\gamma) = P(\gamma) \quad \text{and} \quad \pi_{i+1}^+P(\gamma) = \pi^+\big(\pi_{i}^+P(\gamma)\big)
\end{align*}
for every $i= 0, 1, \dots, m-1$. Similarly as for the standard parent operator, the forward-in-time parent operator is applied before translation, that is, we will write
\begin{align*}
    \pi_i^+ P^\theta(\gamma) = \pi_i^+ P(\gamma) + \big(0,\theta l_t\big(\pi_i P(\gamma)\big)\big) = \pi_i P^{\theta_0 + \theta}(\gamma)
\end{align*}
for any $\theta \in \R$.

\subsection{Properties of parabolic dyadic lattices}
We will mention the most important properties of $\mathcal{D}\big(R(\gamma_0)\big)$ and consequently $\mathcal{D}^{\textnormal{ext}}\big(R(\gamma_0)\big)$. We will not prove these properties as they either follow easily from previous results or are naturally expected from dyadic structures. These include:
\begin{itemize}
\item Covering, that is, for any $m \in \N$
\begin{align*}
    \bigcup_{P(\gamma) \in \mathcal{D}_m(R(\gamma_0))}P(\gamma) = R(\gamma_0).
\end{align*}

\item Similarity, that is, if $P(\gamma),Q(\alpha) \in \mathcal{D}_m\big(R(\gamma_0)\big)$ with $\gamma,\alpha \in [0,1/2]$ for some $m\in \N$, then they are the same up to translation. In particular, $\gamma = \alpha$.

\item Nestedness, that is, for every two rectangles in $\mathcal{D}\big(R(\gamma_0)\big)$ either one of them is contained in the other or they are disjoint.

\item Finite chain of ancestors, that is, for every $P(\gamma) \in \mathcal{D}\big(R(\gamma_0)\big)$ with $0\leq \gamma \leq 1/2$ there exists $m\in \N$ such that $P(\gamma) \in \mathcal{D}_m\big(R(\gamma_0)\big)$ and
\begin{align*}
    P(\gamma) \subset \pi P(\gamma) \subset \pi_2 P(\gamma) \subset \dots \subset \pi_m P(\gamma) = R(\gamma_0).
\end{align*}

\item Comparability to the parent, that is, if $P(\gamma) \in \mathcal{D}\big(R(\gamma_0)\big)$ with $0\leq \gamma \leq 1/2$ such that $P(\gamma) \subset R(\gamma_0)$, then
\begin{align*}
    \frac{1}{2}|\pi_i P(\gamma)| \leq 2^{d(n+p)i}|P(\gamma)| \leq 2 |\pi_i P(\gamma)|
\end{align*}
for every $i \in \N$, given that the parent rectangle is well-defined.

\item Approximation, that is, if $P(\gamma) \subseteq R(\gamma_0)$ with $0\leq \gamma \leq 1/2$, then there exists $Q(\alpha) \in \mathcal{D}\big(R(\gamma_0)\big)$ with $0\leq \alpha \leq 1/2$ such that $Q(\alpha) \subseteq P(\gamma)$ and
\begin{align*}
    |Q(\alpha)| \geq \Big(\frac{1}{4}\Big)^{d(n+p)} |P(\gamma)|.
\end{align*}
\end{itemize}

We also briefly discuss the properties unique to $\mathcal{D}^{\textnormal{ext}}\big(R(\gamma_0)\big)$. Namely, the forward-in-time operator induces chains of parabolic rectangles. These chains are an integral component of the main results. We introduce a short lemma for the forward-in-time parent operator, which is used several times in Section \ref{sec: stopping conditions}. This lemma will give an upper bound for the temporal length of the chains.
\begin{lemma} \label{lem: bounded distance}
    Let $R(\gamma_0) \subseteq \Rn$ be a parabolic rectangle with $0\leq \gamma_0 \leq 1/2$, and let $\theta_0 \geq 2$ be an integer. If $m\geq 1$ and $P(\gamma) \in \mathcal{D}_m^{\textnormal{ext}}\big(R(\gamma_0)\big)$ with $0\leq \gamma \leq 1/2$, then
    \begin{align*}
        \dist_t \big(\partial_{\textnormal{low}}P(\gamma), \partial_{\textnormal{low}}\pi_j^+P(\gamma)\big)  < 2\theta_0\frac{2^{dp}}{2^{dp}-1}l_t\big(\pi_{j}^+P(\gamma)\big).
    \end{align*}
    for every $j =1,2,\dots,m$.
\end{lemma}

\begin{proof}
     Fix integers $1\leq j \leq m$ and let $P(\gamma) \in \mathcal{D}_m^{\textnormal{ext}}\big(R(\gamma_0)\big)$ with $0\leq \gamma \leq 1/2$. Assuming $i = 0,1,\dots,j$, then by Corollary \ref{cor: dyadic scale} the side lengths of the forward-in-time parent rectangles are comparable as
     \begin{align*}
         l_t\big(\pi_i^+ P(\gamma)\big) \leq 2\cdot 2^{(i-j)dp}l_t\big(\pi_{j-i}^+\big(\pi_i^+ P(\gamma)\big)\big) =  2^{(i-j)dp +1}l_t\big(\pi_j^+P(\gamma)\big).
     \end{align*}
    On the other hand, the distance between any consecutive forward-in-time parent rectangles is bounded by
    \begin{align*}
        \dist_t\big(\partial_{\textnormal{low}}\pi_{i}^+P(\gamma),\partial_{\textnormal{low}}\pi_{i+1}^+P(\gamma)\big) &\leq \theta_0 l_t\big(\pi_{i+1}^+P(\gamma)\big) \\
        &\leq 2^{(1-j)dp +1}\theta_0\cdot 2^{idp}l_t\big(\pi_j^+ P(\gamma)\big),
    \end{align*}
    for any $i =0,1,\dots, j -1 $. We split the distance between $P(\gamma)$ and $\pi_j^+P(\gamma)$ into distances between $\pi_{i}^+P(\gamma)$ and $\pi_{i+1}^+P(\gamma)$ to obtain
    \begin{align*}
        \dist_t\big(\partial_{\textnormal{low}}P(\gamma),\partial_{\textnormal{low}}\pi_j^+P(\gamma)\big) &= \sum_{i=0}^{j-1} \dist_t\big(\partial_{\textnormal{low}}\pi_i^+P(\gamma),\partial_{\textnormal{low}}\pi_{i+1}^+P(\gamma)\big)  \\
        &\leq \sum_{i=0}^{j-1} 2^{(1-j)dp +1}\theta_0\cdot 2^{idp}l_t\big(\pi_j^+ P(\gamma)\big) \\
        &= 2^{(1-j)dp +1}\theta_0\frac{2^{jdp}-1}{2^{dp}-1} l_t\big(\pi_j^+ P(\gamma)\big)\\
        &<2\theta_0\frac{2^{dp}}{2^{dp}-1} l_t\big(\pi_j^+P(\gamma)\big),
    \end{align*}
    finishing the proof.
\end{proof}

\section{Parabolic weak porosity and Muckenhoupt classes} \label{sec: parabolic weak porosity}

In this section we introduce the concept of the parabolic weak porosity. We shall focus on the basic properties of parabolic weakly porous sets, while briefly tackling some of the main results. In particular, we show the first direction of our main result, that is, parabolic Muckenhoupt distance weights induce parabolic weakly porous sets. We first go through couple concepts that are used to define the parabolic weak porosity.

\subsection{Maximal dyadic subrectangles}
Given a nonempty closed set $E \subseteq \Rn$, we say that a set $A \subseteq \Rn$ is $E$-free if $A \cap E = \emptyset$. Since $E$ is closed, then $\dist_p\big((x,t),E\big) >0$ for any point in $(x,t) \in \Rn \setminus E$. Hence, there will always exist some $E$-free ball under the parabolic distance metric. This also implies that for any $0\leq \gamma \leq 1/2$ there exists an $E$-free parabolic rectangle $P(\gamma) \in \Rn$ such that $(x,t) \in P(\gamma)$.

In particular, if $(x,t) \in R(\gamma_0) \setminus E$ for some, possibly not $E$-free, parabolic rectangle $R(\gamma_0) \subseteq \Rn$ with $0\leq \gamma_0 \leq 1/2$,  then by the approximation properties of the dyadic lattice, $(x,t)$ is contained in some $E$-free dyadic $P(\gamma) \in \mathcal{D}_m\big(R(\gamma_0)\big)$ with $0 \leq \gamma \leq 1/2$ and $m \in \N$. Moreover, there exists a finite chain of ancestors $\pi_i P(\gamma)$ for $i = 0,1,\dots, m$, so there also exists the maximal $E$-free dyadic $P_{(x,t)}(\alpha) \in \mathcal{D}\big(R(\gamma_0)\big)$ with $0 \leq \alpha \leq 1/2$ such that $(x,t) \in P_{(x,t)}(\alpha)$. We denote the collection of the maximal $E$-free dyadic rectangles by
\begin{align}
    \mathcal{F}\big(R(\gamma_0)\big) =  \big\{P_{(x,t)}(\alpha) \; :\; (x,t) \in R(\gamma_0) \setminus E\big\}. \label{eq: pre F}
\end{align}

On the other hand, if we are given just a parabolic rectangle $R(\gamma_0)$, we can also consider its largest $E$-free dyadic subrectangle. 
\begin{definition}
    Suppose $E \subseteq \Rn$ is nonempty closed set, and let $R(\gamma_0) \subseteq \Rn$ be a parabolic rectangle with $0\leq \gamma_0 \leq 1/2$. Then, $\mathcal{M}\big(R(\gamma_0)\big)\in\mathcal{D}\big(R(\gamma_0)\big)$ is a largest $E$-free dyadic subrectangle, that is, $\mathcal{M}\big(R(\gamma_0)\big)\in\mathcal{F}\big(R(\gamma_0)\big)$ and it satisfies
    \begin{align*}
        l_x\big(P(\gamma)\big) \leq l_x\big(\mathcal{M}\big(R(\gamma_0)\big)\big) 
    \end{align*}
    for every $P(\gamma) \in \mathcal{F}\big(R(\gamma_0)\big)$ with $0\leq \gamma \leq 1/2$. If $\mathcal{F}\big(R(\gamma_0)\big) = \emptyset$, then $\mathcal{M}\big(R(\gamma_0)\big) = \emptyset$ for completeness.
\end{definition}

The side length of the maximal $E$-free dyadic subrectangles are in the elliptic case comparable to the essential supremum of the distance function. Hence, it is quite expected that similar results would hold with the parabolic distance metric. We have the following lemma for the pointwise maximal $E$-free rectangles and the maximal hole function $\mathcal{M}$.

\begin{proposition} \label{prop: maximal hole}
    Suppose $E \subseteq \Rn$ is a nonempty closed set and $R(\gamma_0) \subseteq \Rn$ is a parabolic rectangle with $0\leq \gamma_0 \leq 1/2$ such that $R(\gamma_0) \cap E \neq \emptyset$. Then, the following are true:
    \begin{enumerate}[label=(\roman*)]
    \item Every $P(\gamma) \in \mathcal{F}\big(R(\gamma_0)\big)$ with $0\leq \gamma \leq 1/2$ satisfies
    \begin{align*}
        \frac{1}{2} l_x\big(P(\gamma)\big) \leq \esssup_{(x,t)\in P(\gamma)}\dist_p\big((x,t),E \big) \leq 2^d l_x\big(P(\gamma)\big).
    \end{align*} \label{item: maximal hole 1}
    
    \item The maximal $E$-free subrectangle $\mathcal{M}\big(R(\gamma_0)\big) \in \mathcal{F}\big(R(\gamma_0)\big)$ satisfies
    \begin{align*}
        \frac{1}{2} l_x\big(\mathcal{M}\big(R(\gamma_0)\big)\big) \leq \esssup_{(x,t)\in R(\gamma_0)}\dist_p\big((x,t),E \big) \leq 2^d l_x\big(\mathcal{M}\big(R(\gamma_0)\big)\big).
    \end{align*} \label{item: maximal hole 2}
    \end{enumerate}
    Moreover, the leftmost inequalities apply also when $R(\gamma_0) \cap E = \emptyset$.
\end{proposition}

\begin{proof}
    \ref{item: maximal hole 1} Let $P(\gamma) \in  \mathcal{F} = \mathcal{F}\big(R(\gamma_0)\big)$ with $0\leq \gamma \leq 1/2$. Denote $(x_0,t_0) \in P(\gamma)$ the center point of $ P(\gamma)$. It follows that the largest distance between any $(x,t) \in P(\gamma)$ and $E$ is bounded below by
    \begin{align*}
        \esssup_{(x,t)\in P(\gamma)}  \dist_p\big((x,t),E\big) &\geq \dist_p\big((x_0,t_0),E\big) \geq \dist_p\big((x_0,t_0),\partial P(\gamma) \big) \\
        &\geq \max \bigg\{ \frac{1}{2}l_x\big(P(\gamma)\big), \Big(\frac{1}{2}l_t\big(P(\gamma)\big) \Big)^\frac{1}{p} \bigg\} \\
        &= \max \bigg\{ \frac{1}{2}, \Big(\frac{1-\gamma}{2} \Big)^\frac{1}{p} \bigg\}l_x\big(P(\gamma)\big) \\
        &\geq \frac{1}{2} l_x\big(P(\gamma)\big),
    \end{align*}
    proving the first inequality.
    
    For the second inequality, observe that $l_x\big(P(\gamma)\big) < l_x\big(R(\gamma_0)\big)$ since $R(\gamma_0) \cap E \neq \emptyset$. Therefore, by maximality of $P(\gamma)$, we have $\pi P(\gamma) \cap E \neq \emptyset$. It follows that the parabolic distance between any $(x,t) \in P(\gamma)$ and $E$ is then bounded by the parabolic diameter of the parent rectangle $\pi P(\gamma)$. We get
    \begin{align*}
        \dist_p\big((x,t),E\big) &\leq \text{diam}_p\big(\pi P(\gamma)\big) = l_x\big(\pi P(\gamma)\big) = 2^d l_x\big(P(\gamma)\big).
    \end{align*}
    Taking the essential supremum on the left hand side yields
    \begin{align*}
        \esssup_{(x,t)\in P(\gamma)}  \dist_p\big((x,t),E\big) \leq 2^d l_x\big(P(\gamma)\big),
    \end{align*}
    proving the second inequality.

    \ref{item: maximal hole 2} Let us take the first inequality of item \ref{item: maximal hole 1}, that is,
    \begin{align*}
        \frac{1}{2} l_x\big(P(\gamma)\big) \leq \esssup_{(x,t)\in P(\gamma)}\dist_p\big((x,t),E \big).
    \end{align*}
    The essential supremum can be extended to the whole $R(\gamma_0)$ as
    \begin{align*}
        \frac{1}{2} l_x\big(P(\gamma)\big) \leq \esssup_{(x,t)\in R(\gamma_0)}  \dist_p\big((x,t),E\big),
    \end{align*}
    On the other hand, since the right hand side does not depend on $P(\gamma)$,we take the maximum over every $P(\gamma) \in \mathcal{F}$ to obtain
    \begin{align*}
        \frac{1}{2} l_x\big(\mathcal{M}\big(R(\gamma_0)\big)\big) = \max_{P(\gamma) \in \mathcal{F}} \frac{1}{2} l_x\big(P(\gamma)\big) \leq \esssup_{(x,t)\in R(\gamma_0)}  \dist_p\big((x,t),E\big),
    \end{align*}
    proving the first inequality of \ref{item: maximal hole 2}.

     Let us then take the second inequality of item \ref{item: maximal hole 1}, that is,
    \begin{align*}
        \esssup_{(x,t)\in P(\gamma)}\dist_p\big((x,t),E \big) \leq 2^d l_x\big(P(\gamma)\big).
    \end{align*}
    We can clearly replace the side length of $P(\gamma)$ with the maximal $E$-free subrectangle to get
    \begin{align*}
        \esssup_{(x,t)\in P(\gamma)}  \dist_p\big((x,t),E\big) \leq  2^d l_x\big(\mathcal{M}\big(R(\gamma_0)\big)\big).
    \end{align*}
    Since the right hand side does not depend on $P(\gamma)$, we can expand the essential supremum over the union of every $P(\gamma) \in \mathcal{F}$. However, we observe
    \begin{align*}
        (x,t) \in \bigcup_{P(\gamma)\in \mathcal{F}}P(\gamma) = R(\gamma_0) \setminus E.
    \end{align*}
    Note that if $(x,t) \in R(\gamma_0) \cap E$, then $ \dist_p\big((x,t),E\big) = 0$. We get
    \begin{align*}
        \esssup_{(x,t)\in R(\gamma_0)} \dist_p\big((x,t),E\big) \leq  2^d l_x\big(\mathcal{M}\big(R(\gamma_0)\big)\big),
    \end{align*}
    finishing the proof.
\end{proof}

\subsection{Parabolic weak porosity}

Following the steps of \cite{ALMV24}, the parabolic weak porosity is defined similarly. The key differences to the definition in the elliptic case are the parabolic geometry via the usage of parabolic dyadic lattice and the time-lag in the maximal hole function.
\begin{definition} \label{def: pwp general}
    A nonempty closed set $E \subseteq \Rn$ is $(c,\delta,\theta)$-weakly porous for $c,\delta \in (0,1)$ and $\theta \in \R$ if for every parabolic rectangle $R(\gamma_0) \subset\mathbb{R}^{n+1}$ with $0 \leq \gamma_0 \leq 1/2$ there exist pairwise disjoint rectangles $P_k(\gamma_k)\in \mathcal{D}\big(R(\gamma_0)\big)$ with $0\leq \gamma_k \leq 1/2$ for $k=1,\dots,N$, such that they are $E$-free,
    \begin{align*}
        \lvert P_k(\gamma_k) \rvert \geq \delta \lvert \mathcal{M}\big(R^\theta(\gamma_0)\big) \rvert
    \end{align*}
     and satisfy
    \[
    \sum_{k=1}^N \lvert P_k(\gamma_k) \rvert \geq c \lvert R(\gamma_0) \vert.
    \]
    In this context, we say that a subrectangle $P(\gamma) \in \mathcal{D}\big(R(\gamma_0)\big)$ is $\delta$-admissible if it is $E$-free and $\lvert P(\gamma) \rvert \geq \delta \lvert \mathcal{M}\big(R^\theta(\gamma_0)\big) \rvert$.
\end{definition}

The time-lag appears in the definition via the parameter $\theta \in \R$ as we evaluate the maximal $E$-free hole in the translated parent $R^\theta(\gamma_0)$ instead of $R(\gamma_0)$. When $\theta >1$, that is, there is nonzero time-lag between the $R(\gamma_0)$ and $R^\theta(\gamma_0)$, we have a link to the theory of the parabolic Muckenhoupt weights. Then, we are able to show various forward-in-time doubling results, which are later used to prove the parabolic Muckenhoupt distance weight characterization.

\begin{remark}
    It makes sense to define the parabolic weak porosity for general $\theta \in \R$, since some of our results apply also then. Observe that the case $\theta = 0$ would revert back to the elliptic theory, while $\theta < -1$ would be symmetric with the case $\theta >1$. The limiting case $\theta = 1$ is also interesting.
\end{remark}

\begin{remark}
    Since the dyadic lattice of any $R(\gamma_0) \subseteq \Rn$ with $0\leq \gamma_0\leq 1/2$ consists of rectangles $P(\gamma) \in \mathcal{D}\big(R(\gamma_0)\big)$ with $0\leq \gamma \leq 1/2$, it 
    motivates to define the parabolic weak porosity to every parabolic rectangle of such type, instead of for fixed $\gamma_0$. This simplifies the proofs by a large degree.
\end{remark}

\begin{remark}
    One can show that our definition of $(c,\delta,\theta)$-weak porosity is equivalent with $\gamma$-FIT weak porosity in \cite{KongYangYuan26}, whenever $\theta>1$ and $0<\gamma<1$.
\end{remark}

\subsection{Basic properties of parabolic weak porosity}
While the Definition \ref{def: pwp general} leaves some flexibility how to choose the pairwise disjoint $\delta$-admissible rectangles, in practice it is convenient to choose them from some standard collection. Namely, for any parabolic rectangle $R(\gamma_0) \subseteq \Rn$ with $0\leq \gamma_0 \leq 1/2$ it is quite natural to choose the $\delta$-admissible rectangles from the collection of the maximal $E$-free rectangles $\mathcal{F}\big(R(\gamma_0))$, see (\ref{eq: pre F}). We define the collection of the maximal $\delta$-admissible rectangles as 
\begin{align}
    \mathcal{F}_\delta^\theta\big(R(\gamma_0)\big) = \big\{ P(\gamma) \in \mathcal{F}\big(R(\gamma_0)\big) \;:\; |P(\gamma)| \geq \delta \mathcal{M}\big(R^\theta(\gamma_0)\big) \big\}. \label{eq: Fdelta}
\end{align}
Observe that $ \mathcal{F}_\delta^\theta\big(R(\gamma_0)\big) \subseteq \mathcal{F}\big(R(\gamma_0)\big)$ are collections of pairwise disjoint rectangles with possibly different truncation parameters $0\leq \gamma\leq1/2$. The rectangles have to be pairwise disjoint since by the nestedness of the dyadic lattice, if any different $P(\gamma),Q(\alpha) \in \mathcal{F}\big(R(\gamma_0)\big)$ with $\gamma,\alpha \in [0,1/2]$ intersected, then one of them would not be maximal.

We have the following result which allows us to use the $\mathcal{F}_\delta^\theta\big(R(\gamma_0)\big)$ collections to characterize parabolic weak porosity.
\begin{proposition}\label{prop: Fdelta}
    A nonempty closed set $E \subseteq \Rn$ is $(c,\delta,\theta)$-weakly porous for some $c,\delta\in(0,1)$ and $\theta \in \R$  if and only if every parabolic rectangle $R(\gamma_0) \subseteq \Rn$ with $0\leq \gamma_0 \leq 1/2$ satisfies
    \begin{align*}
        \sum_{P(\gamma)\in \mathcal{F}_\delta^\theta(R(\gamma_0))}|P(\gamma)| \geq c |R(\gamma_0)|.
    \end{align*}
\end{proposition}

\begin{proof}
    Let $R(\gamma_0) \subseteq \Rn$ be a parabolic rectangle. For the first direction, consider the $\delta$-admissible rectangles $P_k(\gamma_k) \in \mathcal{D}\big(R(\gamma_0)\big)$ with $0\leq \gamma_k \leq 1/2$ for $k =1, \dots, N$. It follows that for every $k = 1,\dots, N$ there exists $P(\gamma) \in \mathcal{F}_\delta^\theta = \mathcal{F}_\delta^\theta\big(R(\gamma_0)\big)$ with $P_k(\gamma_k) \subseteq P(\gamma)$ and $0\leq \gamma \leq 1/2$. A short calculation verifies
    \begin{align*}
        \sum_{P(\gamma)\in \mathcal{F}_\delta^\theta} |P(\alpha)| &\geq \sum_{P(\alpha)\in \mathcal{F}_\delta^\theta} \sum_{k=1}^N \mathbbm{1}_{P_k(\gamma_k) \subseteq P(\alpha)}|P_k(\gamma_k)| \nonumber\\
        &= \sum_{k=1}^N \sum_{P(\alpha)\in \mathcal{F}_\delta^\theta}\mathbbm{1}_{P_k(\gamma_k) \subseteq P(\alpha)}|P_k(\gamma_k)|\\
        &\geq  \sum_{k =1}^N |P_k(\gamma_k)| \geq c |R(\gamma_0)|.
    \end{align*}

    The other direction follows clearly as $ \mathcal{F}_\delta^\theta$ is a finite pairwise disjoint collection of $\delta$-admissible rectangles. The explanation for disjointedness is right after (\ref{eq: Fdelta}).
\end{proof}

Similar to the elliptic theory, the Lebesgue measure of a parabolic weakly porous set is zero.

\begin{proposition} \label{prop: 0 measure}
    Suppose $E \subseteq \Rn$ be $(c,\delta, \theta)$-weakly porous for some $c,\delta \in (0,1)$ and $\theta \in \R$. Then, $|E| = 0$.
\end{proposition}
\begin{proof}
    Let $(x,t) \in \Rn$. Define a parabolic rectangle $R_i(\gamma_0) = R(x,t,2^{-i},\gamma_0)$ with $0\leq \gamma_0 \leq 1/2$ for every $i\in \N$. By the parabolic weak porosity there exist pairwise disjoint $\delta$-admissible $P_j(\gamma_j) \in \mathcal{D}\big(R(\gamma_0)\big)$ with $0\leq \gamma_j \leq 1/2$ for $j =1,\dots, N_i$ such that 
    \begin{align*}
        \sum_{j=1}^{N_i} |P_j(\gamma_j)| \geq c |R_i(\gamma_0)|,
    \end{align*}
    We shall denote the union of the $\delta$-admissible rectangles by
    \begin{align*}
        F_i = \bigcup_{j=1}^{N_i} P_j(\gamma_j) \subseteq R_i(\gamma_0).
    \end{align*}
    Now, by letting $i \rightarrow \infty$, the sets $F_i$ converges regularly to $(x,t)$ in the parabolic geometry, see \cite[Definition 2.2, Lemma 2.3]{KinnunenMyyryYang2024}, a version of the Lebesgue differentiation theorem. The lemma implies that for $\mathbbm{1}_E \in L_{\textnormal{loc}}^1(\Rn)$ we have
    \begin{align*}
    \lim_{i\rightarrow \infty} \dashint_{F_i} \big|\mathbbm{1}_E (y,s) - \mathbbm{1}_E (x,t)\big| \dy \ds = 0
    \end{align*}
    for almost every $(x,t) \in \Rn$, and consequently
    \begin{align*}
    \mathbbm{1}_E (x,t) = \lim_{i\rightarrow 0} \dashint_{F_i} \mathbbm{1}_E (y,s) \dy \ds = \lim_{i\rightarrow \infty} \dashint_{F_i} 0 \dy \ds = 0.
    \end{align*}
    If $(x,t) \in E$, then clearly $\mathbbm{1}_E(x,t) = 1 \neq 0$, so necessarily $|E| = 0$.
\end{proof}

\subsection{Parabolic Muckenhoupt distance weights} \label{subsec: parabolic A1}

The natural counterpart of the theory of the parabolic weak porosity are the parabolic Muckenhoupt classes introduced in \cite{KinnunenMyyry2024}. We briefly include the notation of these works to draw the connection to our work. The established way of writing is
\begin{align*}
    R^-(\gamma) = Q(x,L) \times (t - L^p, t - \gamma L^p) \quad \text{and} \quad R^+(\gamma) = Q(x,L) \times (t + \gamma L^p, t + L^p).
\end{align*}
This notation combines the truncation and translation, however, in our work we prefer to separate these features. Observe that we can rewrite these sets as
\begin{align}
    R^-(\gamma) = R^0(\gamma) = R(\gamma) \quad \text{and} \quad  R^+(\gamma) = R^{\theta}(\gamma), \label{eq: R plusminus}
\end{align}
when $\theta = (1+\gamma)(1-\gamma)^{-1}$ and $0\leq \gamma <1$. 

The parabolic Muckenhoupt classes introduce time dependency by comparing integrals over $R^-(\gamma)$ and $R^+(\gamma)$. Crucially, there is some time-lag between the upper and lower part. Using $R^\pm(\gamma)$ notation, this corresponds to the case $\gamma>0$. A weight is a nonnegative locally integrable function, and the definition of the parabolic $A_1$ class is as follows.
\begin{definition}\label{def: parabolic A1}
Let $0< \gamma<1$. A weight $w \in L^1_{\text{loc}}(\Rn)$ belongs to the parabolic Muckenhoupt class $A^+_1(\gamma)$ if
there exists a constant $C>0$ such that
\[
\dashint_{R^-(\gamma)} w \dx \dt \leq C \essinf_{(x,t)\in R^+(\gamma)} w(x,t)
\]
for every pair of $R^-(\gamma),R^+(\gamma) \subseteq \mathbb R^{n+1}$.
\end{definition}

One of the known results of the parabolic Muckenhoupt classes is the invariance of $0<\gamma<1$ in the definition of $A_1^+(\gamma)$. Moreover, we can also add any translation between $R^-(\gamma)$ and $R^+(\gamma)$ as long as the time-lag is positive. Therefore, we may reformulate the time-lag invariance to the following lemma, which will be our bridge to $A_1^+(\gamma)$, see the proof of \cite[Theorem 3.1]{KinnunenMyyry2024}.
\begin{lemma}
    Let $\theta >1$ and $0< \gamma<1$. Then, a weight $w \in A^+_1(\gamma)$ if and only if there exists a constant $C >0$ such that
    \begin{align*}
        \dashint_{R(\gamma_0)} w \dx \dt \leq C \essinf_{(x,t)\in R^\theta(\gamma_0)} w(x,t).
    \end{align*}
    for every parabolic rectangle $R(\gamma_0) \subseteq \Rn$ with $0\leq \gamma_0 \leq 1/2$. \label{lem: A1 starting point}
\end{lemma}

\subsection{Parabolic weak porosity via Muckenhoupt distance weights}
We shall prove the first direction of our main theorem, that is, if a distance function $\dist_p(\cdot, E)^{-\alpha(n+p)} \in A_1^+(\gamma)$ for some set $E \subseteq \Rn$, then the set $E$ is parabolic weak porous. While this direction was proven in \cite{KongYangYuan26} for $\gamma$-FIT parabolic weak porosity, we have still included our proof. As a matter of fact, we show a slightly stronger statement, which will motivate our planned approach for the reverse direction.

\begin{theorem} \label{theo: A1 => wp}
Suppose $E\subseteq \mathbb{R}^{n+1}$ is a nonempty closed set, $\theta>1$, $\alpha>0$ and $0< \gamma<1$. If $\dist_p(\cdot, E)^{-\alpha(n+p)} \in A_1^+(\gamma) $, then $E$ is $(c_0,\delta_0,\theta)$-weakly porous for some $c_0,\delta_0 \in (0,1)$. Moreover, there exists a constant $C = C(n,p,d,\gamma,\theta,\alpha,E) >0$ such that for any $0<\delta\leq \delta_0$ the set $E$ is $(c,\delta,\theta)$-weakly porous for some $0<c<1$ satisfying
    \[1-c \leq C \delta^\alpha.\]
\end{theorem}

\begin{proof}
    Suppose $\dist_p(\cdot, E)^{-\alpha(n+p)} \in A_1^+(\gamma)$ for some closed nonempty set $E \subseteq \Rn$ with $\alpha >0$ and $0< \gamma <1$. Note that $\lvert E \rvert =0$ since $w$ is locally integrable and $w=\infty$ in $E$. Fix a parabolic rectangle $R(\gamma_0)\subseteq\mathbb{R}^{n+1}$ with $0 \leq \gamma_0 \leq 1/2$. Note that if $R(\gamma_0) \cap E = \emptyset$, then
    \begin{align*}
        |R(\gamma_0)| \geq \delta |R^\theta(\gamma_0)| \geq \delta |\mathcal{M}\big(R^\theta(\gamma_0)\big)|
    \end{align*}
    and $|R(\gamma_0)| \geq c |R(\gamma_0)|$ for every pair of $c,\delta \in(0,1)$. Therefore, we may assume $R(\gamma_0) \cap E \neq \emptyset$. 
    
    Consider the collection $\mathcal{F}_\delta^\theta = \mathcal{F}_\delta^\theta\big(R(\gamma_0)\big) \subseteq \mathcal{F}\big(R(\gamma_0)\big)$, where $\theta >1$ and $0<\delta<1$ is yet to be determined. Let us define
    \begin{align}
        G = \big(R(\gamma_0) \setminus E\big)\setminus \bigcup_{P(\gamma_1) \in \mathcal{F}_\delta^\theta}P(\gamma_1). \label{eq: A1 => pwp G}
    \end{align}
    Since $E$ is closed, for any $(x,t) \in G \subseteq R(\gamma_0) \setminus E$ there exists the maximal $E$-free dyadic subrectangle $Q_{(x,t)}(\gamma_2) \in \mathcal{F}\big(R(\gamma_0)\big)$ with $0\leq \gamma_2 \leq 1/2$ such that $(x,t) \in Q_{(x,t)}(\gamma_2)$. Denote $Q(\gamma_2) = Q_{(x,t)}(\gamma_2)$ for simplicity. However, since $(x,t) \notin \bigcup_{P(\gamma_1) \in \mathcal{F}_\delta^\theta} P(\gamma_1)$, the measure of $Q(\gamma_2)$ has the upper bound
    \[
    \lvert Q(\gamma_1) \rvert < \delta \big\lvert \mathcal{M}\big(R^\theta(\gamma_0) \big) \big\rvert.
    \]
    The above is equivalently expressed using the side length as
    \begin{align*}
        (1-\gamma_2)\big(l_x\big(Q(\gamma_2)\big)\big)^{n+p} <  \delta (1-\beta_1) l_x\big(\mathcal{M}\big(R^\theta(\gamma_0) \big)\big)^{n+p},
    \end{align*}
    where  $0\leq \beta_1 \leq 1/2$ is the truncation parameter of $\mathcal{M}\big(R^\theta(\gamma_0) \big)$. By solving $l_x\big(Q(\gamma_2)\big)$, we get
    \begin{align*}
        l_x\big(Q(\gamma_2)\big) &<  \Big(\delta \frac{1-\beta_1}{1-\gamma_2}\Big)^\frac{1}{n+p} l_x\big(\mathcal{M}\big(R^\theta(\gamma_0) \big)\big)\\
        &\leq (2\delta)^\frac{1}{n+p} l_x\big(\mathcal{M}\big(R^\theta(\gamma_0)\big)\big).
    \end{align*}
    By Proposition \ref{prop: maximal hole}\ref{item: maximal hole 1} the distance between $(x,t)$ and $E$ is now bounded by
    \begin{align*}
        \dist_p\big((x,t),E\big)  &\leq \esssup_{(y,s)\in Q(\gamma_2)}\dist_p\big((y,s),E \big) \leq 2^d l_x\big(Q(\gamma_2)\big) \\
        &\leq  C_1 \delta^{\frac{1}{n+p}} l_x\big(\mathcal{M}\big(R^\theta(\gamma_0)\big)\big),
    \end{align*}
    where $C_1 = 2^{d+\frac{1}{n+p}}$. Therefore, by raising both sides to the power of $-\alpha(n+p)$, we obtain
    \[
    l_x\big(\mathcal{M}\big(R^\theta(\gamma_0)\big)\big)^{-\alpha (n+p)} \leq C_1^{\alpha(n+p)} \delta^\alpha \dist_p\big((x,t),E\big)^{-\alpha(n+p)}
    \]
    for every $(x,t)\in G$.
    
    Integrating both sides of the previous inequality over the set $G$, we get
    \begin{align}
        \begin{split}
            l_x\big(\mathcal{M}\big(R^\theta(\gamma_0)\big)\big)^{-\alpha(n+p)} \lvert G \rvert
            &\leq C_1^{\alpha(n+p)} \delta^\alpha \int_G \dist_p\big((x,t),E\big)^{-\alpha(n+p)}\dx\dt \\
            &\leq C_1^{\alpha(n+p)} \delta^\alpha \int_{R(\gamma_0)} \dist_p\big((x,t),E\big)^{-\alpha(n+p)}\dx\dt.
    \end{split}\label{eq: A1 => pwp main inequality}
    \end{align}
    Since $\dist_p\big(\cdot,E\big)^{-\alpha(n+p)} \in A_1^+(\gamma)$, we apply Lemma \ref{lem: A1 starting point}, and then Proposition \ref{prop: maximal hole}\ref{item: maximal hole 2} to obtain
    \begin{align*}
        \int_{R(\gamma_0)} \dist_p\big((x,t),E\big)^{-\alpha(n+p)}\dx\dt &\leq C_2 |R^\theta(\gamma_0)| \essinf_{(x,t)\in R^\theta(\gamma_0)} \dist_p\big((x,t),E\big)^{-\alpha(n+p)}\\
        & = C_2 |R^\theta(\gamma_0)| \bigg(\esssup_{(x,t)\in R^\theta(\gamma_0)} \dist_p\big((x,t),E\big) \bigg)^{-\alpha(n+p)} \\
        &\leq 2^{\alpha(n+p)} C_2 |R^\theta(\gamma_0)| l_x\big(\mathcal{M}\big(R^\theta(\gamma_0)\big)\big)^{-\alpha(n+p)},
    \end{align*}
    where $C_2>0$ is a constant from Lemma \ref{lem: A1 starting point}. Substituting the above to (\ref{eq: A1 => pwp main inequality}), we obtain
    \begin{align*}
        |G| \leq 2^{\alpha(n+p)}C_2 C_1^{\alpha(n+p)} \delta^\alpha |R(\gamma_0)| \leq C \delta^\alpha |R(\gamma_0)|,
    \end{align*}
    where $ C = \max \big\{ 2^{\alpha(n+p)}C_2 C_1^{\alpha(n+p)},1 \big\}$.
    
    To finish the proof, recall the definition of the set $G$, see (\ref{eq: A1 => pwp G}). By substituting the above, we have
    \begin{align*}
    |R(\gamma_0)| - \sum_{P(\gamma_1)\in \mathcal{F}_\delta^\theta} |P(\gamma_1)| &= \Big|R(\gamma_0)\setminus \bigcup_{P(\gamma_1)\in \mathcal{F}_\delta^\theta} P(\gamma)\Big|\\
    &=\big\lvert G \cup E \big\rvert = \lvert G \rvert \leq C  \delta^\alpha \lvert R(\gamma_0) \rvert.
    \end{align*}
    Rearranging the terms yields
    \begin{align}
    \sum_{P(\gamma_1)\in \mathcal{F}_\delta^\theta} |P(\gamma_1)| \geq (1-C\delta^\alpha ) |R(\gamma_0)|. \label{eq: A1 => pwp end}
    \end{align}
    We choose $0<\delta_0 <1$ small enough, for example $\delta_0= 2^{-1}C^{-\frac{1}{\alpha}} \in (0,1)$, and set
    \begin{align*}
        c_0 = 1 - C \delta_0^\alpha =  1- 2^{-\alpha} \in (0,1),
    \end{align*}
    verifying  $0< c_0 <1$. Furthermore, for any $0<\delta \leq \delta_0$, we set 
    \begin{align*}
        c = 1 - C \delta^\alpha \in [c_0,1).
    \end{align*}
    Substituting the above to (\ref{eq: A1 => pwp end}), we get
    \[
    \sum_{P(\gamma_1)\in \mathcal{F}_\delta^\theta} |P(\gamma_1)| \geq (1- C \delta^\alpha ) |R(\gamma_0)| = c |R(\gamma_0)|,
    \]
    which by Proposition \ref{prop: Fdelta} shows that $E$ is $(c,\delta,\theta)$-weakly porous.
\end{proof}

\section{The $\alpha$-improvement of parabolic weakly porous set} \label{sec: alpha improvement}

In this section, we show that a certain relationship between the weak porosity constants $c$ and $\delta$ is a necessary and sufficient condition for a set $E\subseteq \Rn$ to induce parabolic Muckenhoupt distance weights. 

\subsection{Weak porosity and $\alpha$-improvement}
Theorem \ref{theo: A1 => wp} motivates a necessary condition for the link between the parabolic weak porosity and the $A_1^+(\gamma)$. Namely, the inequality (\ref{eq: A1 => pwp end}) implies that the weak porosity constant $0<c<1$ can be taken to arbitrarily close to one by letting $0<\delta<1$ approach zero, in a sense improving the weak porosity. The exact improvement relationship is closely related to the exponent $\alpha>0$ of $\dist_p\big(\cdot,E\big)^{-\alpha(n+p)} \in A_1^+(\gamma)$, and is stated in Theorem \ref{theo: A1 => wp}. This relationship between the parameters $c$ and $\delta$ has been studied in \cite{Vargas2026} for the standard $A_q$ weights for $q>1$. The exponent $\alpha$ plays similar significant role in their works as well.

Following the statement of Theorem \ref{theo: A1 => wp}, we formulate the following definition.
\begin{definition} \label{def: alpha improvement}
    Suppose $E\subseteq \Rn$ is $(c_0,\delta_0,\theta)$-weakly porous for $c_0,\delta_0 \in(0,1)$ and $\theta \in \R$. We say, $E$ is $\alpha$-improving if there exists constants $\alpha,K >0$ such that for any $0<\delta\leq \delta_0$ the set $E$ is $(c,\delta,\theta)$-weakly porous for some $0<c<1$ satisfying
    \[1-c \leq K \delta^\alpha.\]
\end{definition}

Instead of using the definition above, it is sometimes easier to work with a sequential definition of the $\alpha$-improvement. Now, the cumbersome requirement of finding for every $0<\delta\leq \delta_0$ some $0<c<1$ can be replaced with countable sequences of certain type.
\begin{proposition} \label{prop: alpha improvement eqv}
    Let $\alpha >0$ and suppose $E\subseteq \Rn$ is $(c_0,\delta_0,\theta)$-weakly porous for  $c_0,\delta_0 \in(0,1)$ and $\theta \in \R$. Then, $E$ is $\alpha$-improving if and only if there exists constants $C,\eta>0$ and sequences $(c_i)_{i\in \N}$ and $(\delta_i)_{i\in \N}$ with $c_i,\delta_i \in (0,1)$ satisfying 
    \begin{align*}
        1-c_i \leq C\delta_i^\alpha, \quad \frac{\delta_{i+1}}{\delta_i} \geq \eta, \quad \text{and} \quad \lim_{i\rightarrow \infty}\delta_i\rightarrow 0,
    \end{align*}
    and the set $E$ is $(c_i,\delta_i,\theta)$-weakly porous for every $i \in \N$.
\end{proposition}

\begin{proof}
    The first direction, assuming $E$ is $\alpha$-improving, follows trivially by letting $\eta = 2^{-1}$ and $\delta_i = 2^{-i}\delta_0$ for every $i \in \N$.
    
    For the converse direction, fix $0<\delta_0<1$ from the sequence $(\delta_i)_{i\in\N}$. Let $0<\delta \leq \delta_0$ and choose any parabolic rectangle $R(\gamma_0) \subseteq \Rn$ with $0\leq \gamma_0 \leq 1/2$. Since $\delta_i \rightarrow 0$ as $i \rightarrow \infty$, there exists some $j \in \N$ such that $\delta_{j+1} \leq \delta\leq \delta_j$. Consequently, any $\delta_j$-admissible rectangle $P(\gamma) \in \mathcal{D}\big(R(\gamma_0)\big)$ with $0\leq \gamma \leq 1/2$ is also $\delta$-admissible. This is true because $P(\gamma)$ is $E$-free and
    \begin{align*}
         |P(\gamma)| \geq \delta_j|\mathcal{M}\big(R^\theta(\gamma_0)\big)| \geq \delta|\mathcal{M}\big(R^\theta(\gamma_0)\big)|.
    \end{align*}
    In other words, $E$ is $(c_j,\delta,\theta)$-weakly porous. Applying the two other conditions shows that
    \begin{align*}
        C\delta^\alpha \geq C \delta_{j+1}^\alpha \geq C \eta^\alpha \delta_j^\alpha \geq \eta^\alpha (1-c_j).
    \end{align*}
    We can substitute $K = C \eta^{-\alpha}$ to finish the proof.
\end{proof}

\subsection{Muckenhoupt characterization via weakly porous sets}
Definition \ref{def: alpha improvement} states a profound feature of weakly porous sets in general. To show that the $\alpha$-improvement of parabolic weakly porous sets characterizes $A_1^+(\gamma)$ distance weights, we need two lemmas. On the other hand, proving that any parabolic weakly porous set is $\alpha$-improving is difficult, and it is done in Section \ref{sec: full charac}.

\begin{lemma} \label{lem: annular estimate}
    Suppose $E\subseteq \Rn$ is a nonempty set and $R(\gamma) \subseteq \Rn$ is an $E$-free parabolic rectangle with $0\leq \gamma \leq1/2$ such that $R(\gamma)$ is $E$-free. Let $0<\alpha< (n+p)^{-1}$. Then, 
    \[\int_{R(\gamma)} \dist_p\big((x,t),E\big)^{-\alpha(n+p)}\dx \dt \leq C |R(\gamma)|^{1-\alpha}\]
    for $C = C(n,p, \alpha) >0$.
\end{lemma}

\begin{proof}
    Let $E \subseteq \Rn$ be a nonempty set and choose $R(\gamma)\subseteq \Rn$ with $0\leq \gamma \leq 1/2$ and a side length $L>0$. Moreover, we assume that $R(\gamma)$ is an $E$-free parabolic rectangle. Since the parabolic distance $\dist_p(\cdot,\cdot)$ is translation invariant, we may also assume that the center point of $R(\gamma)$ is at the origin. 
    
    We choose a sequence $\big(R_i(\gamma)\big)_{i\in \N}$ of parabolic rectangles $R_i(\gamma) \subseteq R(\gamma)$ such that the center point of $R_i(\gamma)$ is also at the origin and $l_x\big(R_i(\gamma)\big) = (1-2^{-i})L$ for each $i = 1,2,\dots $, and set $R_0(\gamma) = \emptyset$. It follows that $R_i(\gamma) \subseteq R_{i+1}(\gamma)$. Since
    \begin{align*}
        |R_i(\gamma)|=(1-2^{-i})^{n+p}|R(\gamma)|,
    \end{align*}
    we conclude that the union of the sequence approximates $R(\gamma)$, that is,
    \begin{align*}
        |R(\gamma)\setminus \bigcup_{i=0}^\infty R_i(\gamma)| = 0.
    \end{align*}
    Furthermore, the distance from $R_i(\gamma)$ to $E$ can be estimated with $|R(\gamma)|$. Let us take any $(x,t) \in R_i(\gamma)$. Since $R(\gamma)$ is $E$-free, the parabolic distance between $(x,t)$ and $E$ satisfies
    \begin{align}
        \begin{split}
            \dist_p\big((x,t),E\big) &\geq \dist_p\big((x,t),\partial R(\gamma)\big) \geq \dist_p\big(\partial R_i(\gamma), \partial R(\gamma)\big) \\
            &= \max \bigg\{ \frac{1}{2}\cdot 2^{-i}L, \Big(\frac{1-\gamma}{2} \cdot\big(1- (1-2^{-i})^p\big)L^p\Big)^\frac{1}{p} \bigg\} \\
            &\geq 2^{-i}L \max \bigg\{ \frac{1}{2}, \Big(\frac{1-\gamma}{2}\Big)^\frac{1}{p}\bigg\}  \geq \frac{1}{2}\cdot 2^{-i} L \\
            &= \frac{1}{2}\cdot2^{-i}\cdot (1-\gamma)^{-\frac{1}{n+p}} |R(\gamma)|^\frac{1}{n+p} \\
            &\geq \frac{1}{2}\cdot 2^{-i} |R(\gamma)|^\frac{1}{n+p}. 
        \end{split} \label{eq: annular lower bound}
    \end{align}

    Next, we define an auxiliary sequence of disjoint sets for every $i =1,2,\dots$ as
    \begin{align*}
        A_i = R_i(\gamma)\setminus R_{i-1}(\gamma) \subseteq R_i(\gamma).
    \end{align*}
    It clearly follows from the construction that 
    \begin{align*}
        \bigcup_{i=1}^\infty R_i(\gamma) = \bigcup_{i=1}^\infty A_i,
    \end{align*}
    so the union of the auxiliary sequence also approximates $R(\gamma)$, that is,
    \begin{align}
        \Big|R(\gamma)\setminus \bigcup_{i=1}^\infty A_i \Big| = 0. \label{eq: annular approximation}
    \end{align}
    To estimate the measure of each $A_i$, we consider the cases $i = 1$ and $i = 2,3,\dots$ separately. For $i =1$ we opt for a crude estimate to match better the estimate the other case. We have
    \begin{align*}
        |A_1| = 2^{-(n+p)} |R(\gamma)| \leq  2^{-1} (n+p) |R(\gamma)|.
    \end{align*}
    For the other indices $i =2,3,\dots$ the same estimate is sharper as
    \begin{align}
        \begin{split}
            |A_i| &= |R_i(\gamma)| - |R_{i-1}(\gamma)| = \big((1-2^{-i})^{n+p}-(1-2^{-i+1})^{n+p}\big) |R(\gamma)|  \\
            &=2^{-(n+p)i} \cdot \big((2^{i}-1)^{n+p} - (2^{i}-2)^{n+p}\big) |R(\gamma)|  \\
            &=2^{-(n+p)i} (n+p) |R(\gamma)|  \int_{2^{i}-2}^{2^{i} -1} \xi^{n+p-1} d\xi  \\
            &\leq 2^{-(n+p)i} (n+p) |R(\gamma)|\int_{2^{i}-2}^{2^{i} -1} (2^i)^{n+p-1} d\xi  \\
            &=2^{-i}(n+p)|R(\gamma)|.
        \end{split} \label{eq: annular upper bound}
    \end{align}
    
    Finally, the proof is completed by combining (\ref{eq: annular approximation}), (\ref{eq: annular lower bound}) and (\ref{eq: annular upper bound}) to obtain
    \begin{align*}
        \int_{R(\gamma)} \dist_p\big((x,t),E\big)^{-\alpha(n+p)} \dx \dt &= \sum_{i=1}^\infty \int_{A_i} \dist_p\big((x,t),E\big)^{-\alpha(n+p)} \dx \dt \\
        &\leq  \sum_{i=1}^\infty \int_{A_i}C_1^{-\alpha(n+p)} 2^{\alpha i(n+p)} |R(\gamma)|^{-\alpha} \dx \dt \\
        &= C_1^{-\alpha(n+p)} |R(\gamma)|^{-\alpha} \sum_{i=1}^\infty 2^{\alpha i(n+p)} |A_i| \\
        &\leq C_1^{-\alpha(n+p)} (n+p) |R(\gamma)|^{1-\alpha} \sum_{i=1}^\infty 2^{(\alpha(n+p) -1) i} \\
        &= C_2|R(\gamma)|^{1-\alpha},
    \end{align*}
    where the sum converges as $\alpha(n+p) - 1 <0$, and thus, $C_2 = C_2(n,p,\alpha) < \infty$.
\end{proof}

The proof the main theorem of this section, see Theorem \ref{theo: alpha improvement => A1}, uses heavily the following collections. Given $\theta\in \R$, a nonempty closed set $E \subseteq \Rn$, a parabolic rectangle $R(\gamma_0) \subseteq \Rn$ with $0\leq \gamma_0 \leq 1/2$ and a decreasing sequence $(\delta_i)_{i\in \N}$ with $\delta_i \rightarrow 0$, we define for every $i = 1,2,\dots$ that
\begin{align}
    \mathcal{C}_0 = \mathcal{C}_0\big(R(\gamma_0)\big)= \mathcal{F}_{\delta_0}^\theta\big(R(\gamma_0)\big) \quad \text{and} \quad \mathcal{C}_i = \mathcal{C}_i\big(R(\gamma_0)\big)=\mathcal{F}_{\delta_i}^\theta\big(R(\gamma_0)\big)\setminus \mathcal{F}_{\delta_{i-1}}^\theta\big(R(\gamma_0)\big). \label{eq: C collections}
\end{align}
If we further assume that $E$ is an $\alpha$-improving parabolic weakly porous set, these collections have some structure which we present in the following lemma.

\begin{lemma} \label{lem: towering sets}
    Suppose $E\subseteq \Rn$ is $(c_i, \delta_i,\theta)$-weakly porous for some $\theta \in \R$ and  $c_i,\delta_i \in (0,1)$ for every $i \in \N$, where $(\delta_i)_{i\in \N}$ is a decreasing sequence with $\delta_i  \rightarrow 0$, and $0 < 1- c_i \leq K \delta_i^{\alpha}$ for $\alpha,K >0$. Let $R(\gamma_0) \subseteq \Rn$ be a parabolic rectangle with $0\leq \gamma_0 \leq 1/2$, and define the collections $\mathcal{C}_i = \mathcal{C}_i\big(R(\gamma_0)\big)$ as above. Then, the following are true:
    
    \begin{enumerate}[label=(\roman*)]

        \item For any $k \in \N$ we have
        \[\bigcup_{i=0}^k \mathcal{C}_i = \mathcal{F}_{\delta_k}^\theta\big(R(\gamma_0)\big).\] \label{item: towering sets1}
        
        \item For every $i,j\in \N$ and for every $P(\gamma) \in \mathcal{C}_i$ with $0 \leq \gamma \leq 1/2$ and $Q(\beta) \in \mathcal{C}_j$ with $0\leq \beta \leq 1/2$ such that $P(\gamma) \neq Q(\beta)$ we have $P(\gamma) \cap Q(\beta) = \emptyset$. Moreover, if $i \neq j$, then necessarily $P(\gamma) \neq Q(\beta)$. \label{item: towering sets2}
        
        \item For every $i =1,2,\dots$ we have
        \[\sum_{P(\gamma) \in \mathcal{C}_i} |P(\gamma)| \leq (1-c_{i-1})|R(\gamma_0)| \leq K \delta_{i-1}^\alpha |R(\gamma_0)|.\] \label{item: towering sets3}
        
        \item Collections $\mathcal{C}_i$ satisfy \[\Big|R(\gamma_0) \setminus \bigg(\bigcup_{i=0}^\infty \bigcup_{P(\gamma) \in \mathcal{C}_i}P(\gamma)\bigg)\Big| = 0.\]\label{item: towering sets4}
    \end{enumerate}
\end{lemma}
\begin{proof}
    \ref{item: towering sets1}  Let $P(\gamma) \in \mathcal{F}_{\delta_i}^\theta = \mathcal{F}_{\delta_i}^\theta \big(R(\gamma_0)\big) $ with $0\leq \gamma \leq 1/2$ for any $i \in \N$.  Clearly,
    \begin{align*}
        |P(\gamma)| \geq \delta_{i} \big|\mathcal{M}\big(R^\theta(\gamma_0)\big)\big| \geq \delta_{i+1} \big|\mathcal{M}\big(R^\theta(\gamma_0)\big)\big|,
    \end{align*}
    since $(\delta_i)_{i\in \N}$ is decreasing sequence. Thus, by the maximality of $P(\gamma)$ we have $P(\gamma) \in \mathcal{F}_{\delta_{i+1}}^\theta$, implying $\mathcal{F}_{\delta_i}^\theta \subseteq \mathcal{F}_{\delta_{i+1}}^\theta$ for every $i \in \N$. It is now clear from the definition of the collections $\mathcal{C}_i$ that
    \begin{align*}
        \bigcup_{i=0}^k \mathcal{C}_i = \mathcal{F}_{\delta_k}^\theta\big(R(\gamma_0)\big).
    \end{align*}
    for any $k\in \N$.

    \ref{item: towering sets2}  Take any $P(\gamma) \in \mathcal{C}_i \subseteq \mathcal{F}_{\delta_{i}}^\theta$ with $0\leq \gamma \leq 1/2$. From the properties of $\mathcal{F}_{\delta_{i}}^\theta$ it is clear that $\mathcal{C}_i$ is a pairwise disjoint collection and that $P(\gamma)$ is maximal. We then take $Q(\beta) \in \mathcal{C}_j \subseteq \mathcal{F}_{\delta_j}^\theta$ with $0 \leq \beta \leq 1/2$ for $j\neq i$. By the definition of the collections $\mathcal{C}_i$, then $P(\gamma) \neq Q(\beta)$.
    
    Without loss of generality we may assume $i>j$. Since $P(\gamma),Q(\beta) \in \mathcal{D}\big(R(\gamma_0)\big)$, by the nestedness of the dyadic lattice we have either $P(\gamma) \subset Q(\beta)$, $Q(\beta) \subset P(\gamma)$ or $P(\gamma) \cap Q(\beta) = \emptyset$. However, the two former cases lead to a contradiction, as by maximality, neither set can be contained in the other. This leaves us with only $P(\gamma) \cap Q(\beta) = \emptyset$.

    \ref{item: towering sets3}  By item \ref{item: towering sets2} and item \ref{item: towering sets1}, for every $i=1,2,\dots$ we must have 
    \begin{align*}
        \bigcup_{P(\gamma) \in \mathcal{C}_i} P(\gamma) \subseteq R(\gamma_0) \setminus \bigcup_{j=0}^{i-1} \bigcup_{Q(\beta) \in \mathcal{C}_{j}}Q(\beta) = R(\gamma_0) \setminus   \bigcup_{Q(\beta) \in \mathcal{F}_{\delta_{i-1}}^\theta}Q(\beta).
    \end{align*}
    Combining the result with Proposition \ref{prop: Fdelta} implies that
    \begin{align*}
        \sum_{P(\gamma) \in \mathcal{C}_i} |P(\gamma)| \leq (1-c_{i-1})|R(\gamma_0)| \leq K \delta_{i-1}^\alpha|R(\gamma_0)|.
    \end{align*}

    \ref{item: towering sets4}  As $|R(\gamma_0)| < \infty$, by monotone convergence, item \ref{item: towering sets1} and Proposition \ref{prop: Fdelta} we obtain
    \begin{align*}
        &\phantom{=}\Big|R(\gamma_0) \setminus \bigg(\bigcup_{i=0}^\infty \bigcup_{P(\gamma) \in \mathcal{C}_i}P(\gamma)\bigg)\Big| = \Big|\bigcap_{i=0}^\infty \Big(R(\gamma_0)\setminus \bigcup_{P(\gamma) \in \mathcal{C}_i}P(\gamma)\Big)\Big| \\
        &= \lim_{k \rightarrow \infty}\Big|\bigcap_{i=0}^k \Big(R(\gamma_0)\setminus \bigcup_{P(\gamma) \in \mathcal{C}_i}P(\gamma)\Big)\Big| =\lim_{k \rightarrow \infty}\Big|R(\gamma_0)\setminus \bigcup_{i=0}^k\bigcup_{P(\gamma) \in \mathcal{C}_i}P(\gamma)\Big| \\
        &=\lim_{k \rightarrow \infty}\Big|R(\gamma_0)\setminus \bigcup_{P(\gamma) \in  \mathcal{F}_{\delta_k}^\theta}P(\gamma)\Big| = \lim_{k\rightarrow \infty}\bigg(|R(\gamma_0)| - \sum_{P(\gamma)\in \mathcal{F}_{\delta_k}^\theta}|P(\gamma)|\bigg)\\
        &\leq \lim_{k\rightarrow \infty} (1 -c_k)|R(\gamma_0)| \leq \lim_{k \rightarrow \infty} K \delta_k^\alpha  |R(\gamma_0)| = 0.
    \end{align*}
\end{proof}

We are ready to prove that $\alpha$-improvement of a parabolic weakly porous set implies $A_1$-type inequalities. Observe that the following theorem holds for any translation $\theta \in \R$. In particular, when $\theta>1$, then we will have a characterization for distance weights in $A_1^+(\gamma)$, see Corollary \ref{cor: alpha improvement => A1}. A strategy similar to ours can be found in \cite{Vargas2026}.
\begin{theorem} \label{theo: alpha improvement => A1}
    Suppose $E \subseteq \Rn$ is $(c_0,\delta_0,\theta)$-weakly porous for some and $c_0,\delta_0 \in(0,1)$ and $\theta \in \R$.  If $E$ is $\alpha$-improving with constants $0< \alpha<(n+p)^{-1}$ and $K >0$, then for every $0<\beta < \alpha$ there exists a constant $C = C(n,p,d,\delta_0,\theta,K,\alpha,\beta) >0$ such that 
    \begin{align*}
        \dashint_{R(\gamma_0)} \dist_p\big((x,t),E\big)^{-\beta(n+p)} \dx \dt \leq C \essinf_{(x,t)\in R^\theta(\gamma_0)} \dist_p\big((x,t),E\big)^{-\beta(n+p)}.
    \end{align*}
    for every parabolic rectangle $R(\gamma_0) \subseteq \Rn$ with $0\leq \gamma_0 \leq 1/2$.
\end{theorem}
\begin{proof}
     Let $R(\gamma_0) \subseteq \Rn$ be a parabolic rectangle with $0 \leq \gamma_0 \leq 1/2$. We choose the sequence $(\delta_i)_{i\in \N}$ with $\delta_i = 2^{-i} \delta_0$, and by $\alpha$-improvement $E$ is $(c_i,\delta_i, \theta)$-weakly porous for every $i \in \N$ where $0<c_i<1$ satisfies
     \begin{align*}
         1-c_i \leq K \delta_i^\alpha.
     \end{align*}
     
    Next, we recall the collections $\mathcal{C}_i = \mathcal{C}_i\big(R(\gamma_0)\big)$. By Lemma \ref{lem: towering sets}\ref{item: towering sets4} the collections can be used to cover $R(\gamma_0)$ almost everywhere, that is,
    \begin{align*}
        \Big|R(\gamma_0) \setminus \bigg(\bigcup_{i=0}^\infty \bigcup_{P(\gamma) \in \mathcal{C}_i}P(\gamma)\bigg)\Big| = 0.
    \end{align*}
    Thus, by setting $0<\beta < \alpha$ we may write using Lemma \ref{lem: annular estimate} that
    \begin{align}
        \begin{split}
            &\phantom{=}\int_{R(\gamma_0)} \dist_p\big((x,t),E\big)^{-\beta(n+p)} \dx \dt \\
            &= \int_{R(\gamma_0)}  \sum_{i=0}^\infty \sum_{P(\gamma) \in \mathcal{C}_i} \mathbbm{1}_{P(\gamma)}(x,t) \dist_p\big((x,t),E\big)^{-\beta(n+p)} \dx \dt \\
            &= \sum_{i=0}^\infty \sum_{P(\gamma) \in \mathcal{C}_i} \int_{P(\gamma)}  \dist_p\big((x,t),E\big)^{-\beta(n+p)} \dx \dt \\
            &\leq C_0\sum_{i=0}^\infty \sum_{P(\gamma) \in \mathcal{C}_i}  |P(\gamma)|^{1-\beta},
        \end{split} \label{eq: A1 pdc1}
    \end{align}
    where $C_0 = C_0(n,p,\beta)$ is from Lemma \ref{lem: annular estimate}.
    
    We study the sums of (\ref{eq: A1 pdc1}) separately. By $E$ being $(c_0,\delta_0,\theta)$-weakly porous we obtain the estimate
    \begin{align*}
        \sum_{P(\gamma) \in \mathcal{C}_0}  |P(\gamma)|^{1-\beta} &= \sum_{P(\gamma) \in \mathcal{F}_{\delta_0}^\theta}  |P(\gamma)|^{1-\beta} \leq \delta_0^{-\beta} |R(\gamma_0)||\mathcal{M}\big(R^\theta(\gamma_0)\big)|^{-\beta}.
    \end{align*}
    On the other hand, $E$ being $(c_i,\delta_i,\theta)$-weakly porous for every $i = 1,2,\dots$ and Lemma \ref{lem: towering sets}\ref{item: towering sets3} imply
    \begin{align*}
        \sum_{P(\gamma) \in \mathcal{C}_i}  |P(\gamma)|^{1-\beta} &\leq \delta_i^{-\beta} \sum_{P(\gamma) \in \mathcal{C}_i}  |P(\gamma)| |\mathcal{M}\big(R^\theta(\gamma_0)\big)|^{-\beta} \\
        &\leq \delta_i^{-\beta} K \delta_{i-1}^{\alpha} |R(\gamma_0)| |\mathcal{M}\big(R^\theta(\gamma_0)\big)|^{-\beta} \\
        &= 2^\alpha K \delta_0^{\alpha- \beta} \cdot 2^{-(\alpha - \beta)i} |R(\gamma_0) |\mathcal{M}\big(R^\theta(\gamma_0)\big)|^{-\beta} \\
        &= C_1 2^{-\varepsilon i}  |R(\gamma_0)| |\mathcal{M}\big(R^\theta(\gamma_0)\big)|^{-\beta},
    \end{align*}
    where $\varepsilon = \alpha -\beta > 0$ and $C_1 = 2^\alpha K\delta_0^\varepsilon$. Substituting the estimates to (\ref{eq: A1 pdc1}) and dividing by $|R(\gamma_0)|$ yields
    \begin{align}
        \begin{split}
            \dashint_{R(\gamma_0)} \dist_p\big((x,t),E\big)^{-\beta(n+p)} \dx \dt  &\leq   |\mathcal{M}\big(R^\theta(\gamma_0)\big)|^{-\beta}C_0\bigg(\delta_0^{-\beta}+ C_1\sum_{i=1}^\infty (2^{-\varepsilon})^i\bigg) \\
            &=  C_2 |\mathcal{M}\big(R^\theta(\gamma_0)\big)|^{-\beta}, 
        \end{split}\label{eq: A1 pdc2}
    \end{align}
    where $C_2 =C_2(n,p, \delta_0, K, \alpha,\beta) <\infty$ as $\varepsilon >0$.
    
    Since the maximal hole function is closely related to the essential supremum of the distance function, the estimate (\ref{eq: A1 pdc2}) is rather close to the desired form. For the next part, we denote $l_x\big(R(\gamma_0)\big) = L$ for simplicity, and restrict ourselves to study three cases:
    \begin{enumerate}[label=(\roman*)]
        \item $R^\theta(\gamma_0) \cap E \neq \emptyset$ \label{item: A1 pdc1 case 1}
        
        \item $R^\theta(\gamma_0) \cap E = \emptyset$ and $\dist_p\big(R^\theta(\gamma_0), E\big) \leq  2(|\theta| +1)^\frac{1}{p} L$ \label{item: A1 pdc1 case 2}

        \item $R^\theta(\gamma_0) \cap E = \emptyset$ and $\dist_p\big(R^\theta(\gamma_0), E\big) > 2(|\theta| +1)^\frac{1}{p} L$. \label{item: A1 pdc1 case 3}
    \end{enumerate}
    
    \ref{item: A1 pdc1 case 1} Suppose $R^\theta(\gamma_0) \cap E \neq \emptyset$. This is the easy case as then by Proposition \ref{prop: maximal hole}\ref{item: maximal hole 2} the essential supremum is directly comparable to the side length of the maximal hole. Recall $E$ is implicitly assumed to be closed. We have
    \begin{align*}
        \esssup_{(x,t)\in R^\theta(\gamma_0)} \dist_p\big((x,t),E\big) &\leq 2^d l_x\big(\mathcal{M}\big(R^\theta(\gamma_0)\big)\big)\\
        &\leq 2^{d}(1-\gamma_1)^{-\frac{1}{n+p}}|\mathcal{M}\big(R^\theta(\gamma_0)\big)|^{\frac{1}{n+p}} \\
        &= C_3 |\mathcal{M}\big(R^\theta(\gamma_0)\big)|^{\frac{1}{n+p}},
    \end{align*}
    where $C_3 = 2^{d+\frac{1}{n+p}}$ and $0\leq\gamma_1\leq 1/2$ is the truncation parameter of $\mathcal{M}\big(R^\theta(\gamma_0)\big)$. Raising both sides of the inequality above to the power of $\beta(n+p)$ and reordering yields
    \begin{align*}
        |\mathcal{M}\big(R^\theta(\gamma_0)\big)|^{-\beta} \leq C_3^{\beta(n+p)} \essinf_{(x,t) \in R^\theta(\gamma_0)} \dist_p\big((x,t),E\big)^{-\beta(n+p)}.
    \end{align*}
    Substituting the above into (\ref{eq: A1 pdc2}) results in
    \begin{align}
        \dashint_{R(\gamma_0)} \dist_p((x,t),E)^{-\beta(n+p)} \dx \dt \leq C_2 C_3^{\beta(n+p)} \essinf_{(x,t) \in R^\theta(\gamma_0)} \dist_p\big((x,t),E\big)^{-\beta(n+p)}, \label{eq: A1 pdc3}
    \end{align}
    proving the claim for the first case.

    \ref{item: A1 pdc1 case 2} Suppose $R^\theta(\gamma_0) \cap E = \emptyset$ and $\dist_p\big(R^\theta(\gamma_0), E\big) \leq  2(|\theta| +1)^\frac{1}{p} L$. Observe that is this case $\mathcal{M}\big(R^\theta(\gamma_0)\big) = R^\theta(\gamma_0)$. Since the $E$ is not very far from $R^\theta(\gamma_0)$ in this case, the parabolic distance between any $(x,t) \in R^\theta(\gamma_0)$ and $E$ is still comparable to the side length of $R^\theta(\gamma_0)$ and consequently to the side length of $\mathcal{M}\big(R^\theta(\gamma_0)\big)$. To show this, we use the fact that the parabolic distance is a metric and is thereby a subject to the triangle inequality. Hence, we get
    \begin{align*}
        \dist_p\big((x,t),E\big) &\leq \text{diam}_p\big(R^\theta(\gamma_0)\big) + \dist_p\big(R^\theta(\gamma_0), E\big) \leq \big(1+ 2(|\theta| +1)^\frac{1}{p}\big)L\\
        &= \big(1+ 2(|\theta| +1)^\frac{1}{p}\big) l_x\big(\mathcal{M}\big(R^\theta(\gamma_0)\big)\big) \\
        &= (1-\gamma_1)^{-\frac{1}{n+p}}\big(1+2(|\theta| +1)^\frac{1}{p}\big) |\mathcal{M}\big(R^\theta(\gamma_0)\big)|^{\frac{1}{n+p}} \\
        &\leq C_4 |\mathcal{M}\big(R^\theta(\gamma_0)\big)|^{\frac{1}{n+p}},
    \end{align*}
    where $C_4 = 2^{\frac{1}{n+p}}\big(1+2(|\theta| +1)^\frac{1}{p}\big)$. Since the right hand side is independent of $(x,t)$, we can take the essential supremum over every $(x,t) \in R^\theta(\gamma_0)$ and raise both sides to the power of $\beta(n+p)$. After some some reordering, we get
    \begin{align*}
        |\mathcal{M}\big(R^\theta(\gamma_0)\big)|^{-\beta} \leq C_4^{\beta(n+p)} \essinf_{(x,t) \in R^\theta(\gamma_0)} \dist_p\big((x,t),E\big)^{-\beta(n+p)}.
    \end{align*}
    Substituting the above into (\ref{eq: A1 pdc2}) results in
    \begin{align}
        \dashint_{R(\gamma_0)} \dist_p\big((x,t),E\big)^{-\beta(n+p)} \dx \dt \leq C_2 C_4^{\beta(n+p)} \essinf_{(x,t) \in R^\theta(\gamma_0)} \dist_p\big((x,t),E\big)^{-\beta(n+p)}, \label{eq: A1 pdc4}
    \end{align}
    proving the claim in the second case.

    \ref{item: A1 pdc1 case 3} Suppose $R^\theta(\gamma_0) \cap E = \emptyset$ and $\dist_p\big(R^\theta(\gamma_0), E\big) >  2(|\theta| +1)^\frac{1}{p} L$. In this case it is no longer possible compare the parabolic distance to the side length of $\mathcal{M}\big(R^\theta(\gamma_0)\big)$. This means that the estimate (\ref{eq: A1 pdc2}) is no longer useful. Fortunately, now the set $E$ is also far apart from $R(\gamma_0)$.
    
    Let us take any $(x,t) \in R^\theta(\gamma_0)$ and $(y,s) \in R(\gamma_0)$. The triangle inequality of the parabolic metric is useful also here. We get
    \begin{align*}
        &\phantom{=}\dist_p\big((y,s),E\big) \geq \dist_p\big((x,t),E\big) - \dist_p\big((x,t),(y,s)\big) \\
        &\geq \dist_p\big((x,t),E\big) -\max \bigg\{ L ,\Big(\dist_t\big(\partial_{\textnormal{low}}R(\gamma_0), \partial_{\textnormal{low}}R^\theta(\gamma_0)\big) + l_t\big(R(\gamma_0)\big)\Big)^\frac{1}{p} \bigg\} \\
        &= \dist_p\big((x,t),E\big) - \max \Big\{1, \Big((1-\gamma_0)(|\theta|+1)\Big)^\frac{1}{p} \Big\}L\\
        &\geq \dist_p\big((x,t),E\big) - (|\theta|+1)^\frac{1}{p} L\\
        &> \dist_p\big((x,t),E\big) - \frac{1}{2}\dist_p\big(R^\theta(\gamma_0), E\big)\\
        &\geq \frac{1}{2}\dist_p\big((x,t),E\big).
    \end{align*}
    Since the left hand side is independent of any $(x,t)\in R^\theta(\gamma_0)$, we can take the essential supremum over every $(x,t) \in R^\theta(\gamma_0)$ and raise both sides to power of $-\beta(n+p)$. We get
    \begin{align*}
        \dist_p\big((y,s),E\big)^{-\beta(n+p)} \leq 2^{\beta(n+p)} \essinf_{(x,t) \in R^\theta(\gamma_0)} \dist_p\big((x,t),E\big)^{-\beta(n+p)},
    \end{align*}
    for which we can take the integral average over $R(\gamma_0)$ on both sides, yielding
    \begin{align}
        \dashint_{R(\gamma_0)} \dist_p\big((y,s),E\big)^{-\beta(n+p)} \dy \ds \leq 2^{\beta(n+p)} \essinf_{(x,t) \in R^\theta(\gamma_0)} \dist_p\big((x,t),E\big)^{-\beta(n+p)}. \label{eq: A1 pdc5}
    \end{align}
    
    We conclude the proof by combining the estimates (\ref{eq: A1 pdc3}), (\ref{eq: A1 pdc4}) and (\ref{eq: A1 pdc5}) to obtain
    \begin{align*}
        \dashint_{R(\gamma_0)} \dist_p\big((x,t),E\big)^{-\beta(n+p)} \dx \dt \leq C \essinf_{(x,t) \in R^\theta(\gamma_0)} \dist_p\big((x,t),E\big)^{-\beta(n+p)},
    \end{align*}
    where $C = C(n,p,d,\delta_0,\theta, K, \alpha, \beta) = \max \big\{ 2^{\beta(n+p)}, C_2 C_3^{\beta(n+p)}, C_2 C_4^{\beta(n+p)} \big\}$.
\end{proof}

As a natural corollary, we obtain an $A_1^+(\gamma)$ distance weights characterization via $\alpha$-improvement of parabolic weakly porous sets. The first direction was already shown in Theorem \ref{theo: A1 => wp}. 
\begin{corollary} \label{cor: alpha improvement => A1}
    Suppose $E \subseteq \Rn$ is $(c,\delta,\theta)$-weakly porous for some $c,\delta \in (0,1)$ and $\theta >1$.  If $E$ is $\alpha$-improving for $0< \alpha<(n+p)^{-1}$, then $\dist_p(\cdot,E)^{-\beta(n+p)} \in A_1^+(\gamma)$ for $0<\gamma<1$ for every $0<\beta < \alpha$.
\end{corollary}
\begin{proof}
    Let $\theta>1$ and $0< \beta <\alpha< (n+p)^{-1}$. Then, Theorem \ref{theo: alpha improvement => A1} and Lemma \ref{lem: A1 starting point} imply that $\dist_p(\cdot,E)^{-\beta(n+p)} \in A_1^+(\gamma)$ for any $0 <\gamma <1$.
\end{proof}

\section{Doubling and translation results} \label{sec: doubling results}

The time-lag invariance of $A_1^+(\gamma)$, see \cite[Theorem 3.1]{KinnunenMyyry2024}, motivates to show similar results for the parabolic weak porosity. Furthermore, the doubling of the maximal hole function in \cite[Lemma 3.2]{ALMV24} also have an analogy as the forward-in-time doubling of the maximal hole function. In this section, we demonstrate both of these features, which are necessary in Section \ref{sec: stopping conditions} and Section \ref{sec: full charac}, while also being interesting as such. 

Here it is essential that we have a positive time-lag, that is, $\theta>1$ to be able to formulate necessary chaining arguments. Moreover, the parabolic geometry via $p>1$ plays an important role. However, instead of proving separately the time-lag invariance and the doubling property, it turns out that we need a stronger result that combines these two into one theorem.

\begin{theorem} \label{theo: doubling porosity}
    Suppose $E \subseteq \Rn$ is $(c, \delta,\theta)$-weakly porous for some $c,\delta \in (0,1)$ and $\theta >1$. Then, for every $\psi > 1$ there exists $0<\sigma <1$ such that for any pair of parabolic rectangles $R(\gamma_0)\subseteq \Rn$ and $P(\gamma) \in \mathcal{D}_1\big(R(\gamma_0)\big)$ with $\gamma_0,\gamma \in [0,1/2]$ there exist pairwise disjoint $E$-free $S_i(\alpha_i) \in \mathcal{D}\big(P(\gamma)\big)$ with $0\leq \alpha_i \leq 1/2$ such that
    \begin{align*}
        |S_i(\alpha_i)| \geq \sigma |\mathcal{M}\big(R^\psi(\gamma_0)\big)|
    \end{align*}
    for each $i = 1,2,\dots, N$ and 
    \begin{align*}
        \sum_{i=1}^N |S_i(\alpha_i)| \geq c |P(\gamma)|.
    \end{align*}
     In particular, given any $0<c_0<1$ and $1<\theta_1\leq \theta_2 < \infty$, if $c\geq c_0$ and $\psi \in [\theta_1,\theta_2]$, then $\sigma = C \delta^{\nu}$ for constants $C = C(n,p,d,c_0,\theta,\theta_1,\theta_2) >0$ and $\nu = \nu(n,p,d,c_0,\theta,\theta_1,\theta_2)>0$.
\end{theorem}

\begin{proof}
    Let $\psi \in [\theta_1,\theta_2]$ for $1<\theta_1 \leq \theta_2 < \infty$ and let $0<c_0\leq c$. For parabolic rectangles $R(\gamma_0) \subseteq \Rn$ and $P(\gamma_1) \in \mathcal{D}_1\big(R(\gamma_0)\big)$ with $\gamma_0,\gamma_1 \in[0,1/2]$, let us consider the dyadic rectangles $Q_0(\alpha_1) \in \mathcal{D}_{m}\big(P(\gamma_1)\big)$ with $0\leq \alpha_1 \leq 1/2$ for some $m \in \N$ yet to be chosen. It follows that then $Q_0(\alpha_1) \in \mathcal{D}_{m+1}\big(R(\gamma_0)\big)$. By Corollary \ref{cor: dyadic scale}, the spatial side lengths of $Q_0(\alpha_1)$ are
    \begin{align}
        L_x = l_x\big(Q_0(\alpha_1)\big) = 2^{-(m+1)d}l_x\big(R(\gamma_0)\big), \label{eq: translation Lx}
    \end{align} 
    while the temporal side length $L_t = l_t\big(Q_0(\alpha_1)\big) $ satisfies
    \begin{align}
         2^{-(m+1)dp-1}l_t\big(R(\gamma_0)\big) \leq L_t \leq 2^{-(m+1)dp +1}l_t\big(R(\gamma_0)\big). \label{eq: translation Lt}
    \end{align}

    For each $Q_0(\alpha_1)$ there exists the collection $\mathcal{F}_\delta^\theta\big(Q_0(\alpha_1)\big)$, that is any $S(\beta_1) \in \mathcal{F}_\delta^\theta\big(Q_0(\alpha_1)\big)$ with $0\leq \beta_1 \leq 1/2$ is $E$-free and satisfies
    \begin{align}
        |S(\beta_1)| \geq \delta |\mathcal{M}\big(Q_0^\theta(\alpha_1)\big)|. \label{eq: translation delta admissible}
    \end{align}
    Since $E$ is $(c,\delta, \theta)$-weakly porous, by Proposition \ref{prop: Fdelta} we have
    \begin{align*}
        F\big(Q_0(\alpha_1)\big)=\sum_{S(\beta_1) \in \mathcal{F}_\delta^\theta(Q_0(\alpha_1))} |S(\beta_1)| \geq c |Q_0(\alpha_1)|.
    \end{align*}
    Clearly, $\mathcal{F}_\delta^\theta \big(Q_0(\alpha_1)\big) \subseteq \mathcal{D}\big(P(\gamma_1)\big)$ is a finite collection, so taking the union of these collections over every $Q_0(\alpha_1) \in \mathcal{D}_m(P(\gamma_1)\big)$ results in a finite collection of pairwise disjoint $E$-free dyadic subrectangles of $P(\gamma_1)$. Furthermore, the covering property of the dyadic lattice implies that
    \begin{align*}
        \sum_{Q_0(\alpha_1) \in \mathcal{D}_m(P(\gamma_1))} F\big(Q_0(\alpha_1)\big) &\geq c \sum_{Q_0(\alpha_1) \in \mathcal{D}_m(P(\gamma_1))} |Q_0(\alpha_1)| = c |P(\gamma_1)|.
    \end{align*}
    
    Next, we want to show that the measures of the rectangles $S(\beta_1)$ are large enough compared to the maximal hole of $R^\psi(\gamma_0)$. Since we already have the lower bound (\ref{eq: translation delta admissible}), we estimate $|\mathcal{M}\big(Q_0^\theta(\alpha_1)\big)|$ with $|\mathcal{M}\big(R^\psi(\gamma_0)\big)|$. To do this, we need a chaining argument. Observe first that by the nestedness of the dyadic lattice, there exists some $M(\gamma_2) \in \mathcal{D}_{m+1}\big(R^\psi(\gamma_0)\big)$ with $0\leq \gamma_2 \leq 1/2$ such that 
    \begin{align*}
        \mathcal{M}\big(R^\psi(\gamma_0)\big) \subseteq M(\gamma_2) \quad \text{or} \quad M(\gamma_2) \subseteq \mathcal{M}\big(R^\psi(\gamma_0)\big).
    \end{align*}
    The inclusions will imply respectively that
    \begin{align*}
        |\mathcal{M}\big(M(\gamma_2)\big)| = |\mathcal{M}\big(R^\psi(\gamma_0)\big)| \quad \text{or} \quad |\mathcal{M}\big(M(\gamma_2)\big)| = |M(\gamma_2)| \geq 2^{-(m+1)d(n+p)-1}|R^\psi(\gamma_0)|. 
    \end{align*}
    The second inequality above follows from the comparability of the dyadic lattice. Naturally, we obtain a lower bound
    \begin{align}
        |\mathcal{M}\big(M(\gamma_2)\big)| \geq 2^{-(m+1)d(n+p)-1}|\mathcal{M}\big(R^\psi(\gamma_0)\big)|. \label{eq: translation M}
    \end{align}
    
    The plan is to construct a chain of rectangles from any $Q_0(\alpha_1)$ to $M(\gamma_2)$. Note that since both $Q_0(\alpha_1)$ and $M(\gamma_2)$ are dyadic rectangles of $R(\gamma_0)$ and $R^\psi(\gamma_0)$ of order $m+1$, then $\gamma_2 = \alpha_1$ by similarity of the dyadic layers. This means that we can write
    \begin{align*}
        M(\gamma_2) = Q_0(\alpha_1) + (y,s)
    \end{align*}
    for some $y = (y_1,\dots,y_n) \in \R^n$ and $s \in \R$. In particular, each $y_k$ is bounded by the side length of the whole $R(\gamma_0)$ for every $k = 1,\dots,n$, which means
    \begin{align}
        \lVert y \rVert_\infty \leq l_x\big(R(\gamma_0)\big). \label{eq: translation y boundary}
    \end{align}
    On the other hand, $s$ is bounded by the minimal and maximal temporal distance between $R(\gamma_0)$ and $R^\psi(\gamma_0)$ as
    \begin{align}
        (\psi-1) l_t\big(R(\gamma_0)\big) \leq s \leq (\psi+1)l_t\big(R(\gamma_0)\big). \label{eq: translation s boundary}
    \end{align}

    To construct the chain connecting $Q_0(\alpha_1)$ to $M(\gamma_2)$, we define recursively
    \begin{align*}
        Q_{i+1}(\alpha_1) = Q_i^{\theta}(\alpha_1) + (\xi_{i},\tau_i) =Q_i(\alpha_1) + (\xi_i,\theta L_t + \tau_i),
    \end{align*}
    where $\xi_i = (\xi_{i,1},\dots,\xi_{i,n})\in \R^n$ and $\tau_i \in \R$ for every $i\in \N$. It follows that we can write
    \begin{align*}
        Q_{k}^{\theta}(\alpha_1) &= Q_k(\alpha_1) +(0,\theta L_t)\\
        &= Q_0(\alpha_1) +(0,\theta L_t) + \sum_{i=0}^{k-1} \big( \xi_{i}, \theta L_t +\tau_i \big) \\
        &=Q_0(\alpha_1) + \bigg(\sum_{i=0}^{k-1} \xi_{i}, (k+1)\theta L_t + \sum_{i=0}^{k-1} \tau_i\bigg).
    \end{align*}
    Notice that each rectangle of the chain is always translated upwards some constant amount that depends on $\theta$. However, to correct any spatial and temporal misalignment, we have also introduced an extra correction terms $(\xi_i,\tau_i)$. We set $N_1 \in \N$ such that $M(\gamma_2) = Q_{N_1}^{\theta}(\alpha_1)$. By the above it is enough to require that the cumulative spatial and temporal corrections matches
    \begin{align}
        y = \sum_{i=0}^{N_1-1} \xi_{i}, \quad \text{and} \quad s = (N_1+1)\theta L_t + \sum_{i=0}^{N_1-1} \tau_i. \label{eq: translation correction sum}
    \end{align}
    
    It is important that each correction term $(\xi_i,\tau_i)$ is small enough to be able to use parabolic weak porosity to link the rectangles of the chain. To achieve this, we restrict the spatial and temporal corrections for each $i = 1,\dots,N_1$ by
    \begin{align}
        \lVert\xi_{i}\rVert_\infty< \varepsilon_{\textnormal{max}} L_x \quad \text{and} \quad 0 \leq \tau_i < \varepsilon_{\textnormal{max}} L_t. \label{eq: translation correction boundary}
    \end{align}
    where $\varepsilon_{\textnormal{max}} >0$. By selecting the maximal proportional correction as
    \begin{align*}
        \varepsilon_{\textnormal{max}} = 1 - (1-c_0/2)^\frac{1}{n+1} \leq 1,
    \end{align*}
    we guarantee that the rectangles of the chains overlap enough for the linking. In particular,
    \begin{align}
        \begin{split}
            |Q_{i+1}(\alpha_1) \setminus  Q_i^\theta(\alpha_1)| &= |Q_{i+1}(\alpha_1)|-(L_t-\tau_i)\prod_{k=1}^{n}\big(L_x-|\xi_{i,k}|\big) \\
            &< |Q_{i+1}(\alpha_1)|-(1-\varepsilon_{\textnormal{max}})^{n+1}L_t L_x^n \\
            &= |Q_{i+1}(\alpha_1)| - \Big(1- \frac{c_0}{2}\Big)|Q_{i+1}(\alpha_1)|  \\ 
            &= \frac{c_0}{2}|Q_{i+1}(\alpha_1)|.
        \end{split} \label{eq: translation Q intersection}
    \end{align}
    
    We then prove the linkage between $Q_i^\theta(\alpha_1)$ and $Q_{i+1}^\theta(\alpha_1)$. For any fixed $i\in \N$, let us consider the collection $\mathcal{F}_i = \mathcal{F}_\delta^\theta\big(Q_{i+1}(\alpha_1)\big)$. Thus, for $T(\alpha_2) \in \mathcal{F}_i$ we have
    \begin{align*}
        |T(\alpha_2)| \geq \delta |\mathcal{M}\big(Q_{i+1}^\theta(\alpha_1)\big)|.
    \end{align*}
    Assume that for every $T(\alpha_2) \in \mathcal{F}_i$ we have
    \begin{align}
        |T(\alpha_2) \setminus  Q_i^\theta(\alpha_1)| \geq \frac{1}{2}|T(\alpha_2)|. \label{eq: translation measure density}
    \end{align}
    However, since $E$ is $(c,\delta,\theta)$-weakly porous, Proposition \ref{prop: Fdelta} implies
    \begin{align*}
       \frac{c_0}{2}|Q_{i+1}(\alpha_1)| &\leq \frac{c}{2}|Q_{i+1}(\alpha_1)| \leq \sum_{T(\alpha_2) \in \mathcal{F}_i}\frac{1}{2}|T(\alpha_2)| \\
       &\leq \sum_{T(\alpha_2) \in \mathcal{F}_i} |T(\alpha_2) \setminus  Q_i^\theta(\alpha_1)|  \\
       &\leq |Q_{i+1}(\alpha_1) \setminus  Q_i^\theta(\alpha_1)|,
    \end{align*}
    which is a contradiction to (\ref{eq: translation Q intersection}). This means that the complement of (\ref{eq: translation measure density}) is true, and there exists $\delta$-admissible $\Tilde{T}(\alpha_2) \in \mathcal{F}_i$ such that 
    \begin{align*}
        |\Tilde{T}(\alpha_2) \cap  Q_i^\theta(\alpha_1)| \geq \frac{1}{2}|\Tilde{T}(\alpha_2)|.
    \end{align*}
    
    Denote $A_i = \Tilde{T}(\alpha_2) \cap  Q_i^\theta(\alpha_1)$. Since $A_i$ is an intersection of rectangular sets, $A_i$ is also a rectangular set. The measure condition implies then that necessarily the spatial and temporal side lengths of $A_i$ must be at least half of the side lengths of $\Tilde{T}(\alpha_2)$. Thus, there exists a parabolic rectangle $U(\alpha_2) \subseteq \Tilde{T}(\alpha_2) \cap  Q_i^{\theta}(\alpha_1)$ with side lengths $2^{-1} l_x\big(\Tilde{T}(\alpha_2)\big)$ and $2^{-p} l_t\big(\Tilde{T}(\alpha_2)\big)$. Consequently, $U(\alpha_2)$ is $E$-free and
    \begin{align*}
        |U(\alpha_2)| =  2^{-(n+p)} |\Tilde{T}(\alpha_2)| \geq 2^{-d(n+p)}\delta |\mathcal{M}\big(Q_{i+1}^{\theta}(\alpha_1)\big)|.
    \end{align*}
    By the approximation property of the dyadic lattice there also exists $V(\beta_2) \in \mathcal{D}\big(Q_i^{\theta}(\alpha_1)\big)$ with $0\leq \beta_2 \leq 1/2$ such that $V(\beta_2) \subseteq U(\alpha_2)$ and
    \begin{align*}
        |V(\beta_2)| \geq 4^{-d(n+p)} |U(\alpha_2)|.
    \end{align*}
    Since then $V(\beta_2) \subseteq \Tilde{T}(\alpha_2) \in \mathcal{F}_i$ is $E$-free, combining the estimates above, we have our crucial chaining result
    \begin{align}
        |\mathcal{M}\big(Q_i^\theta(\alpha_1)\big)| \geq |V(\beta_2)| \geq 4^{-d(n+p)} |U(\alpha_2)| \geq 8^{-d(n+p)} \delta |\mathcal{M}\big(Q_{i+1}^\theta(\alpha_1)\big)|.\label{eq: translation link}
    \end{align}
    
    Next, we define the correction terms $(\xi_i,\tau_i)$. For some auxiliary parameter $1 \leq N_2 \leq N_1$, let 
    \begin{align*}
        \xi_{i,k} &= 
        \begin{cases}
            \frac{y_k}{N_2}, \quad i = 0,\dots, N_2-1,\\
            0, \quad  i = N_2,\dots, N_1 -1
        \end{cases}
    \end{align*}
    for every $k = 1,\dots, n$, and
    \begin{align*}
        \tau_i = \begin{cases}
            \frac{s}{N_2} - \frac{N_1+1}{N_2}\theta L_t, \quad i = 0,\dots, N_2-1,\\
            0, \quad  i = N_2,\dots, N_1 -1.
        \end{cases}
    \end{align*}
    It is straightforward to check that the choice of $\xi_i$ and $\tau_i$ satisfies  (\ref{eq: translation correction sum}), meaning $Q_{N_1}^\theta(\alpha_1) = M(\gamma_2)$. However, we must still show that there exist $N_1$ and $N_2$ such that each pair of $\xi_i$ and $\tau_i$ satisfies (\ref{eq: translation correction boundary}).

     We first check (\ref{eq: translation correction boundary}) for $\tau_i$ given any $i = 1,\dots, N_1-1$. Obviously, if $i = N_2,\dots, N_1-1$, there is nothing to show. Hence, we let $i = 1,\dots, N_2-1$. Observe that
     \begin{align*}
        0 \leq \tau_i<\varepsilon_{\textnormal{max}}L_t
    \end{align*}
    is then equivalent with
    \begin{align*}
          (N_1 +1)\theta L_t \leq s < (N_1 +1)\theta L_t +N_2 \varepsilon_{\textnormal{max}} L_t.
    \end{align*}
    If we require $N_2$ so large such that
    \begin{align}
            N_2\varepsilon_{\text{max}} \geq \theta, \label{eq: translation N2}
    \end{align}
    we get a sufficient condition
    \begin{align}
         (N_1+1) \theta L_t \leq s < (N_1+2) \theta L_t. \label{eq: translation t suf}
    \end{align}
    We will choose $N_2$ satisfying (\ref{eq: translation N2}) later.
    
    By (\ref{eq: translation s boundary}) we have $s < \infty$ and a lower bound
    \begin{align*}
        s \geq (\psi-1) l_t\big(R(\gamma_0)\big) \geq (\theta_1-1) l_t\big(R(\gamma_0)\big),
    \end{align*}
    since $\psi \geq \theta_1$. Thus, there clearly exists some $N_1 \geq N_2$ satisfying (\ref{eq: translation t suf}) if 
    \begin{align*}
        (N_2 +1) \theta L_t \leq (\theta_1-1) l_t\big(R(\gamma_0)\big).
    \end{align*}
    Recalling (\ref{eq: translation Lt}), we can rearrange the terms to get another sufficient condition
    \begin{align*}
        2^{(m+1)dp} &\geq 2 \cdot \frac{(N_2+1)\theta}{\theta_1-1}.
    \end{align*}
    Since the parameter $m$ is free, the inequality above can be satisfied for any $N_2 \geq 1$ by setting $m$ to depend on $N_2$ This is important since by (\ref{eq: translation N2}) we require $N_2$ to be large enough. In other words, we choose 
    \begin{align}
        m  = \bigg\lceil \frac{1}{dp} \log_2\bigg( 2\cdot \frac{(N_2+1)\theta}{\theta_1-1}\bigg)\bigg\rceil -1. \label{eq: translation m}
    \end{align}
    
    We then check (\ref{eq: translation correction boundary}) for $\xi_i$, that is,
    \begin{align}
        \lVert\xi_{i}\rVert_\infty< \varepsilon_{\textnormal{max}} L_x. \label{eq: translation x suf}
    \end{align}
    Here again, we may let $i = 1,\dots, N_2-1$, since else $\xi_i = 0$. Now, by (\ref{eq: translation y boundary}) we have
    \begin{align*}
        |\xi_{i,k}| &= \Big|\frac{y_k}{N_2}\Big| \leq \frac{l_x\big(R(\gamma_0)\big)}{N_2}
    \end{align*}
    for any $k = 1,\dots,n$. Thus, (\ref{eq: translation x suf}) is satisfied if
    \begin{align*}
        \frac{l_x\big(R(\gamma_0)\big)}{N_2} < \varepsilon_{\textnormal{max}}L_x.
    \end{align*}
    Using (\ref{eq: translation Lx}), we may write the above equivalently in terms of the parameter $m$ as
    \begin{align*}
        N_2 > 2^{(m+1)d}\varepsilon_{\textnormal{max}}^{-1}.
    \end{align*}
    The choice of $m$ in (\ref{eq: translation m}) implies
    \begin{align}
        2^{(m+1)d}\varepsilon_{\textnormal{max}}^{-1} \leq 2^d\varepsilon_{\textnormal{max}}^{-1}\bigg(2\cdot \frac{(N_2+1)\theta }{\theta_1-1}\bigg)^\frac{1}{p} \leq C_1 N_2^\frac{1}{p}, \label{eq: translation N to Np}
    \end{align}
    where $C_1 = 2^d \varepsilon_{\textnormal{max}}^{-1}\big(4 \theta(\theta_1-1)^{-1}\big)^\frac{1}{p}$. Thus, we have one more sufficient condition for (\ref{eq: translation x suf}) as
    \begin{align*}
        C_1 N_2^\frac{1}{p} < N_2.
    \end{align*}
    Hence, to satisfy both (\ref{eq: translation N2}) and the above, we set
    \begin{align*}
        N_2 = \max \bigg\{ \Big\lceil C_1^\frac{p}{p-1}\Big\rceil +1,\big\lceil \theta \varepsilon_{\textnormal{max}}^{-1} \big\rceil \bigg\},
    \end{align*}
    which shows the existence of $N_1$ and $N_2$ such that (\ref{eq: translation correction boundary}) is satisfied.

    Before finishing the proof, we find $N_3 \in \N$ as an upper bound for $N_1 \leq N_3$. By (\ref{eq: translation s boundary}) we have the upper bound
    \begin{align*}
        s \leq (\psi +1)l_t\big(R(\gamma_0)\big) \leq (\theta_2 +1)l_t\big(R(\gamma_0)\big),
    \end{align*}
    where we used the fact $\psi \leq \theta_2$. If we set $N_3$ to some number such that
    \begin{align*}
        (\theta_2 +1)l_t\big(R(\gamma_0)\big) < (N_3 +2)\theta L_t,
    \end{align*}
    and since $N_1$ satisfies (\ref{eq: translation t suf}), then $N_1 \leq N_3$. Applying (\ref{eq: translation Lt}), it is enough to require 
    \begin{align*}
        N_3 > 2^{(m+1)dp+1}\frac{\theta_2+1}{\theta} -2.
    \end{align*}
    We choose $N_3$ to be smallest integer that achieves this, while being larger than $N_2$. We set
    \begin{align*}
        N_3 = \max \Bigg\{\bigg \lceil 2^{(m+1)dp+1}\frac{\theta_2+1}{\theta} \bigg\rceil -1, N_2 +1 \Bigg\}.
    \end{align*}
    Finally, we use a substitution $N_3 = \nu -1\in \R$ to make our final inequality in the same form as in the theorem.
    
    We have constructed the chain such that $Q_{N_1}^\theta(\alpha_1) = M(\gamma_2)$. Recall the rectangles $S_j(\beta_1)$ satisfying (\ref{eq: translation delta admissible}). We recursively apply (\ref{eq: translation link}), and also remember (\ref{eq: translation M}), to obtain
    \begin{align*}
        |S_j(\beta_j)| &\geq \delta |\mathcal{M}\big(Q_0^\theta(\alpha_1)\big)| \geq 8^{-d(n+p)N_1}\delta^{N_1+1}|\mathcal{M}\big(M(\gamma_2)\big)|\\
        &\geq  2^{-(m+1)d(n+p)-1}\cdot 8^{-d(n+p)\nu}\delta^{\nu} |\mathcal{M}\big(R^\psi(\gamma_0)\big)| \\
        &= C\delta^{\nu} |\mathcal{M}\big(R^\psi(\gamma_0)\big)|,
    \end{align*}
    finishing the proof by setting $\sigma = C \delta^\nu$.
\end{proof}

\begin{remark}
    Theorem \ref{theo: doubling porosity} could be modified to a more general form. Instead, we could take $P(\gamma_1) \in \mathcal{D}_k(R(\gamma_0))$ for any $k \in \N$. However, this level of generality is not needed in this paper.
\end{remark}

A natural corollary of Theorem \ref{theo: doubling porosity} is the time-lag invariance of parabolic weak porosity. We show this result next.

\begin{corollary} \label{cor: translation}
    Suppose $E \subseteq \Rn$ is $(c, \delta,\theta)$-weakly porous for some $c,\delta \in (0,1)$ and $\theta >1$. Then, $E$ is $(c, \sigma, \psi)$-weakly porous for any $\psi >1$, where $\sigma = \sigma(n,p,d,c,\delta,\theta,\psi)$ is from Theorem \ref{theo: doubling porosity}.
\end{corollary}

\begin{proof}
    Let $\psi >1$, and choose a parabolic rectangle $R(\gamma_0) \subseteq \Rn$ with $0\leq \gamma_0 \leq 1/2$. Since $E$ is $(c,\delta,\theta)$-weakly porous, by Theorem \ref{theo: doubling porosity} for every $P(\gamma) \in \mathcal{D}_1\big(R(\gamma_0)\big)$ with $0\leq \gamma \leq 1/2$ there exists pairwise disjoint $E$-free $S_i(\alpha_i) \in \mathcal{D}\big(P(\gamma)\big)$ with $0\leq \alpha_i \leq 1/2$ for $i=1,\dots,N$ such that 
    \begin{align*}
        |S_i(\alpha_i)| \geq \sigma |\mathcal{M}\big(R^\psi(\gamma_0)\big)|
    \end{align*}
    for some $\sigma = \sigma(n,p,d,c,\delta,\theta,\psi) \in (0,1)$. Moreover,
    \begin{align*}
        F\big(P(\gamma)\big) = \sum_{i=1}^N|S_i(\alpha_i)| \geq c |P(\gamma)|.
    \end{align*}
    
    Clearly, $\big(S_i(\alpha_i)\big)_{i=1}^N \subseteq \mathcal{D}\big(R(\gamma_0)\big)$ is a finite collection, so taking the union of these collections over every $P(\gamma) \in \mathcal{D}_1(R(\gamma_0))$ results in a finite collection of $\sigma$-admissible dyadic subrectangles of $R(\gamma_0)$. On the other hand, summing over every $P(\gamma) \in \mathcal{D}_1\big(R(\gamma_0)\big)$ yields
    \begin{align*}
        \sum_{P(\gamma) \in \mathcal{D}_1(R(\gamma_0))}F\big(P(\gamma)) &\geq \sum_{P(\gamma) \in \mathcal{D}_1(R(\gamma_0))} c|P(\gamma)| \\
        &= c|R(\gamma_0)|, 
    \end{align*}
    showing that $E$ is $(c,\sigma,\psi)$-weakly porous.
\end{proof}

We have another corollary of Theorem \ref{theo: doubling porosity}, which is the forward-in-time doubling of maximal hole. The doubling property allows us to compare the maximal hole of the dyadic child rectangles to the maximal hole of the forward-in-time parent. For results in Section \ref{sec: stopping conditions}, we require a certain extension of the doubling property to any order forward in time parent and any translation.

\begin{corollary} \label{cor: mh doubling}
    Suppose $E\subseteq\mathbb{R}^{n+1}$ is $(c,\delta,\theta)$-weakly porous for some $c,\delta \in (0,1)$ and $\theta >1$. If $\psi \in [\theta_1-\theta_0,\theta_2-\theta_0]$ for some integer $\theta_0 \geq 2$ and $1<\theta_1\leq \theta_2 <\infty$, then for every pair of parabolic rectangles $R(\gamma_0) \subseteq \Rn$ and $P(\gamma) \in \mathcal{D}_m^{\textnormal{ext}}\big(R(\gamma_0)\big)$  with $\gamma_0,\gamma \in [0,1/2]$ and $m\geq 1$, the maximal hole function satisfies
    \[
    \lvert \mathcal{M}\big(P(\gamma)\big) \rvert \geq \sigma^j \lvert  \mathcal{M}\big(\pi_j^+P^\psi (\gamma)\big) \rvert,
    \]
    for every $j=1,\dots,m$, where $\sigma = \sigma(n,p,d,c,\delta,\theta,\theta_0,\theta_1,\theta_2) \in (0,1)$.
\end{corollary}

\begin{remark}
    It is important later that we allow $\psi \in [\theta_1-\theta_0,\theta_2-\theta_0]$, where the range can include negative values. The doubling property requires nonzero time-lag between initial parent and the translated parent, which Corollary \ref{cor: mh doubling} still preserves since the forward-in-time parent operator includes the translation $\theta_0$, see (\ref{eq: theta0}).
\end{remark}

\begin{proof}
    Suppose $\psi \in [\theta_1-\theta_0,\theta_2-\theta_0]$ for some $\theta_0 = 2,3,\dots$ and $1<\theta_1 \leq \theta_2<\infty$. Let $R(\gamma_0) \subseteq \Rn$ be a parabolic rectangle with with $0\leq \gamma_0 \leq 1/2$ and $P(\gamma) \in \mathcal{D}_m^{\textnormal{ext}}\big(R(\gamma_0)\big)$ with $0\leq \gamma \leq 1/2$ for some $m\geq 1$. Fix $j =1,\dots, m$ and define  $Q(\alpha_{j-1}) = \pi_{j-1}^+P(\gamma)$. Observe that
    \begin{align*}
        \pi_{j}^+ P^\psi(\gamma) &= \pi_{j}^+P(\gamma) + \big(0,\psi l_t\big(\pi_{j}P(\gamma)\big)\big) \\
        &= \pi^+Q(\alpha_{j-1}) + \big(0,\psi l_t\big(\pi Q(\alpha_{j-1})\big)\big) \\
        &= \pi Q^{\theta_0 +\psi}(\alpha_{j-1}) =  \pi Q^{\psi_1}(\alpha_{j-1}),
    \end{align*}
    where $\psi_1 = \theta_0 +\psi$. Since $\psi \in [\theta_1-\theta_0,\theta_2-\theta_0]$, then $\psi_1 \in [\theta_1,\theta_2]$. 
    
    We then use the fact that $E$ is $(c,\delta,\theta)$-weakly porous to apply Theorem \ref{theo: doubling porosity}. Now, for $\psi_1 \in [\theta_1,\theta_2]$ there exists some $\sigma_1 = \sigma_1(n,p,c,\delta,\theta,\theta_1,\theta_2) \in (0,1)$ and at least one $E$-free $S(\beta_{j-1}) \in \mathcal{D}\big(Q(\alpha_{j-1})\big) \subseteq \mathcal{D}\big(\pi Q(\alpha_{j-1})\big)$ with $0\leq \beta_{j-1} \leq 1/2$ such that
    \begin{align*}
        |S(\beta_{j-1})| \geq \sigma_1|\mathcal{M}\big(\pi Q^{\psi_1}(\alpha_{j-1})\big)|.
    \end{align*}
    It follows that
    \begin{align}
        \begin{split}
            |\mathcal{M}\big(\pi_{j-1}^+P(\gamma)\big)| &= |\mathcal{M}\big(Q(\alpha_{j-1})\big)| \geq |S(\beta_{j-1})| \\
            &\geq \sigma_1 |\mathcal{M}\big(\pi Q^{\psi_1}(\alpha_{j-1})\big)| \\
            &= \sigma_1 |\mathcal{M}\big(\pi_j^+P^{\psi}(\gamma)\big)|.
        \end{split}\label{eq: mh doubling 1}
    \end{align}

    On the other hand, for any $i =0,\dots, j-2$ we define $U_i(\alpha_i) = \pi_{i}^+P(\gamma)$ with $0\leq \alpha_i\leq 1/2$. Hence we have
    \begin{align*}
        \pi_{i+1}^+ P(\gamma) &= \pi\big(\pi_{i}^+P(\gamma)\big) +\big(0,\theta_0 l_t\big(\pi_{i+1}P(\gamma)\big)\big) \\
        &= \pi U_i(\alpha_i) +\big(0,\theta_0 l_t\big(\pi U_i(\alpha_i)\big)\big) =\pi U_i^{\theta_0}(\alpha_i).
    \end{align*}
    We apply Theorem \ref{theo: doubling porosity} again for each $i = 0,\dots j-2$ to find a pairwise disjoint $E$-free $V(\beta_i) \in \mathcal{D}\big(U_i(\alpha_i)\big) \subseteq \mathcal{D}\big(\pi U_i(\alpha_i)\big)$ with $0\leq \beta_i \leq 1/2$ such that
    \begin{align*}
        |V(\beta_i)| \geq \sigma_2|\mathcal{M}\big(\pi U_i^{\theta_0}(\alpha_i)\big)|,
    \end{align*}
    where $\sigma_2 = \delta_2(n,p,c,\delta,\theta,\theta_0) \in (0,1)$. It follows that
    \begin{align*}
        |\mathcal{M}\big(\pi_i^+P(\gamma)\big)|
        &= |\mathcal{M}\big(U_i(\alpha_i)\big)| \geq |V(\beta_i)| \geq \sigma_2 |\mathcal{M}\big(\pi U_i^{\theta_0}(\alpha_i)\big)| \\
        &=\sigma_2 |\mathcal{M}\big(\pi^+ U_i (\alpha_i)\big)| = \sigma_2 |\mathcal{M}\big(\pi_{i+1}^+P(\gamma)\big)|.
    \end{align*}
    We iterate the estimate above and combine it with (\ref{eq: mh doubling 1}) to obtain
    \begin{align*}
        |\mathcal{M}\big(P(\gamma)\big)| &= |\mathcal{M}\big(\pi_0^+P(\gamma)\big)| \geq \sigma_2^{j-1}|\mathcal{M}\big(\pi_{j-1}^+P(\gamma)\big)| \\
        &\geq  \sigma_2^{j-1}\sigma_1 |\mathcal{M}\big(\pi_j^+ P^{\psi}(\gamma)\big)|.
    \end{align*}
    We finish the proof by setting $\sigma = \min \big\{\sigma_1,\sigma_2\big \}$, which results in the desired
    \begin{align*}
        |\mathcal{M}\big(P(\gamma)\big)| \geq \sigma^{j} |\mathcal{M}\big(\pi_j^+ P^{\psi}(\gamma)\big)|.
    \end{align*}
\end{proof}

\section{Stopping time construction} \label{sec: stopping conditions}

To prove the $\alpha$-improvement for any parabolic weakly porous set, we need various intermediate results. Our planned approach is to formulate an appropriate stopping time construction using the forward-in-time parent operator $\pi^+$. The main lemma of this section is a certain exponential estimate, however, the other results are also important later.

\subsection{Stopping time of forward-in-time parents}
Given a nonempty closed set $E \subseteq \Rn$, a parabolic rectangle $R(\gamma_0) \subseteq \Rn$ with $0\leq \gamma_0 \leq 1/2$ and an integer $m\geq 1$, then for any subrectangle $P(\gamma) \in \mathcal{D}_m^{\text{ext}}(R(\gamma_0))$ with $0\leq \gamma \leq 1/2$ we define a stopping time
\begin{align}
    \tau_\Lambda\big(P(\gamma)\big) = \min\bigg\{k =1,\dots,m \,:\, \max_{\theta\in [\phi-\theta_0,\Phi-\theta_0]} |\mathcal{M}\big(\pi_k^+P^\theta(\gamma)\big)| \geq \Lambda \bigg\}, \label{eq: termination number}
\end{align}
where $\theta_0,\phi,\Phi \in \N$ are fixed parameters such that $2 \leq \phi \leq \theta_0 \leq \Phi <\infty$ and $\Lambda >0$. Recall also that the integer $\theta_0 \geq 2$ is the default translation for the forward-in-time parent operator $\pi^+$, see (\ref{eq: theta0}). The stopping time $\tau_\Lambda$ measures how many times we have to apply our doubling results, Theorem \ref{theo: doubling porosity} and Corollary \ref{cor: mh doubling}, on a chain $\big(\pi_i^+P(\gamma)\big)_i$, such that the forward-in-time parent of $P(\gamma)$ finds a large enough $E$-free region.

Unfortunately there is no obvious or natural choice for the parameters $\theta_0,\phi$ and $\Phi$. Essentially, the parameters are required to satisfy the following five specific conditions that depend on the product $dp$. The restrictions arise from the proofs of this section, and are heavily connected to Lemma \ref{lem: bounded distance}. In this section and Section \ref{sec: full charac} we shall consider the parameters $\theta_0,\phi$ and $\Phi$ fixed such that they satisfy the following.

\begin{lemma} \label{lem: parameters}
    Let $p\geq 1$ and define integer $d = d(p)$ such that it satisfies (\ref{eq: division rate}). Then, there exist integers  $\theta_0,\phi,\Phi \geq 2$ satisfying the following:
    \begin{enumerate}[label=(\roman*)]
        \item \(\phi \leq \Phi - \theta_0 - \big\lceil\frac{2\theta_0}{2^{dp}-1}\big\rceil\). \label{item: parameters 1}
        \item \(\phi \leq \theta_0 \leq \Phi\). \label{item: parameters 2}
        \item \(\Phi\geq  \big\lceil2^{dp}\big\rceil -1 \). \label{item: parameters 3}
        \item \(\phi\leq  \big\lfloor2^{dp}\big\rfloor -1 - \theta_0 - \big\lceil\frac{2\theta_0}{2^{dp}-1}\big\rceil\).\label{item: parameters 4}
        \item \(\phi - \theta_0 \leq - \big\lceil\frac{4\theta_0}{2^{dp}-1}\big\rceil\).\label{item: parameters 5}
    \end{enumerate}
\end{lemma}
\begin{proof}
    Let $p\geq 1$. Choose $\theta_0 = 4$, $\phi = 2$ and $\Phi = \big\lceil2^{dp}\big\rceil -1$. These values satisfy the claim, when $d = d(p)$ is chosen by (\ref{eq: division rate}).
\end{proof}

Before we show some structure of the stopping time $\tau_\Lambda$ defined by (\ref{eq: termination number}), we first show that under certain circumstances, $\tau_\Lambda$ is well defined. The next result also motivates to limit our analysis to $(c,\delta,\Phi)$-weakly porous sets, which is apparent in the later results of this section. However, we will later show that fixing the translation as $\Phi$ is not an obstacle.

\begin{lemma} \label{lem: finite termination}
     Suppose $E \subseteq \Rn$ is a nonempty closed set, $R(\gamma_0)\subseteq \Rn$ is a parabolic rectangle with $0\leq \gamma_0 \leq 1/2$ and $0<\Lambda \leq |\mathcal{M}\big(R^{\Phi}(\gamma_0)\big)|$. Then, for any $m \geq 1$ and any $P(\gamma) \in \mathcal{D}_m\big(R(\gamma_0)\big)$ with $0\leq \gamma \leq 1/2$ the stopping time $\tau_\Lambda$ is well defined. In particular, $\tau_\Lambda\big(P(\gamma)\big) \leq m$.
\end{lemma}
\begin{proof}
     Let $m \geq 1$ and suppose $P(\gamma) \in \mathcal{D}_m\big(R(\gamma_0)\big)$ with $0 \leq \gamma \leq 1/2$. Observe that $\pi_m P(\gamma) = R(\gamma_0)$ and $\pi_{m}^+P(\gamma) \in \mathcal{D}_0^{\textnormal{ext}}\big(R(\gamma_0)\big)$, which means that we can write 
   \begin{align}
       \pi_{m}^+P(\gamma) = R^{\theta_1}(\gamma_0)\label{eq: finite termination parent}
   \end{align}
   for some integer $\theta_1 \geq \theta_0$. To bound $\theta_1$ from above, we estimate the distance between $\partial_{\textnormal{low}}R(\gamma_0)$ and $\partial_{\textnormal{low}}R^{\theta_1}(\gamma_0)$. By Lemma \ref{lem: bounded distance}, we have
    \begin{align*}
        \dist_t\big(\partial_{\textnormal{low}}R(\gamma_0),\partial_{\textnormal{low}}R^{\theta_1}(\gamma_0)\big) &\leq \dist_t\big(\partial_{\textnormal{low}}P(\gamma),\partial_{\textnormal{low}}R^{\theta_1}(\gamma_0)\big) + l_t\big(R(\gamma_0)\big)\\
        &= \dist_t\big(\partial_{\textnormal{low}}P(\gamma),\partial_{\textnormal{low}}\pi_{m}^+P(\gamma)\big) + l_t\big(R(\gamma_0)\big)\\
        &< 2\theta_0\frac{2^{dp}}{2^{dp}-1}l_t\big(\pi_{m}^+P(\gamma)\big) +l_t\big(R(\gamma_0)\big)\\
        &\leq \bigg(2\theta_0 +\bigg\lceil\frac{2\theta_0}{2^{dp}-1}\bigg\rceil +1\bigg) l_t\big(R(\gamma_0)\big).
    \end{align*}
    Since $\theta_1$ is an integer, the strict inequality above tells us that necessarily
    \begin{align*}
        \theta_0\leq \theta_1 \leq 2\theta_0 +\bigg\lceil\frac{2\theta_0}{2^{dp}-1}\bigg\rceil.
    \end{align*}
    
    To finish the proof, notice that if $\theta = \Phi -\theta_1$, then by (\ref{eq: finite termination parent}) we have
    \begin{align}
        \pi_{m}^+P^\theta(\gamma) =R^{\theta_1+\theta}(\gamma_0) = R^{\Phi}(\gamma_0), \label{eq: finite termination translated parent}
    \end{align}
    where integer $\Phi\geq 2$ is from Lemma \ref{lem: parameters}. Observe that using the upper and lower bounds of  $\theta_1$, the parameter $\theta$ is restricted on an interval
    \begin{align*}
        \Phi- 2\theta_0 -\bigg\lceil\frac{2\theta_0}{2^{dp}-1}\bigg\rceil \leq \theta \leq \Phi - \theta_0.
    \end{align*}
    Lemma \ref{lem: parameters}\ref{item: parameters 1} further implies that 
    \begin{align*}
       \phi - \theta_0 \leq \theta \leq \Phi - \theta_0
    \end{align*}
    for an integer $\phi\geq 2$. Furthermore, from (\ref{eq: finite termination translated parent}) it clearly follows that
    \begin{align*}
        |\mathcal{M}\big(\pi_{m}^+P^\theta(\gamma)\big)| = |\mathcal{M}\big(R^{\Phi}(\gamma_0)\big)| \geq \Lambda.
    \end{align*}
    Thus, $\tau_\Lambda\big(P(\gamma)\big) \leq m $ by definition (\ref{eq: termination number}).
\end{proof}

\subsection{Stopping time chains}
We consider any dyadic subcollection $\mathcal{A} \subseteq \mathcal{D}\big(R(\gamma_0)\big)$. Now, each $P(\gamma) \in \mathcal{A}$ serves as a base or a generator for a chain $\big(\pi_i^+P(\gamma)\big)_i$. These chains are easier to work with when sorted by their respective stopping time $\tau_\Lambda$. Hence, for any $k=1,2,\dots$ we define the collections
\begin{align}
    \mathcal{S}_k(\mathcal{A},\Lambda) = \bigcup_{P(\gamma) \in \mathcal{A}} \Big\{\pi_i^+P(\gamma) \,:\, \tau_\Lambda\big(\pi_i^+P(\gamma)\big) = k, \; i =0,1,\dots, \tau_\Lambda \big(P(\gamma)\big)-1 \Big\}.\label{eq: termination collections}
\end{align}
Notice that if $\tau_\Lambda\big(\pi_i^+P(\gamma)\big)$ is not defined, then simply $\pi_i^+P(\gamma) \notin \mathcal{S}_k(\mathcal{A},\Lambda)$.

The geometric process behind $\tau_\Lambda$ is defined in such a way that the search range for $\tau_\Lambda\big(\pi^+ P(\gamma)\big)$ is contained in the search region for $\tau_\Lambda\big(P(\gamma)\big)$. This means that it takes less steps to reach an $E$-free region from $\pi^+P(\gamma)$ than from $P(\gamma)$. We start by formulating exactly how the forward-in-time parent operator controls the stopping time $\tau_\Lambda$. We define the forward-in-time parents of the collection $\mathcal{S}_k(\mathcal{A},\Lambda)$ by
\begin{align*}
    \pi^+ \mathcal{S}_{k}(\mathcal{A},\Lambda) = \big\{\pi^+ P(\gamma) \,:\, P(\gamma) \in \mathcal{S}_{k}(\mathcal{A},\Lambda) \big\}
\end{align*}
for any $k=1,2,\dots$, should the forward-in-time parent be well-defined.

\begin{lemma} \label{lem: parent sets}
    Let $R(\gamma_0) \subseteq \Rn$ be a parabolic rectangle with $0\leq \gamma_0 \leq 1/2$ and $0<\Lambda \leq |\mathcal{M}\big(R^{\Phi}(\gamma_0)\big)|$. Then, 
    \begin{align*}
        \pi^+ \mathcal{S}_{k+1}(\mathcal{A},\Lambda) \subseteq \mathcal{S}_{k}(\mathcal{A},\Lambda)
    \end{align*}
    for every $k=1,2,\dots$ and for any subcollection $\mathcal{A} \subseteq \mathcal{D}\big(R(\gamma_0)\big)$.
\end{lemma}

\begin{proof}
    We first claim that for any fixed integer $m\geq 1$ and $P(\gamma) \in \mathcal{D}_m\big(R(\gamma_0)\big)$ with $0 \leq \gamma \leq 1/2$ such that $j_P  = \tau_\Lambda\big(P(\gamma)\big)\geq 1$ we have
    \begin{align}
        \tau_\Lambda\big(\pi_i^+P(\gamma)\big) = j_P -i \label{eq: parent sets index}
    \end{align}
    for any $i=0,\dots, j_P-1$. Observe that by Lemma \ref{lem: finite termination} $j_P\leq m$ is well defined since $\Lambda \leq |\mathcal{M}\big(R^{\Phi}(\gamma_0)\big)|$, and hence the forward-in-time parent $\pi_i^+$ is also defined.
    
    To prove (\ref{eq: parent sets index}), fix $i = 0,\dots,j_P -1$. Working directly with the definition (\ref{eq: termination number}), and since $\pi_i^+P(\gamma) \in \mathcal{D}_{m-i}^{\textnormal{ext}}\big(R(\gamma_0)\big)$, we get
    \begin{align*}
        \tau_\Lambda \big(\pi_i^+ P(\gamma)\big) &= \min\bigg\{k =1,\dots,m-i \,:\, \max_{\theta\in [\phi-\theta_0,\Phi-\theta_0]}|\mathcal{M}\big(\pi_k^+\big(\pi_i^+P(\gamma)\big)^\theta \big)| \geq \Lambda \bigg\} \\
        &= \min\bigg\{k =1,\dots,m-i \,:\, \max_{\theta\in [\phi-\theta_0,\Phi-\theta_0]}|\mathcal{M}\big(\pi_{k+i}^+P^\theta(\gamma)\big)| \geq \Lambda \bigg\} \\
        &= \min\bigg\{j = 1+i,\dots,m \,:\, \max_{\theta\in [\phi-\theta_0,\Phi-\theta_0]}|\mathcal{M}\big(\pi_{j}^+P^\theta(\gamma)\big)| \geq \Lambda \bigg\} -i.
    \end{align*}
    Since $i \leq j_P -1$, we have an upper bound
    \begin{align*}
         \tau_\Lambda \big(\pi_i^+ P(\gamma)\big) \leq \min\bigg\{j = j_P,\dots,m \,:\, \max_{\theta\in [\phi-\theta_0,\Phi-\theta_0]}|\mathcal{M}\big(\pi_j^+P^\theta(\gamma)\big)| \geq \Lambda \bigg\} -i = j_P -i.
    \end{align*}
    On the other hand, since $i \geq 0$, we get a lower bound
    \begin{align*}
        \tau_\Lambda \big(\pi_i^+ P(\gamma)\big) \geq \min\bigg\{j = 1,\dots,m \,:\, \max_{\theta\in [\phi-\theta_0,\Phi-\theta_0]}|\mathcal{M}\big(\pi_j^+P^\theta(\gamma)\big)| \geq \Lambda \bigg\} -i = j_P -i,
    \end{align*}
    proving the claim.

    Consider then for any $\mathcal{A} \subseteq \mathcal{D}\big(R(\gamma_0)\big)$ and every $k=1,2,\dots$ the collections $\mathcal{S}_k = \mathcal{S}_k(\mathcal{A},\Lambda)$. We take any $Q(\alpha) \in \mathcal{S}_{k+1}$ with $0 \leq \alpha \leq 1/2$ for some fixed $k = 1,2,\dots$. By definition (\ref{eq: termination collections}), we can write $Q(\alpha)$ as a forward-in-time parent of a base rectangle of some chain. Namely, we have
    \begin{align*}
        Q(\alpha) = \pi_{i}^+Q_0(\alpha_0)
    \end{align*}
    for some $Q_0(\alpha_0) \in \mathcal{A}$ with $0 \leq \alpha_0 \leq 1/2$ and $i = 0,\dots, j_{Q_0}-1$, where $j_{Q_0} = \tau_\Lambda\big(Q_0(\alpha_0)\big) \geq 1$. Furthermore, the index $k+1$ of the collection $\mathcal{S}_{k+1}$ sets the stopping time of $Q(\alpha)$ as
    \begin{align*}
        \tau_\Lambda\big(Q(\alpha)\big) =\tau_\Lambda\big(\pi_{i}^+Q_0(\alpha_0)\big) = k+1.
    \end{align*}
    
    We may assume $Q_0(\alpha_0) \neq R(\gamma_0)$, since $\tau\big(\pi_j^+R(\gamma_0)\big)$ is not defined for any $j\in \N$. Thus, $Q_0(\alpha_0) \in \mathcal{D}_{m}\big(R(\gamma_0)\big)$ for some $m \geq 1$. By (\ref{eq: parent sets index}) and the above we have
    \begin{align*}
        k+1  =  j_{Q_0} - i.
    \end{align*}
    Reorganizing the terms shows that the next index $i+1$ is still bounded by
    \begin{align*}
        i +1 = j_{Q_0} - k \leq  j_{Q_0} - 1.
    \end{align*}
    This means that we can apply (\ref{eq: parent sets index}) to $\pi_{i+1}^+Q_0(\alpha_0)$ to obtain
    \begin{align*}
        \tau_\Lambda\big(\pi^+Q(\alpha)\big) =\tau_\Lambda\big(\pi_{i+1}^+Q_0(\alpha_0)\big) =  j_{Q_0} -(i+1) = k.
    \end{align*}
    proving $\pi^+ Q(\alpha) \in \mathcal{S}_k$, and consequently $\pi^+ \mathcal{S}_{k+1} \subseteq \mathcal{S}_k$.
\end{proof}

For the following results we assume that $E \subseteq \Rn$ is $(c,\delta,\Phi)$-weakly porous set. We use the translation $\Phi$ to ensure that the stopping time is finite for every $P(\gamma) \in \mathcal{D}\big(R(\gamma_0)\big)$. On the other hand, the doubling features of parabolic weakly porous sets of Section \ref{sec: doubling results} are useful.

The next result is rather technical but a necessary part of proving the $\alpha$-improvement of parabolic weakly porous sets, see Lemma \ref{lem: wp => polynomial decay}. If $P(\gamma) \in \mathcal{D}\big(R(\gamma_0)\big)$, it is possible that the chain of rectangles $\big(\pi_i^+P(\gamma)\big)_i$ escapes outside of $R(\gamma_0)$. In the proof of the $\alpha$-improvement, we have to show that the union of the rectangles outside of $R(\gamma_0)$ is comparable to the union of those inside. The following result is a necessary prerequisite.

\begin{lemma} \label{lem: decay of interim space}
    Suppose $E\subseteq \Rn$ is $(c_0,\delta_0,\Phi)$-weakly porous for some $c_0,\delta_0\in (0,1)$ and take any parabolic rectangle $R(\gamma_0) \subseteq \Rn$ with $0\leq \gamma_0 \leq 1/2$ and a subcollection $\mathcal{A} \subseteq \mathcal{D}\big(R(\gamma_0)\big)$. Then, the collections $\mathcal{S}_k\big(\mathcal{A},\Lambda)$ with $\Lambda = \delta|\mathcal{M}\big(R^{\Phi}(\gamma_0)\big)|$ for any $0<\delta \leq 1$ satisfy
        \begin{align*}
            \bigcup_{k=1}^{\infty}\bigcup_{P(\gamma)\in \mathcal{S}_k(\mathcal{A},\Lambda)}P(\gamma) \subseteq \bigcup_{\theta \in [0,\psi]} R^{\theta}(\gamma_0),
        \end{align*}
        for some $\psi \leq C\delta^\mu$ where $\mu = \mu(n,p,d,c_0,\delta_0)>0$ and $C = C(p,d)>0$.
\end{lemma}

\begin{proof}
    For some fixed parabolic rectangle $R(\gamma_0) \subseteq \Rn$ with $0\leq \gamma_0 \leq 1/2$ and $0<\delta\leq 1$, let $\Lambda = \delta|\mathcal{M}\big(R^{\Phi}(\gamma_0)\big)|$. We observe first that (\ref{eq: termination collections}) implies
    \begin{align}
        \bigcup_{k=1}^\infty \bigcup_{P(\gamma)\in \mathcal{S}_k(\mathcal{A},\Lambda)}P(\gamma) = \bigcup_{P(\gamma) \in \mathcal{A}}\Big\{\pi_j^+P(\gamma)\, :\, j  = 0,1,\dots,\tau_\Lambda\big(P(\gamma)\big)-1 \Big\}. \label{eq: decay of interim space unionset}
    \end{align}
    In other words, it is enough to study the chains generated by the base rectangles in $\mathcal{A}$ up to the stopping time $\tau_\Lambda$. Notice that since $\tau_\Lambda\big(R(\gamma_0)\big)$ is not defined, we may assume $R(\gamma_0) \notin \mathcal{A}$.
    
    We start by fixing an integer $m\geq 1$ and considering any base rectangle $P(\gamma) \in \mathcal{A} \cap \mathcal{D}_m\big(R(\gamma_0)\big)$ with $0\leq \gamma \leq 1/2$. We shall write $j_P  = \tau_\Lambda\big(P(\gamma)\big) \geq 1$ for simplicity. Thus, the object of our interest is the chain $\big(\pi_j^+P(\gamma)\big)_{j=0}^{j_P-1}$. By Lemma \ref{lem: finite termination}, we have $j_P  \leq m$. To be more precise, (\ref{eq: finite termination translated parent}) states that
    \begin{align*}
        \pi_{m}^+P^\theta(\gamma) = R^{\Phi}(\gamma_0)
    \end{align*}
    for some $\theta \in [\phi-\theta_0,\Phi-\theta_0]$, which clearly implies
    \begin{align}
        |\mathcal{M}\big(\pi_{m}^+P^\theta(\gamma)\big)| = |\mathcal{M}\big(R^{\Phi}(\gamma_0)\big)|. \label{eq: decay of interim space m index}
    \end{align}
    
    We shall next use the parabolic weak porosity to bring information from $R^\Phi(\gamma_0)$ closer to $P(\gamma)$. We do this by applying the doubling of the maximal hole function along the chain of forward-in-time parents of $P(\gamma)$. Since $E$ is $(c_0,\delta_0,\Phi)$-weakly porous, by Corollary \ref{cor: mh doubling} the maximal hole function is doubling, that is, there exists some $\sigma = \sigma(n,p,d,c_0,\delta_0, \theta_0, \phi, \Phi) \in (0,1)$ such that
    \begin{align*}
        |\mathcal{M}\big(\pi_{m-i}^+P(\gamma)\big)| \geq \sigma^i|\mathcal{M}\big(\pi_{i}^+\big(\pi_{m-i}^+P(\gamma)\big)^\theta\big)|= \sigma^i|\mathcal{M}\big(\pi_{m}^+P^\theta(\gamma)\big)|
    \end{align*}
    for every index $i = 1,\dots,m$. If we define 
    \begin{align*}
         \Tilde{\theta}_i = \begin{cases}
            \theta, \quad i = 0, \\
            0, \quad i = 1,\dots,m,
        \end{cases}
    \end{align*}
    then we can extend the inequality above to the case $i=0$, that is, 
    \begin{align}
        |\mathcal{M}\big(\pi_{m-i}^+P^{\Tilde{\theta}_i}(\gamma)\big)| \geq \sigma^i|\mathcal{M}\big(\pi_{m}^+P^\theta(\gamma)\big)| \label{eq: decay of interim space doubling}
    \end{align}
    for every $i = 0,\dots,m$.

    We will then choose $0 \leq i \leq m-1$ to be the largest index such that $\sigma^i \geq \delta$. Such an index indeed exists, since $\sigma^0 = 1 \geq \delta$. Therefore, we set $i$ as
    \begin{align*}
        i = \min\bigg\{m-1, \bigg\lfloor \frac{\ln \delta}{\ln\sigma}\bigg\rfloor\bigg\}.
    \end{align*}
    The doubling inequality (\ref{eq: decay of interim space doubling}), the choice of $i$ and (\ref{eq: decay of interim space m index}) imply that
    \begin{align*}
        |\mathcal{M}\big(\pi_{m-i}^+P^{\Tilde{\theta}_i}(\gamma)\big)| \geq \sigma^i |\mathcal{M}\big(\pi_{m}^+P^\theta(\gamma)\big)| \geq \delta |\mathcal{M}\big(R^{\Phi}(\gamma_0)\big)| = \Lambda.
    \end{align*}
    By Lemma \ref{lem: parameters}\ref{item: parameters 2} the search range parameters are such that $\Tilde{\theta}_i \in [\phi - \theta_0,\Phi-\theta_0]$ for any $i$. On the other hand, $m-i\geq 1$, thus the definition (\ref{eq: termination number}) states that 
    \begin{align}
        j_P = \tau_\Lambda\big(P(\gamma)\big)\leq m-i. \label{eq: decay of interim space index bound}
    \end{align}
    In other words, the chain $\big(\pi_j^+P(\gamma)\big)_{j=0}^{j_P-1}$ ends latest at $\pi_{m-i-1}^+P(\gamma)$. 

    It should be clear that by choosing the translation parameter $\psi_P\geq 0$ large enough, the chain $\big(\pi_j^+P(\gamma)\big)_{j=0}^{j_P-1}$ is contained in the convex hull of $R(\gamma_0)$ and $R^{\psi_P}(\gamma_0)$, that is,
    \begin{align}
        \pi_{j}^+P(\gamma) \subseteq \bigcup_{\theta \in [0,\psi_P]} R^{\theta}(\gamma_0) \label{eq: decay of interim space convex hull}
    \end{align}
    for every $j = 0,\dots, j_{P}-1$. This means that $\psi_P$ has to be so large that
    \begin{align*}
        l_t\Bigg(\bigcup_{\theta \in [0,\psi_P]} R^{\theta}(\gamma_0)\Bigg) \geq \dist_t\big(\partial_{\textnormal{low}}R(\gamma_0),\partial_{\textnormal{low}}\pi_{j_P-1}^+P(\gamma)\big) +  l_t\big(\pi_{j_P-1}^+P(\gamma)\big).
    \end{align*}
    Let us study two cases. In the first case we assume $j_P \geq 2$. Now, a sufficient condition for the above is
    \begin{align*}
         \psi_P l_t\big(R(\gamma_0)\big) + l_t\big(R(\gamma_0)\big) = \dist_t\big(R(\gamma_0),\partial_{\textnormal{low}}\pi_{j_P-1}^+P(\gamma)\big) +  l_t\big(R(\gamma_0)\big) + l_t\big(\pi_{j_P-1}^+P(\gamma)\big),
    \end{align*}
    that is, 
    \begin{align}
        \psi_P = \Big(\dist_t\big(R(\gamma_0),\partial_{\textnormal{low}}\pi_{j_P-1}^+P(\gamma)\big) +  l_t\big(\pi_{j_P-1}^+P(\gamma)\big)\Big) l_t\big(R(\gamma_0)\big)^{-1}. \label{eq: decay of interim space psi}
    \end{align}
    
    We recall that $i$ equals $m-1$ or $\big\lfloor \ln \delta (\ln \sigma)^{-1}\big\rfloor$. If $i = m-1$, then $j_P =1$ by (\ref{eq: decay of interim space index bound}), which is a contradiction. Thus, we may assume
    \begin{align*}
        i =  \bigg\lfloor \frac{\ln \delta}{\ln \sigma}\bigg\rfloor.
    \end{align*}
    The temporal side length of $\pi_{j_P-1}^+P(\gamma)$ can be now bounded using (\ref{eq: decay of interim space index bound}) and Corollary \ref{cor: dyadic scale} as
    \begin{align}
        \begin{split}
        l_t\big(\pi_{j_P-1}^+P(\gamma)\big) &\leq l_t\big(\pi_{m-i-1}^+P(\gamma)\big) = l_t\big(\pi_{m-i-1} P(\gamma)\big) \\
        &\leq 2\cdot 2^{-(i+1)dp}l_t\big(R(\gamma_0)\big) \\
        &\leq 2 \delta^{-dp\frac{\ln2}{\ln\sigma}}l_t\big(R(\gamma_0)\big).
        \end{split} \label{eq: decay of interim space length estimate}
    \end{align}
    The temporal side length of $\pi_{j_P-1}^+P(\gamma)$ is hence comparable to the temporal side length of $R(\gamma_0)$ with $\delta$ as a factor. We then bound the distance between $R(\gamma_0)$ and $\partial_{\textnormal{low}}\pi_{j_P-1}^+P(\gamma)$ with Lemma \ref{lem: bounded distance} and (\ref{eq: decay of interim space length estimate}), which yields
    \begin{align*}
         \dist_t\big(R(\gamma_0),\partial_{\textnormal{low}}\pi_{j_P-1}^+P(\gamma)\big) &\leq \dist_t\big(\partial_{\textnormal{low}}P(\gamma),\partial_{\textnormal{low}}\pi_{j_P-1}^+P(\gamma)\big)\\
         &< 2 \theta_0\frac{2^{dp}}{2^{dp}-1} l_t\big(\pi_{j_P-1}^+P(\gamma)\big) \\
        &\leq 4 \theta_0\frac{2^{dp}}{2^{dp}-1}\cdot \delta^{-dp\frac{\ln2}{\ln\sigma}}l_t\big(R(\gamma_0)\big).
    \end{align*}
    Substituting the above and (\ref{eq: decay of interim space length estimate}) into (\ref{eq: decay of interim space psi}) results in
    \begin{align*}
        \psi_P &\leq  \Big(4\theta_0\frac{2^{dp}}{2^{dp}-1} +2\Big) \delta^{-dp\frac{\ln2}{\ln\sigma}} = C \delta^\mu,
    \end{align*}
    where $C = C(p,d,\theta_0) >0$ and $\mu = \mu(n,p,d,c_0,\delta_0,\theta_0,\phi,\Phi) >0$. Since $\theta_0,\phi$ and $\Phi$ are fixed, we may remove them from the dependencies. We can also remove the dependency of $P(\gamma)$ from $\psi_P$ by defining
    \begin{align*}
        0\leq \psi = \sup_{P(\gamma)\in \mathcal{A}}\psi_P \leq C \delta^\mu.
    \end{align*}
    
    The case $j_P = 1$ is easy. Now the chain $\big(\pi_j^+P(\gamma)\big)_{j=0}^{j_P-1}$ only consists of the base rectangle $\pi_0^+P(\gamma) = P(\gamma)$. Clearly, then 
    \begin{align*}
        \pi_0^+P(\gamma) \subseteq R(\gamma_0) = R^0(\gamma_0) \subseteq \bigcup_{\theta \in [0,\psi]} R^{\theta}(\gamma_0),
    \end{align*}
    proving (\ref{eq: decay of interim space convex hull}) in both cases. Taking the union over every base rectangle $P(\gamma) \in \mathcal{A}$ yields 
    \begin{align*}
        \bigcup_{P(\gamma) \in \mathcal{A}}\Big\{\pi_j^+P(\gamma)\, :\, j =0,1,\dots,\tau_\Lambda\big(P(\gamma)\big)-1 \Big\} \subseteq \bigcup_{\theta \in [0,\psi]} R^{\theta}(\gamma_0),
    \end{align*}
    which by (\ref{eq: decay of interim space unionset}) finishes the proof.
\end{proof}

\subsection{Choosing the base rectangles}
Choosing the collection $\mathcal{A} \subseteq \mathcal{D}\big(R(\gamma_0)\big)$ correctly for $\mathcal{S}_k(\mathcal{A},\Lambda)$ is crucial. The maximal complementary collection of $\mathcal{F}_\delta^\Phi\big(R(\gamma_0)\big)$ is a natural choice. We define the collection of the maximal complementary rectangles by
\begin{align*}
    \mathcal{G}_{\delta}^\theta\big(R(\gamma_0)\big) = \big\{P(\gamma)\in \mathcal{D}\big(R(\gamma_0)\big) : P(\gamma) \cap F_\delta^\theta = \emptyset, \; \pi P(\gamma) \cap F_\delta^\theta \neq \emptyset  \big\},
\end{align*}
for every $0<\delta<1$ and $\theta \in \R$, where 
\begin{align*}
F_\delta^\theta = \bigcup_{Q(\alpha)\in \mathcal{F}_\delta^\theta(R(\gamma_0))} Q(\alpha).
\end{align*}

We want to show that the collections $\mathcal{S}_k \big(\mathcal{G}_\delta^{\Phi}\big(R(\gamma_0)\big),\Lambda\big)$ contain the whole $\mathcal{G}_\delta^\Phi\big(R(\gamma_0)\big)$ with some appropriate $\Lambda >0$. On the other hand, rectangles in $\mathcal{S}_k\big(\mathcal{G}_\delta^{\Phi}\big(R(\gamma_0)\big),\Lambda\big)$ should not intersect any rectangle in $\mathcal{F}_\delta^\Phi\big(R(\gamma_0)\big)$. We prove this next.

\begin{lemma} \label{lem: G and F extra} 
    Suppose $E\subseteq \Rn$ is $(c,\delta,\Phi)$-weakly porous for some $c,\delta \in (0,1)$. For any parabolic rectangle $R(\gamma_0) \subseteq \Rn$ with $0 \leq \gamma_0 \leq 1/2$ consider the collections $\mathcal{F}_\delta^{\Phi}\big(R(\gamma_0)\big)$, $\mathcal{G}_\delta^{\Phi}\big(R(\gamma_0)\big)$ and $\mathcal{S}_k = \mathcal{S}_k\big(\mathcal{G}_\delta^{\Phi}\big(R(\gamma_0)\big),\Lambda\big)$ with $\Lambda = \delta|\mathcal{M}\big(R^\Phi(\gamma_0)\big)|$. Then, the following statements are true:
    \begin{enumerate}[label=(\roman*)]
        \item Every $P(\gamma) \in \mathcal{G}_\delta^{\Phi}\big(R(\gamma_0)\big)$ with $0\leq \gamma \leq 1/2$ is contained in $\mathcal{S}_k$ for some $k \in \N$, that is, 
        \begin{align*}
            \mathcal{G}_\delta^{\Phi}\big(R(\gamma_0)\big) \subseteq \bigcup_{k=1}^\infty \mathcal{S}_k.
        \end{align*} \label{item: G and F extra 1}
        
        \item If $P(\gamma) \in \mathcal{S}_k$ with $0\leq \gamma \leq 1/2$ for any $k=1,2,\dots$, then $P(\gamma)$ does not intersect any $Q(\alpha) \in \mathcal{F}_\delta^{\Phi}\big(R(\gamma_0)\big)$ with $0 \leq \alpha \leq 1/2$, that is, $P(\gamma) \cap Q(\alpha) = \emptyset.$ \label{item: G and F extra 2}

    \end{enumerate}
\end{lemma}

\begin{proof}
     \ref{item: G and F extra 1} Let $R(\gamma_0) \subseteq \Rn$ be a parabolic rectangle with $0 \leq \gamma_0\leq 1/2$. Consider the collections $\mathcal{F}_\delta^\Phi = \mathcal{F}_\delta^\Phi\big(R(\gamma_0)\big)$ and $\mathcal{G}_\delta^\Phi = \mathcal{G}_\delta^\Phi\big(R(\gamma_0)\big)$, given that $E$ is $(c,\delta,\Phi)$-weakly porous. We may assume that $R(\gamma_0)$ is not $E$-free, since then $\mathcal{G}_\delta^\Phi = \emptyset$ and there is nothing to prove.
     
     Fix an integer $m\geq 1$ and choose $P(\gamma) \in \mathcal{G}_\delta^\Phi \cap \mathcal{D}_m\big(R(\gamma_0)\big)$ with $0 \leq \gamma \leq 1/2$. As $\Lambda = \delta |\mathcal{M}\big(R^\Phi(\gamma_0)\big)| \leq |\mathcal{M}\big(R^\Phi(\gamma_0)\big)|$, by Lemma \ref{lem: finite termination}
     \begin{align*}
         1 \leq \tau_\Lambda\big(P(\gamma)\big) \leq m < \infty.
     \end{align*}
     Then, naturally, by definition (\ref{eq: termination collections})
     \begin{align*}
        P(\gamma) = \pi_0^+P(\gamma) \in \mathcal{S}_k
    \end{align*}
    for some $k\leq m$, which proves the claim. 
     
    \ref{item: G and F extra 2} We may again assume that $R(\gamma_0)$ is not $E$-free, since then $\mathcal{G}_\delta^\Phi = \emptyset$. For any fixed $k=1,2,\dots$ choose a rectangle $P(\alpha) \in \mathcal{S}_k$ with $0 \leq \alpha \leq 1/2$. Clearly, $\mathcal{G}_\delta^\Phi$ and $\mathcal{F}_\delta^\Phi$ are collections of mutually pairwise disjoint rectangles, so we may assume $P(\alpha) \notin \mathcal{G}_\delta^\Phi$. The proof would otherwise be trivial.
     
     By (\ref{eq: termination collections}) we can write $P(\alpha)$ as some forward-in-time parent of a base rectangle in $\mathcal{G}_\delta^\Phi$. In particular,
     \begin{align*}
         P(\alpha) = \pi_i^+P_0(\alpha_0)
     \end{align*}
     for $P_0(\alpha_0) \in \mathcal{G}_\delta^\Phi$ with $0 \leq \alpha_0 \leq 1/2$ and 
     \begin{align}
         i = 1,\dots, \tau_\Lambda\big(P_0(\alpha_0)\big)-1. \label{eq: G and F extra indeces}
     \end{align}
     If we allowed $i =0$, then $ P(\alpha) = P_0(\alpha_0) \in \mathcal{G}_\delta^\Phi$, which is the case we already covered. Based on (\ref{eq: G and F extra indeces}), we will prove by contradiction that $\pi_i^+P_0(\alpha_0)$ cannot intersect any $Q(\beta) \in \mathcal{F}_\delta^\Phi$ with $0 \leq \beta \leq 1/2$. Let us assume that there exists some $\delta$-admissible $Q(\beta) \in \mathcal{F}_\delta^\Phi$ such that 
    \begin{align*}
        \pi_i^+P_0(\alpha_0) \cap Q(\beta) \neq \emptyset.
    \end{align*}
    Since $\pi_i^+P_0(\alpha_0),Q(\beta) \in \mathcal{D}^{\textnormal{ext}}\big(R(\gamma)\big)$, then by the nestedness of the dyadic lattice either 
    \begin{align*}
        \pi_i^+P_0(\alpha_0) \subseteq Q(\beta) \quad \text{or} \quad Q(\beta) \subseteq \pi_i^+P_0(\alpha_0).
    \end{align*}
    We study these two cases separately.

    In the first case, that is, $\pi_i^+P_0(\alpha_0) \subseteq Q(\beta)$ it is clear that $\pi_i^+P_0(\alpha_0)$ is $E$-free. Furthermore, the measure of the forward-in-time parent follows the dyadic scaling as
    \begin{align*}
        |\pi_i^+P_0(\alpha_0)| = |\pi_i P_0(\alpha_0)| \geq \frac{1}{2}\cdot 2^{d(n+p)i}|P_0(\alpha_0)| \geq |P_0(\alpha_0)|.
    \end{align*}
    On the other hand, $P_0(\alpha_0) \in \mathcal{G}_\delta^\Phi$ which means $\pi P_0(\alpha_0) \cap S(\beta_0) \neq \emptyset$ for some $\delta$-admissible $S(\beta_0) \in \mathcal{F}_\delta^\Phi$ with $0 \leq \beta_0 \leq 1/2$. The nestedness of the dyadic lattice implies that $S(\beta_0) \subset \pi P_0(\alpha_0)$. The inclusion to the other direction can be excluded since else $P_0(\alpha_0) \subseteq \pi P_0(\alpha_0) \subseteq S(\beta_0)$ contradicting the definition of $\mathcal{G}_\delta^\Phi$. Thus, we have the lower bound for the measure of $P_0(\alpha_0)$ as
    \begin{align*}
        |P_0(\alpha_0)| \geq |S(\beta_0)|\geq  \delta |\mathcal{M}\big(R^\Phi(\gamma_0)\big)|.
    \end{align*}
    The combination of the previous arguments shows that $\pi_i^+P_0(\alpha_0)$ is large enough $E$-free rectangle, satisfying
    \begin{align*}
        |\mathcal{M}\big(\pi_i^+P_0(\alpha_0)\big)| = |\pi_i^+P_0(\alpha_0)| \geq \delta |\mathcal{M}\big(R^\Phi(\gamma_0)\big)| = \Lambda.
    \end{align*}
    
    We study then the the second case, that is, $Q(\beta) \subseteq \pi_i^+P_0(\alpha_0)$. Since $Q(\beta)$ is $\delta$-admissible, we directly get
    \begin{align*}
        |\mathcal{M}\big(\pi_i^+P_0(\alpha_0)\big)| \geq |\mathcal{M}\big(Q(\beta)\big)| \geq \delta |\mathcal{M}\big(R^\Phi(\gamma_0)\big)| = \Lambda,
    \end{align*}
    proving in both cases that
    \begin{align*}
         |\mathcal{M}\big(\pi_i^+P_0(\alpha_0)\big)| \geq  \Lambda.
    \end{align*}

    To show the contradiction, we analyze whether the stopping time $\tau_\Lambda\big(P_0(\alpha_0)\big)$ finds this large enough $E$-free region. By Lemma \ref{lem: parameters}\ref{item: parameters 2}, the search range parameters establish an interval such that $0 \in [\phi -\theta_0, \Phi -\theta_0]$, and since $\pi_i^+P_0(\alpha_0) = \pi_i^+P_0^0(\alpha_0)$, by the definition (\ref{eq: termination number}) we have an upper bound
    \begin{align*}
        \tau_\Lambda\big(P_0(\alpha_0)\big) \leq i.
    \end{align*}
    However, this upper bound is a contradiction with (\ref{eq: G and F extra indeces}) as
    \begin{align*}
        i \leq \tau_\Lambda\big(P_0(\alpha_0) \big) -1 \leq i -1,
    \end{align*}
    and thus $\pi_i^+P_0(\alpha_0)$ does not intersect any $Q(\beta) \in \mathcal{F}_\delta^\Phi$. The proof is complete.
\end{proof}

\subsection{Exponential estimate}
The final result of this section is the exponential estimate. Thanks to the appropriate stopping time combined with the doubling features of parabolic weak porosity, there is a certain type of exponential decay between the measures of the unions of the collections $\mathcal{S}_k\big(\mathcal{G}_\delta^\Phi\big(R(\gamma_0)\big),\Lambda\big)$.  We prove this result in three parts.

\begin{lemma} \label{lem: exponential decay} Let $E$ be $(c,\delta,\Phi)$-weakly porous for some $c,\delta \in (0,1)$. For any parabolic rectangle $R(\gamma_0) \subseteq \Rn$ with $0\leq \gamma_0 \leq 1/2$ consider $\mathcal{S}_k = \mathcal{S}_k\big(\mathcal{G}_\delta^\Phi\big(R(\gamma_0)\big),\Lambda\big)$ with $\Lambda = \delta|\mathcal{M}\big(R^\Phi(\gamma_0)\big)|$. Then, for every $k =1,2,\dots$ the following are true:
    \begin{enumerate}[label=(\roman*)]
        \item \label{item: exponential decay 1} Consider any nonempty subcollection $\mathcal{C} \subseteq \mathcal{S}_{k}$ such that $\mathcal{C} \subseteq \mathcal{D}_1\big(\pi P(\gamma)\big)$ for every $P(\gamma) \in \mathcal{C}$ with $0 \leq \gamma \leq 1/2$. Then, $\mathcal{C}$ is a proper subset of $\mathcal{D}_1\big(\pi P(\gamma)\big)$.

        \item \label{item: exponential decay 2}The collection
        \begin{align*}
            \pi^+ \mathcal{S}_{k+1} = \big\{\pi^+ P(\gamma) \,:\, P(\gamma) \in \mathcal{S}_{k+1} \big\}
        \end{align*}
        consists of pairwise disjoint rectangles.
    
        \item There exists $\lambda = \lambda(n,p,d) \in (0,1)$ such that
        \begin{align*}
            \Big|\bigcup_{Q(\alpha) \in \mathcal{S}_{k+1}}Q(\alpha) \Big| \leq \lambda^k \Big|\bigcup_{P(\gamma) \in \mathcal{S}_{1}}P(\gamma)\Big|.
        \end{align*} \label{item: exponential decay 3}
    \end{enumerate}
\end{lemma}
\begin{remark}
    For item \ref{item: exponential decay 1}, there actually exists at least $2^n$ rectangles $Q(\gamma) \in \mathcal{D}_1\big(P(\gamma)\big)$ such that $Q(\gamma) \notin \mathcal{C}$, which would improve $\lambda$ from item \ref{item: exponential decay 3} to some degree. However, any $\lambda <1$ is satisfactory for our purposes.
\end{remark}

\begin{proof}
    \ref{item: exponential decay 1} Let $R(\gamma_0) \subseteq \Rn$ be a parabolic rectangle with $0\leq \gamma_0 \leq 1/2$. Consider the collections $\mathcal{F}_\delta^\Phi = \mathcal{F}_\delta^\Phi\big(R(\gamma_0)\big)$ and $\mathcal{G}_\delta^\Phi = \mathcal{G}_\delta^\Phi\big(R(\gamma_0)\big)$, given that $E$ is $(c,\delta,\Phi)$-weakly porous. We may assume that $R(\gamma_0)$ is not $E$-free, since then $\mathcal{G}_\delta^\Phi = \emptyset$ and there is nothing to prove. For a fixed $k = 1,2,\dots$ take any nonempty subcollection $\mathcal{C} \subseteq \mathcal{S}_k$ such that $\mathcal{C} \subseteq \mathcal{D}_1\big(\pi P(\gamma)\big)$ for every $P(\gamma) \in \mathcal{C}$ with $0 \leq \gamma \leq 1/2$. We prove the first claim by contradiction. Thus, we assume that $\mathcal{C} = \mathcal{D}_1\big(\pi P(\gamma)\big)$. We split the proof into two different cases.
    
    In the first case, we assume that there exists at least one $P(\gamma) \in \mathcal{C} \cap \mathcal{G}_\delta^\Phi$. Since $ \mathcal{G}_\delta^\Phi$ is the collection of maximal rectangles that are pairwise disjoint with every rectangle in $\mathcal{F}_\delta^\Phi$, we have
    \begin{align*}
        \pi P(\gamma) \cap S(\beta) \neq \emptyset.
    \end{align*}
    for some $S(\beta)\in \mathcal{F}_\delta^\Phi$ with $0\leq \beta \leq 1/2$. Because $\mathcal{C} = \mathcal{D}_1\big(\pi P(\gamma)\big)$, then there exists some $Q(\gamma) \in \mathcal{C} \subseteq \mathcal{S}_k$ such that
    \begin{align*}
        Q(\gamma) \cap S(\beta) \neq \emptyset.
    \end{align*}
    However, this is a contradiction by Lemma \ref{lem: G and F extra}\ref{item: G and F extra 2}.
    
    We then study the other case. Now, we assume that $\mathcal{C} \cap \mathcal{G}_\delta^\Phi = \emptyset$. Let us concentrate on uppermost and lowermost rectangles of $\mathcal{C}$, and denote them by $U(\gamma),L(\gamma) \in \mathcal{C}$ respectively. In this case, $U(\gamma),L(\gamma) \notin \mathcal{G}_\delta^\Phi$. Moreover, we choose them such that they are from the same column, that is,
    \begin{align}
        U(\gamma) = L^{\theta_1}(\gamma),  \label{eq: exponential decay alignment}
    \end{align}
    where $\theta_1 \in \big\{\big\lfloor2^{dp}\big\rfloor -1, \big\lceil2^{dp}\big\rceil -1 \big\}$, depending on how the dyadic division was performed. Since $\mathcal{C} \subseteq \mathcal{S}_k$, we must be able to write both $U(\gamma)$ and $L(\gamma)$ as forward-in-time parents of some base rectangles as
    \begin{align*}
        U(\gamma) = \pi_i^+U_0(\alpha_0) \quad \text{and} \quad L(\gamma) = \pi_j^+L_0(\beta_0).
    \end{align*}
    for some $U_0(\alpha_0),L_0(\beta_0) \in \mathcal{G}_\delta^\Phi$ with $\alpha_0,\beta_0 \in [0,1/2]$ and
    \begin{align}
        i = 1,\dots, \tau_\Lambda\big(U_0(\alpha_0)\big)-1, \quad \text{and} \quad j = 1,\dots, \tau_\Lambda\big(L_0(\beta_0)\big)-1. \label{eq: exponential decay indeces 1}
    \end{align}
    If we allowed $i =0$ or $j=0$, then $ U(\gamma) = U_0(\alpha_0) \in \mathcal{G}_\delta^\Phi$ or $ L(\gamma) = L_0(\beta_0) \in \mathcal{G}_\delta^\Phi$ which is the case we already covered. With these assumptions, we are going to show a contradiction based on (\ref{eq: exponential decay indeces 1}).

    Due to the maximality of $\mathcal{G}_\delta^\Phi$, and $U_0(\alpha_0) \in \mathcal{G}_\delta^\Phi$, there exists some $\delta$-admissible $S(\beta') \in \mathcal{F}_\delta^\Phi$ with $0 \leq \beta'\leq 1/2$ such that $U_0(\alpha_0) \cap S(\beta') \neq \emptyset$. By the nestedness of the dyadic lattice, we have
    \begin{align*}
        \quad S(\beta') \subset \pi U_0(\alpha_0) \quad \text{or} \quad \pi U_0(\alpha_0) \subseteq S(\beta').
    \end{align*}
    The latter case is clearly impossible by the definition of $\mathcal{G}_\delta^\Phi$, since then $U_0(\alpha_0)$ would intersect $S(\beta')$. Thus, we have
    \begin{align*}
        S(\beta') \subset \pi U_0(\alpha_0) \subseteq \pi_iU_0(\alpha_0) = U^{-\theta_2}(\gamma)
    \end{align*}
    for some integer $ \theta_2 \geq \theta_0$. Since $U(\gamma)$ and $L(\gamma)$ are of the same column, by (\ref{eq: exponential decay alignment}) we can write
    \begin{align*}
        \pi_j^+ L_0^{\theta}(\beta_0) = L^\theta(\gamma) = U^{-\theta_2}(\gamma)
    \end{align*}
    for $\theta = \theta_1 -\theta_2 $. Thus, the maximal $E$-free region of $\pi_j^+ L_0^\theta(\beta_0)$ is bounded below by
    \begin{align}
        |\mathcal{M}\big(\pi_j^+ L_0^\theta(\beta_0)\big)| \geq |S(\beta')| \geq \delta |\mathcal{M}\big(R^\Phi(\gamma_0)\big)| = \Lambda. \label{eq: exponential decay maximal hole 1}
    \end{align}

    We then study $\theta$ by estimating $\theta_2$. By Lemma \ref{lem: bounded distance}, temporal distance between $U^{-\theta_2}(\gamma)$ and $U(\gamma)$ is bounded by
    \begin{align*}
        \dist_t\big(\partial_{\textnormal{low}}U^{-\theta_2}(\gamma),\partial_{\textnormal{low}} U(\gamma)\big) &= \dist_t\big(\partial_{\textnormal{low}}\pi_i U_0(\alpha_0),\partial_{\textnormal{low}}\pi_i^+ U_0(\alpha_0)\big) \\
        &\leq \dist_t\big(\partial_{\textnormal{low}}U_0(\alpha_0),\partial_{\textnormal{low}}\pi_i^+U_0(\alpha_0)\big) + l_t\big(\pi_i U_0(\alpha_0)\big)\\
        &<2\theta_0\frac{2^{dp}}{2^{dp}-1}l_t\big(\pi_i^+ U_0(\alpha_0)\big) + l_t\big(\pi_i U_0(\alpha_0)\big) \\
        &\leq \bigg(2\theta_0 + \bigg\lceil\frac{2\theta_0}{2^{dp}-1} \bigg\rceil +1 \bigg)l_t\big( U(\gamma)\big).
    \end{align*}
    Since $\theta_2$ is an integer, the strict inequality above tells us that necessarily
    \begin{align*}
        \theta_0 \leq \theta_2 \leq 2 \theta_0 + \bigg\lceil\frac{2\theta_0}{2^{dp}-1} \bigg\rceil.
    \end{align*}
    Using the upper and lower bound of $\theta_2$, we restrict $\theta$ on an interval
    \begin{align*}
        \theta_1 - 2\theta_0 -\bigg\lceil\frac{2\theta_0}{2^{dp}-1} \bigg\rceil \leq \theta \leq \theta_1 - \theta_0.
    \end{align*}
    Recall  $\theta_1 \in \big\{\big\lfloor2^{dp}\big\rfloor -1, \big\lceil2^{dp}\big\rceil -1 \big\}$, which implies 
    \begin{align*}
        \big\lfloor2^{dp}\big\rfloor -1 - 2\theta_0 -\bigg\lceil\frac{2\theta_0}{2^{dp}-1} \bigg\rceil \leq \theta \leq  \big\lceil2^{dp}\big\rceil -1 - \theta_0.
    \end{align*}
    Lemma \ref{lem: parameters}\ref{item: parameters 3} and \ref{item: parameters 4} states that the above implies
    \begin{align*}
        \phi - \theta_0 \leq \theta \leq \Phi - \theta_0.
    \end{align*}
    
     Finally, we check the stopping time. The bounds on $\theta$ are appropriate and by (\ref{eq: exponential decay maximal hole 1}) 
    \begin{align*}
        |\mathcal{M}\big(\pi_j^+ L_0^\theta(\beta_0)\big)| \geq  \Lambda.
    \end{align*}
    Thus, $\tau_\Lambda\big(L_0(\beta_0)\big) \leq j$ by definition (\ref{eq: termination number}). However, by (\ref{eq: exponential decay indeces 1}) we have
    \begin{align*}
        j \leq \tau_\Lambda\big(L_0(\beta_0)\big) -1 \leq j-1, 
    \end{align*}
    which is a contradiction, proving the claim.

    \ref{item: exponential decay 2} We show that $\pi^+\mathcal{S}_{k+1}$ is a collection of pairwise disjoint rectangles with a proof by a contradiction. For some fixed $k=1,2,\dots$, let $P(\alpha),Q(\beta) \in \pi^+\mathcal{S}_{k+1}$ with $\alpha,\beta \in [0,1/2]$ be intersecting rectangles. Without loss of generality we may assume $Q(\beta) \subset P(\alpha)$. By the definition of $\pi^+ \mathcal{S}_{k+1}$, we can write $P(\alpha)$ and $Q(\beta)$ as some forward-in-time parents of some base rectangles as
    \begin{align*}
        P(\alpha) = \pi_i^+P_0(\alpha_0) \quad \text{and} \quad Q(\beta) = \pi_j^+Q_0(\beta_0)
    \end{align*}
    for some $P_0(\alpha_0),Q_0(\beta_0) \in \mathcal{G}_\delta^\Phi$ with $\alpha_0,\beta_0 \in [0,1/2]$ and 
    \begin{align}
        i =1,\dots,\tau_\Lambda\big(P_0(\alpha_0)\big) -1 \quad \text{and} \quad j =1,\dots,\tau_\Lambda\big(Q_0(\beta_0)\big) -1. \label{eq: exponential decay indeces 2}
    \end{align}
    The indices $i$ and $j$ cannot be zero, since $\pi^+ \mathcal{S}_{k+1}$ includes taking one forward-in-time parent of any rectangle in $\mathcal{S}_{k+1}$.

    Since $Q_0(\beta_0) \in \mathcal{G}_\delta^\Phi$, there exists a $\delta$-admissible $S(\gamma) \in \mathcal{F}_\delta^\Phi$ with $0 \leq \gamma \leq 1/2$ such that 
    \begin{align*}
        S(\gamma) \subseteq \pi Q_0(\beta_0) \subseteq \pi_j Q_0(\beta_0) =  Q^{\theta_1}(\beta)
    \end{align*}
    for some integer $\theta_1 \leq -\theta_0$. Moreover, since $Q(\beta) \subset P(\alpha)$ are spatially aligned, there exists an integer $\theta \leq 0$ such that 
    \begin{align*}
        Q^{\theta_1}(\beta) \subset P^\theta (\alpha) = \pi_i^+ P_0^{\theta}(\alpha_0).
    \end{align*}
    Thus, the maximal $E$-free hole in $\pi_i^+ P_0^\theta(\alpha_0)$ is bounded from below by
    \begin{align}
        |\mathcal{M}\big(\pi_i^+ P_0^\theta(\alpha_0)\big)| \geq |S(\gamma)| \geq \delta |\mathcal{M}\big(R^\Phi(\gamma_0)\big)| = \Lambda. \label{eq: exponential decay maximal hole 2}
    \end{align}
    
    We then study $\theta$.  Since  $ Q_0(\beta_0) \subseteq \pi_j Q_0(\beta_0) \subset P^{\theta}(\alpha)$ and $Q(\beta) \subset P(\alpha)$, by Lemma \ref{lem: bounded distance} the temporal distance between $P^{\theta}(\alpha)$ and $P(\alpha)$ satisfies
    \begin{align*}
         \dist_t\big(\partial_{\textnormal{low}}P^{\theta}(\alpha),\partial_{\textnormal{low}}P(\alpha)\big) & \leq \dist_t\big(\partial_{\textnormal{low}} Q_0(\beta_0),\partial_{\textnormal{low}} Q(\beta)\big) + l_t\big(P(\alpha)\big)\\
         &= \dist_t\big(\partial_{\textnormal{low}}Q_0(\beta_0),\partial_{\textnormal{low}}\pi_j^+Q_0(\beta_0)\big) + l_t\big( P(\alpha)\big)\\
        &<2\theta_0\frac{2^{dp}}{2^{dp}-1}l_t\big(\pi_j^+ Q_0(\beta_0)\big) + l_t\big(P(\alpha)\big) \\
        &\leq \bigg(\frac{4\theta_0}{2^{dp}-1}  \bigg) l_t\big(P(\alpha)\big) +  l_t\big( P(\alpha)\big)\\
        &\leq \bigg(\bigg\lceil\frac{4\theta_0}{2^{dp}-1} \bigg\rceil  +1 \bigg) l_t\big(P(\alpha)\big).
    \end{align*}
    We also used Corollary \ref{cor: dyadic scale} for $2^{dp}l_t\big(\pi_j^+Q_0(\beta_0)\big) \leq 2 l_t\big(P(\alpha)\big)$. Since $\theta$ is an integer, the strict inequality above tells us that necessarily
    \begin{align*}
        -\bigg\lceil\frac{4\theta_0}{2^{dp}-1} \bigg\rceil \leq \theta \leq 0.
    \end{align*}
    Lemma \ref{lem: parameters}\ref{item: parameters 2} and \ref{item: parameters 5} state that the above implies
    \begin{align*}
        \phi - \theta_0 \leq \theta \leq \Phi - \theta_0.
    \end{align*}
    
    Finally, we check the definition (\ref{eq: termination number}). The bounds on $\theta$ are appropriate and by (\ref{eq: exponential decay maximal hole 2}), we have 
    \begin{align*}
        |\mathcal{M}\big(\pi_i^+ P_0^\theta(\alpha_0)\big)| \geq  \Lambda.
    \end{align*}
    Thus, $\tau_\Lambda\big(P_0(\alpha_0)\big) \leq i$ by definition (\ref{eq: termination number}). However, by (\ref{eq: exponential decay indeces 2}) we have
    \begin{align*}
        i \leq \tau_\Lambda\big(P_0(\alpha_0)\big) -1 \leq i-1, 
    \end{align*}
    which is a contradiction, proving the claim.

    \ref{item: exponential decay 3} To prove the decay estimate, we utilize the previous parts. For a fixed $k=1,2,\dots$, we first observe that for every $P(\gamma) \in \pi^+ \mathcal{S}_{k+1}$ with $0\leq \gamma \leq 1/2$, there exists $Q(\alpha) \in \mathcal{S}_{k+1}$ with $0 \leq \alpha \leq 1/2$ such that
    \begin{align*}
        \pi^+Q(\alpha) = P(\gamma).
    \end{align*} 
    Equivalently, we can say $Q(\alpha) \in \mathcal{D}_1\big(P^{-\theta_0}(\gamma)\big) \cap \mathcal{S}_{k+1}$. Let us denote every such $Q(\alpha)$ by $\mathcal{C}\big(P(\gamma)\big) \subseteq \mathcal{D}_1\big(P^{-\theta_0}(\gamma)\big)$. Since $\pi Q(\alpha) = P^{-\theta_0}(\gamma)$
    for any $Q(\alpha) \in \mathcal{C}\big(P(\gamma)\big)$, the collection $\mathcal{C}\big(P(\gamma)\big)$ satisfies the requirements of item \ref{item: exponential decay 1}. We can write
    \begin{align*}
        \sum_{Q(\alpha) \in \mathcal{S}_{k+1}}|Q(\alpha)| =  \sum_{P(\gamma) \in \pi^+\mathcal{S}_{k+1}} \sum_{Q(\alpha) \in \mathcal{C}(P(\gamma))} |Q(\alpha)|.
    \end{align*}
    
    Depending on how the dyadic division was performed, notice that $\mathcal{D}_1\big(P(\gamma)\big)$ is made of $N \in \big\{2^{dn}\big\lfloor2^{dp}\big\rfloor, 2^{dn}\big\lceil2^{dp}\big\rceil \big\}$ rectangles. However, item \ref{item: exponential decay 1} implies that $\mathcal{C}\big(P(\gamma)\big)$ can have at most $N-1$ rectangles. This means
    \begin{align*}
        \sum_{P(\gamma) \in \pi^+\mathcal{S}_{k+1}} \sum_{Q(\alpha) \in \mathcal{C}(P(\gamma))} |Q(\alpha)| &\leq  \sum_{P^{-\theta_0}(\gamma) \in \pi^+\mathcal{S}_{k+1}} \frac{N-1}{N} |P(\gamma)| \\
        &\leq  \sum_{P(\gamma) \in \pi^+\mathcal{S}_{k+1}} \bigg(1-\frac{1}{2^{dn}\big\lceil2^{dp}\big\rceil}\bigg) |P(\gamma)|\\
        &= \lambda \sum_{P(\gamma) \in \pi^+\mathcal{S}_{k+1}} |P(\gamma)|,
    \end{align*}
    where $\lambda = \lambda(n,p) \in (0,1)$.  By Lemma \ref{lem: parent sets} $\pi^+ \mathcal{S}_{k+1}\subseteq \mathcal{S}_k$, while item \ref{item: exponential decay 2} implies $\pi^+ \mathcal{S}_{k+1}$ is a collection of pairwise disjoint rectangles. Hence, we obtain
    \begin{align*}
        \sum_{P(\gamma) \in \pi^+\mathcal{S}_{k+1}} |P(\gamma)| &= \Big|\bigcup_{P(\gamma) \in \pi^+\mathcal{S}_{k+1}} P(\gamma)\Big| \leq \Big|\bigcup_{P(\gamma) \in \mathcal{S}_{k}} P(\gamma)\Big|.
    \end{align*}
    Combining the above with the previous estimates yields
    \begin{align*}
        \sum_{Q(\alpha) \in \mathcal{S}_{k+1}}|Q(\alpha)| \leq \lambda \Big|\bigcup_{P(\gamma) \in \mathcal{S}_{k}} P(\gamma)\Big|,
    \end{align*}
    from which it is easy to deduce
    \begin{align*}
        \Big|\bigcup_{P(\gamma) \in \mathcal{S}_{k+1}}P(\gamma)\Big| \leq \lambda^{k} \Big|\bigcup_{P(\gamma) \in \mathcal{S}_1} P(\gamma)\Big|
    \end{align*}
    for every $k=1,2,\dots$, finishing the proof.
\end{proof}

\section{Full characterization of distance weights} \label{sec: full charac}
In this final section, we demonstrate the full characterization of parabolic Muckenhoupt distance weights via parabolic weakly porous sets. We first show the last missing direction of the characterization, and then gather the results into the main theorem of this paper.

\subsection{The $\alpha$-improvement of parabolic weakly porous sets}
We begin by demonstrating the $\alpha$-improvement of parabolic weakly porous sets. We show that if $E$ is $(c_0,\delta_0,\Phi)$-weakly porous for fixed $\Phi$ of Lemma \ref{lem: parameters}, then $E$ is $\alpha$-improving. We will bridge the $A_1^+(\gamma)$ distance weights and weakly porous sets using a fixed translation, and then show that the characterization holds for every translation with positive time-lag. The parameters $\theta_0$ and $\phi$ are also considered fixed to satisfy Lemma \ref{lem: parameters}.

\begin{lemma} \label{lem: wp => polynomial decay}
    Let $E \subseteq \Rn$ be an nonempty $(c_0,\delta_0,\Phi)$-weakly porous for $c_0,\delta_0 \in (0,1)$. Then, $E$ is also $\alpha$-improving for $\alpha = \alpha(n,p,d,c_0,\delta_0)>0$.
\end{lemma}

\begin{proof}
    We show $\alpha$-improvement using the sequential version of $\alpha$-improvement, see Proposition \ref{prop: alpha improvement eqv}. We generate inductively sequences $(c_i)_{i\in\N}$ and $(\delta_i)_{i\in \N}$ with $\delta_i \rightarrow 0$ and $\delta_{i+1}/\delta_i \geq \eta$ for some $\eta >0$ such that $E$ is $(c_i,\delta_i,\Phi)$-weakly porous and 
    \begin{align}
        K\delta_i^{2\alpha} \leq 1-c_i \leq K \delta_i^\alpha \label{eq: termination => pd induction assumption}
    \end{align}
    for every $i \in \N$. The base case of the induction will be the initial assumption of $E$ being $(c_0,\delta_0,\Phi)$-weakly porous. We set $K = (1-c_0) \delta_0^{-\alpha} >0$, while fixing $\alpha>0$ later. The base case of the induction now follows
    \begin{align*}
        (1-c_0) = (1-c_0)\delta_0^{-\alpha}\delta_0^\alpha = K \delta_0^{\alpha}.
    \end{align*}
    On the other hand, the choice of $K$ implies the inequality 
    \begin{align*}
        K \delta_0^{2\alpha} =  (1-c_0)\delta_0^\alpha \leq 1-c_0,
    \end{align*}
    proving the base case.
    
    We proceed by assuming that for some fixed $i \in \N$ set $E$ is $(c_i,\delta_i,\Phi)$-weakly porous for some $c_i,\delta_i \in (0,1)$ satisfying (\ref{eq: termination => pd induction assumption}). Suppose $R(\gamma_0) \subseteq \Rn$ is a parabolic rectangle with $0 \leq \gamma_0 \leq 1/2$, and let us denote $\mathcal{F}_{\delta_i}^\Phi = \mathcal{F}_{\delta_i}^\Phi\big(R(\gamma_0)\big)$ and $\mathcal{G}_{\delta_i}^\Phi = \mathcal{G}_{\delta_i}^\Phi\big(R(\gamma_0)\big)$ for the complementary rectangles. We further denote
    \begin{align*}
        |F_{i}| = \Big|\bigcup_{P(\gamma) \in \mathcal{F}_{\delta_i}^\Phi} P(\gamma)\Big|  \geq c_i|R(\gamma_0)|,
    \end{align*}
    where the inequality follows from Proposition \ref{prop: Fdelta}, and denote the complement
    \begin{align*}
        |G_{i}| = \big|R(\gamma_0) \setminus F_{i}\big| \leq (1-c_i)|R(\gamma_0)|.
    \end{align*}

    Consider then the collections $\mathcal{S}_k^i = \mathcal{S}_k(\mathcal{G}_{\delta_i}^\Phi,\Lambda)$ with $\Lambda = \delta_i|\mathcal{M}\big(R^\Phi(\gamma_0)\big)|$ for every $k=1,2,\dots$, see (\ref{eq: termination collections}), and define recursively
    \begin{align*}
        S_{1}^i = \bigcup_{P(\gamma)\in \mathcal{S}_{1}^i}P(\gamma) \quad \text{and} \quad S_{k+1}^i = \bigg(\bigcup_{P(\gamma)\in \mathcal{S}_{k+1}^i}P(\gamma) \bigg) \setminus \bigcup_{j=1}^{k}S_j^i .
    \end{align*}
    We introduce the key features of these sets. Clearly, $S_{k}^i$ are pairwise disjoint for each $k=1,2,\dots$ with
    \begin{align*}
        S_{k}^i \subseteq \bigcup_{P(\gamma) \in \mathcal{S}_{k}^i} P(\gamma) \quad \text{and} \quad \bigcup_{j=1}^k S_{j}^i = \bigcup_{j=1}^k \bigcup_{P(\gamma) \in \mathcal{S}_{k}^i} P(\gamma)
    \end{align*}
    for any $k = 1,2,\dots$. Furthermore, Lemma \ref{lem: G and F extra} implies
    \begin{align}
         G_{i} = \bigcup_{P(\gamma) \in \mathcal{G}_i}P(\gamma) \subseteq \bigcup_{k=1}^\infty S_k^i \label{eq: termination => pd G}
    \end{align}
    and
    \begin{align}
         \bigcup_{k=1}^\infty S_k^i \cap \bigcup_{P(\gamma) \in \mathcal{F}_i}P(\gamma) = \bigcup_{k=1}^\infty S_k^i \cap F_{i} = \emptyset. \label{eq: termination => pd F}
    \end{align}
    
    For an estimate which is needed later in the proof, it is important that we show that the measure of the union of $S_k^i$ outside of $R(\gamma_0)$ is proportional to $1-c_i$. Since $E$ is also $(c_0,\delta_0,\Phi)$-weakly porous, we can apply Lemma \ref{lem: decay of interim space}, which states
    \begin{align*}
        \bigcup_{k=1}^\infty S_k^i \subseteq \bigcup_{\theta \in [0,\psi_i]}R^{\theta}(\gamma_0), \quad \text{for some} \quad  0\leq \psi_i \leq C_0 \delta_i^\mu
    \end{align*}
    where $\mu = \mu(n,p,c_0,\delta_0)>0$ and $C_0 = C_0(p,d) >0$. We restrict $2\alpha \leq \mu$, which allows us to apply (\ref{eq: termination => pd induction assumption}) to estimate $\psi_i$. Now, by (\ref{eq: termination => pd F}) we obtain
    \begin{align}
        \begin{split}
        \sum_{k=1}^\infty |S_k^i| &= \Big|\bigcup_{k=1}^\infty S_k^i\Big| \leq \Big|\bigcup_{\theta \in [0,\psi_i]}R^{\theta}(\gamma_0)\setminus F_{i}\Big| =|G_{i}| + \Big|\bigcup_{\theta \in [0,\psi_i]}R^{\theta}(\gamma_0)\setminus R(\gamma_0)\Big| \\
        &\leq (1-c_i + \psi_i) |R(\gamma_0)| \leq (1-c_i +C_0\delta_i^\mu) |R(\gamma_0)| \\
        &\leq \big(1-c_i +C_0K^{-1}(1-c_i)\big) |R(\gamma_0)| \\
        &= \Big(1+C_0 \frac{\delta_0^\alpha}{1-c_0}\Big)(1-c_i) |R(\gamma_0)| \\
        &\leq (1+C)(1-c_i) |R(\gamma_0)|,
        \end{split} \label{eq: termination => pd Sk estimate}
    \end{align}
    where $C = C_0(1-c_0)^{-1}$.
    
    Next, we get to the key part of the argument. We claim that for any $h \in [1/2,1)$, yet to be chosen, we can write
    \begin{align}
        \Big|F_i \cup \bigg(G_{i} \cap  \bigcup_{k=1}^{H_i} S_k^i \bigg)\Big| &\geq \big(c_i+ h(1-c_i)\big)|R(\gamma_0)| , \label{eq: termination => pd new covering layer estimate}
    \end{align}
     where $H_i \in \N$ is large enough. Notice that by (\ref{eq: termination => pd G}) we can formulate (\ref{eq: termination => pd new covering layer estimate}) equivalently as
    \begin{align*}
        \Big|G_{i} \cap \bigcup_{k=H_i +1}^\infty S_k^i\Big| &= |G_i| - \Big|G_i\cap \bigcup_{k=1}^{H_i} S_k^i\Big| \\
        &\leq |G_i| - \Big(\big(c_i+ h(1-c_i)\big)|R(\gamma_0)| - |F_i|\Big)\\
        &= (1-h)(1-c_i) |R(\gamma_0)|.
    \end{align*}
    We want to show that $H_i$ can be chosen independent of $i$. To do this, we use Lemma \ref{lem: exponential decay}\ref{item: exponential decay 3} to estimate the measures of the sets $S_k^i$ as
    \begin{align*}
        |S_k^i| \leq \Big |\bigcup_{P(\gamma) \in \mathcal{S}_k^i}P(\gamma)\Big| \leq \lambda^{k-1} \Big|\bigcup_{P(\gamma) \in \mathcal{S}_1^i}P(\gamma)\Big|= \lambda^{k-1}|S_1^i|
    \end{align*}
    for every $k =2,3,\dots$, where $\lambda = \lambda(n,p,d) \in (0,1)$. With this exponential estimate and (\ref{eq: termination => pd Sk estimate}) we get
    \begin{align*}
        \Big|G_{i} \cap \bigcup_{k=H_i +1}^\infty S_k^i\Big| &\leq \Big|\bigcup_{k=H_i +1}^\infty S_k^i\Big| = \sum_{k=H_i +1}^\infty |S_k^i| \leq \sum_{k=H_i +1}^\infty \lambda^{k-1}|S_1^i| =  \frac{\lambda^{H_i}}{1-\lambda} |S_1^i|\\
        &\leq \frac{\lambda^{H_i}}{1-\lambda} \sum_{k=1}^\infty |S_k^i| \leq (1+C)\frac{\lambda^{H_i}}{1-\lambda} (1-c_i) |R(\gamma_0)|.
    \end{align*}
    It follows that it is enough to require $H_i$ to be so large that
    \begin{align*}
        (1+C)\frac{\lambda^{H_i}}{1-\lambda} (1-c_i) |R(\gamma_0)| \leq (1-h)(1-c_i)|R(\gamma_0)|,
    \end{align*}
    which would imply (\ref{eq: termination => pd new covering layer estimate}). We choose $H_i$ be the smallest integer satisfying the above, in particular, $H_i$ satisfies
    \begin{align}
        \frac{\lambda(1-\lambda)(1-h)}{1+C} < \lambda^{H_i} \leq \frac{(1-\lambda)(1-h)}{1+C} \leq (1-\lambda)(1-h).\label{eq: termination => pd lambda estimate}
    \end{align}
    Note the upper and lower bound of $H_i$ are independent of $i$, which will be crucial for later.

    We recall that the definition (\ref{eq: termination collections}) necessitates
    \begin{align*}
        \bigcup_{k=1}^{H_i}\mathcal{S}_k^i \subseteq \mathcal{D}^{\text{ext}}\big(R(\gamma_0)\big).
    \end{align*}
    Hence, for every $P(\gamma) \in \bigcup_{k=1}^{H_i}\mathcal{S}_k^i$ with $0\leq \gamma \leq 1/2$ and $P(\gamma) \in \mathcal{D}\big(R(\gamma_0)\big)$ there exists a maximal rectangle $\Tilde{P}(\gamma_1) \in \bigcup_{k=1}^{H_i}\mathcal{S}_k^i$ with $0 \leq \gamma_1\leq 1/2$ such that $P(\gamma) \in \mathcal{D}\big(\Tilde{P}(\gamma_1)\big) \subseteq \mathcal{D}\big(R(\gamma_0)\big)$. Denote the collection of these $\Tilde{P}(\gamma_1)$ by
    \begin{align*}
        \mathcal{A}_i \subseteq \bigcup_{k=1}^{H_i}\mathcal{S}_k^i \subseteq \mathcal{D}\big(R(\gamma_0)\big).
    \end{align*}
    By the nestedness of the dyadic lattice $\mathcal{A}_i$ is a collection of pairwise disjoint rectangles, since if any of them intersected, then the other rectangle would not be maximal. Moreover, the construction of $\mathcal{A}_i$ and (\ref{eq: termination => pd F}) imply
    \begin{align*}
        \bigcup_{P(\gamma)\in \mathcal{A}_i}P(\gamma)  = R(\gamma_0) \cap \bigcup_{k=1}^{H_i} S_k^i = \big(R(\gamma_0)\setminus F_i\big) \cap \bigcup_{k=1}^{H_i} S_k^i = G_{i} \cap \bigcup_{k=1}^{H_i} S_k^i.
    \end{align*}
    
    To obtain $c_{i+1}$, we cover $R(\gamma_0)$ with the set $F_i$ and using parts of $G_i$. Take any $P(\gamma)\in \mathcal{A}_i$ with $0\leq \gamma \leq 1/2$. Observe that then $P(\gamma) \in \mathcal{D}_1\big(\pi P(\gamma)\big) \subseteq \mathcal{D}_1\big(R(\gamma_0)\big)$. Since $E$ is $(c_0,\delta_0,\Phi)$-weakly porous, by Theorem \ref{theo: doubling porosity} there exists $\sigma_1 = \sigma_1(n,p,d,c_0,\delta_0,\phi,\Phi) \in (0,1)$ such that for any $\psi \in [\phi, \Phi]$ we can find pairwise disjoint $E$-free subrectangles
    $Q_j(\beta_j) \in\mathcal{D}\big(P(\gamma)\big)$ with $0 \leq \beta \leq 1/2$ for $j= 1,2,\dots, N$ such that
    \begin{align}
        |Q_j(\beta_j)|\geq \sigma_1 |\mathcal{M}\big(\pi P^{\psi}(\gamma)\big)| = \sigma_1 |\mathcal{M}\big(\pi^+ P^{\psi-\theta_0}(\gamma)\big)| \label{eq: termination => pd doubling}
    \end{align}
    and
    \begin{align*}
        F\big(P(\gamma)\big) = \sum_{j=1}^N |Q_j(\beta_j)| \geq c_0 |P(\gamma)|.
    \end{align*}

    We will first concentrate on determining $c_{i+1} >0$. By summing the estimate above over every rectangle in $  \mathcal{A}_i$, we get
    \begin{align*}
        \sum_{P(\gamma) \in \mathcal{A}_i} F \big(P(\gamma)\big) \geq c_0 \sum_{P(\gamma) \in \mathcal{A}_i} |P(\gamma)| = c_0 \Big|G_{i} \cap \bigcup_{k=1}^{H_i} S_k^i\Big|.
    \end{align*}
    The collections $\mathcal{F}_{\delta_i}^\Phi$ and $\mathcal{A}_i$ are clearly finite, pairwise disjoint and $\mathcal{F}_{\delta_i}^\Phi\cup \mathcal{A}_i \subseteq \mathcal{D}\big(R(\gamma_0)\big)$. Furthermore, since $F_i$ and $G_i$ are disjoint, by (\ref{eq: termination => pd new covering layer estimate}) we have
    \begin{align*}
        \sum_{P(\gamma) \in \mathcal{A}_i} F\big(P(\gamma)\big) + \sum_{P(\gamma) \in \mathcal{F}_{\delta_i}^\Phi}|P(\gamma)| &\geq c_0 \Big|G_{i} \cap \bigcup_{k=1}^{H_i} S_k^i\Big| + |F_i|\\
        &= c_0 \Big| F_i\cup \bigg(G_{i} \cap \bigcup_{k=1}^{H_i} S_k^i\bigg)\Big| + (1-c_0) |F_i|\\
        &\geq \Big(c_0\big(c_i+h(1-c_i)\big) + (1-c_0)c_i\Big) |R(\gamma_0)| \\
        &= \big(hc_0 (1-c_i) + c_i\big) |R(\gamma_0)|.
    \end{align*}
    We define $c_{i+1} = hc_0 (1-c_i) + c_i \in (0,1)$ as wanted. Moreover,
    \begin{align}
        1 - c_{i+1} = 1- c_i - hc_0(1-c_i) = (1-hc_0)(1-c_i). \label{eq: termination => pd ci+1}
    \end{align}
    
    Next, we determine the coefficient $\delta_{i+1} >0$. Since $\mathcal{A}_i \subseteq \bigcup_{k=1}^{H_i}\mathcal{S}_k^i$, for each $P(\gamma) \in \mathcal{A}_i$ we have $\tau_\Lambda\big(P(\gamma)\big) = j_P$ for some $j_P=1,\dots,H_i$, recall (\ref{eq: termination collections}). By definition (\ref{eq: termination number}), $j_P$ is the first index such that
    \begin{align*}
        |\mathcal{M}\big(\pi_{j_P}^+P^\theta(\gamma)\big)| \geq \Lambda = \delta_i|\mathcal{M}\big(R^\Phi(\gamma_0)\big)|
    \end{align*}
    for some $\theta \in [\phi-\theta_0, \Phi-\theta_0]$. If $j_P = 1$, we set $\psi = \theta + \theta_0 \in [\phi, \Phi]$, and then clearly (\ref{eq: termination => pd doubling}) becomes
    \begin{align*}
        |Q_j(\beta_j)| \geq \sigma_1 |\mathcal{M}\big(\pi^+ P^{\psi-\theta_0}(\gamma)\big)| = \sigma_1|\mathcal{M}\big(\pi_{j_P}^+P^\theta(\gamma)\big)| \geq \sigma_1\delta_i|\mathcal{M}\big(R^\Phi(\gamma_0)\big)|.
    \end{align*}
    On the other hand, if $j_P\geq 2$, we set $\psi = \theta_0\in [\phi,\Phi]$, see Lemma \ref{lem: parameters}\ref{item: parameters 2}. Since $E$ is $(c_0,\delta_0,\Phi)$-weakly porous, we apply Corollary \ref{cor: mh doubling}, transforming (\ref{eq: termination => pd doubling}) into
    \begin{align*}
        |Q_j(\beta_j)| \geq \sigma_1|\mathcal{M}\big(\pi^+P(\gamma)\big)| &\geq \sigma_1\sigma_2^{j_P-1}|\mathcal{M}\big(\pi_{j_P-1}^+(\pi^+P(\gamma))^\theta\big)|\\
        &= \sigma_1\sigma_2^{j_P-1}|\mathcal{M}\big(\pi_{j_P}^+P^\theta(\gamma)\big)| \\
        &\geq \sigma_1\sigma_2^{H_i-1} \delta_i|\mathcal{M}\big(R^\Phi(\gamma_0)\big)|,
    \end{align*}
    where $\sigma_2 = \sigma_2(n,p,d,c_0,\delta_0,\phi,\Phi) \in (0,1)$. We set $\sigma = \min\{\sigma_1,\sigma_2\}$ to obtain
    \begin{align*}
        |Q_j(\beta_j)| \geq \sigma^{H_i} \delta_i|\mathcal{M}\big(R^\Phi(\gamma_0)\big)|
    \end{align*}
    in both cases, that is, $j_P=1$ and $j_P \geq 2$. Hence, we define $\delta_{i+1} = \sigma^{H_i} \delta_i \in (0,1)$. It follows that $E$ is $(c_{i+1},\delta_{i+1},\Phi)$-weakly porous, where the rectangles of $\mathcal{F}_{\delta_i}^\Phi$ and  $Q_j(\beta_j) \in\mathcal{D}\big(P(\gamma)\big)$ for each $P(\gamma) \in \mathcal{A}_i$ form the $\delta_{i+1}$-admissible rectangles. The choice of rectangles is valid since  $\mathcal{F}_{\delta_i}^\Phi$ is also a collection of $\delta_{i+1}$-admissible rectangles as $\delta_{i+1} \leq \delta_i$.
    
    It is left to prove (\ref{eq: termination => pd induction assumption}) and verify that the sequence $(\delta_i)_{i\in \N}$ satisfies its extra conditions to show $\alpha$-improvement. We do this by choosing the parameter $1/2\leq h<1$ carefully. We set it such that
    \begin{align}
        \big((1-\lambda)(1-h)\big)^{\frac{3}{2}\alpha\log_\lambda(\sigma)} = 1-hc_0, \label{eq: termination => pd setting h}
    \end{align}
    for any $0 \leq \alpha \leq \mu/2$ small enough. We show that such $1/2\leq h<1$ exists. We test the end point values of $h$. First, if $h=1$,
    then
    \begin{align*}
        \big((1-\lambda)(1-h)\big)^{\frac{3}{2}\alpha\log_\lambda(\sigma)} = 0 < 1-c_0 = 1-hc_0.
    \end{align*} On the other hand, if we further restrict
    \begin{align}
        \alpha \leq \frac{1}{2 \log_\lambda(\sigma)} \cdot \frac{\ln\big(1-\frac{1}{2}c_0\big)}{\ln\big(\frac{1}{2}(1-\lambda)\big)}, \label{eq: termination => pd alpha1}
    \end{align}
    then for $h= 1/2$ we have
     \begin{align*}
        \big((1-\lambda)(1-h)\big)^{\frac{3}{2}\alpha\log_\lambda(\sigma)} &> \big((1-\lambda)(1-h)\big)^{2\alpha\log_\lambda(\sigma)} \\
        &= \Big(\frac{1}{2}(1-\lambda)\Big)^{2\alpha\log_\lambda(\sigma)}\\
        &\geq 1-\frac{1}{2}c_0 = 1- hc_0.
    \end{align*}
    Since (\ref{eq: termination => pd setting h}) is continuously dependent on $h$ for any $0 \leq \alpha \leq \mu/2$ satisfying (\ref{eq: termination => pd alpha1}), there must exist some $h = h(\alpha,c_0,\lambda, \sigma) \in [1/2,1)$ satisfying it. Substituting $\delta_{i+1} = \sigma^{H_i} \delta_i$, (\ref{eq: termination => pd lambda estimate}),  (\ref{eq: termination => pd induction assumption}) and (\ref{eq: termination => pd setting h}) shows
    \begin{align*}
         K\delta_{i+1}^{2\alpha} &= K(\sigma^{H_i}\delta_i)^{2\alpha} =K\delta_i^{2\alpha}(\lambda^{H_i})^{2\alpha\log_\lambda(\sigma)} \\
         &\leq K\delta_i^{2\alpha}\big((1-\lambda)(1-h)\big)^{2\alpha\log_\lambda(\sigma)} \\
         &\leq \big((1-\lambda)(1-h)\big)^{\frac{3}{2}\alpha\log_\lambda(\sigma)}(1-c_i) \\
         &= (1-hc_0)(1-c_i) \\
         &= 1-c_{i+1},
    \end{align*}
    where the last line was from (\ref{eq: termination => pd ci+1}), proving the first inequality of the induction.    
    
    We then prove the other inequality of the induction. We again substitute $\delta_{i+1} = \sigma^{H_i} \delta_i$, (\ref{eq: termination => pd lambda estimate}), (\ref{eq: termination => pd induction assumption}) and (\ref{eq: termination => pd setting h}) to show
    \begin{align*}
        K \delta_{i+1}^\alpha &= K(\sigma^{H_i}\delta_i)^\alpha =K\delta_i^{\alpha}(\lambda^{H_i})^{\alpha\log_\lambda(\sigma)} \\
        &\geq K \delta_i^{\alpha} \Big(\frac{\lambda(1-\lambda)(1-h)}{1 + C}\Big)^{\alpha \log_\lambda(\sigma)} \\
        &\geq \Big(\frac{\lambda}{1 + C}\Big)^{\alpha \log_\lambda(\sigma)} (1-hc_0)^\frac{2}{3}(1-c_i).
    \end{align*}
    We impose one last restriction to $\alpha$ by setting
    \begin{align*}
        \alpha \leq \frac{1}{\log_\lambda(\sigma)} \cdot \frac{\frac{1}{3}\ln\big(1-\frac{1}{2}c_0\big)}{\ln\big(\frac{\lambda}{1+C}\big)},
    \end{align*}
    which implies a lower bound 
    \begin{align*}
        \Big(\frac{\lambda}{1 + C}\Big)^{\alpha \log_\lambda(\sigma)} \geq (1-\frac{1}{2}c_0)^\frac{1}{3}\geq (1-hc_0)^\frac{1}{3}.
    \end{align*}
    Thus, we naturally obtain
    \begin{align*}
        K \delta_{i+1}^\alpha \geq (1-hc_0)(1-c_i) = 1- c_{i+1},
    \end{align*}
    which proves the second inequality of the induction. Moreover, we can fix
    \begin{align*}
        \alpha = \min\bigg\{\frac{\mu}{2},\frac{1}{2 \log_\lambda(\sigma)} \cdot \frac{\ln\big(1-\frac{1}{2}c_0\big)}{\ln\big(\frac{1}{2}(1-\lambda)\big)}, \frac{1}{\log_\lambda(\sigma)} \cdot \frac{\frac{1}{3}\ln\big(1-\frac{1}{2}c_0\big)}{\ln\big(\frac{\lambda}{1+C}\big)}\bigg\}.
    \end{align*}

    Finally, we observe that by (\ref{eq: termination => pd lambda estimate}) the ratio between $\delta_{i+1}$ and $\delta_i$ is bounded below by
    \begin{align*}
        \frac{\delta_{i+1}}{\delta_i} = \sigma^{H_i} &=  (\lambda^{H_i})^{\log_\lambda(\sigma)} \geq \Big(\frac{\lambda(1-\lambda)(1-h)}{1 + C}\Big)^{\log_\lambda(\sigma)} = \eta >0
    \end{align*}
    for each $i\in \N$. On the other hand, by (\ref{eq: termination => pd lambda estimate})
    \begin{align*}
        \lim_{i \rightarrow \infty} \delta_i = \lim_{i \rightarrow \infty}\bigg(\delta_0 \prod_{j=0}^{i-1} \sigma^{H_j} \bigg)\leq \delta_0 \lim_{i \rightarrow \infty} \Big(\big((1-\lambda)(1-h)\big)^{\log_\lambda(\sigma)}\Big)^i = 0.
    \end{align*}
    The proof is now complete by Proposition \ref{prop: alpha improvement eqv}.
\end{proof}

\subsection{Full characterization}
Finally, we have the full characterization of the parabolic weak porosity by gathering the results of the paper. Our main theorem shows a quantitative and complete connection between the parabolic weakly porous sets with positive time-lag and $A_1^+(\gamma)$ distance weights for $0<\gamma<1$.

\begin{theorem} \label{theo: grand theorem}
    Suppose $E \subseteq \Rn$ is a nonempty closed set, $\theta>1$ and  $0 <\gamma <1$. Then, the following statements are equivalent:
    \begin{enumerate}[label=(\roman*)]
        \item $E$ is $(c,\delta,\theta)$-weakly porous for some $c,\delta \in (0,1)$. \label{item: grand theorem 1}
        \item $E$ is $\alpha$-improving parabolic weakly porous set for some $\alpha>0$. \label{item: grand theorem 2}
        \item $\dist_p\big(\cdot,E\big)^{-\beta(n+p)} \in A_1^+(\gamma)$ for some $\beta>0$. \label{item: grand theorem 3}
    \end{enumerate}
    Moreover, the dependencies are quantitative.
\end{theorem}

\begin{proof}
    Theorem \ref{theo: A1 => wp} shows the direction \ref{item: grand theorem 3} $\Rightarrow$ \ref{item: grand theorem 2}. On the other hand, \ref{item: grand theorem 2} $\Rightarrow$ \ref{item: grand theorem 1} follows directly from Definition \ref{def: alpha improvement}. 
    
    Starting from \ref{item: grand theorem 1}, we assume $E$ is $(c,\delta,\theta)$-weakly porous for some $c,\delta \in (0,1)$ and $\theta>1$. Corollary \ref{cor: translation} shows that $E$ is $(c,\sigma,\Phi)$-weakly porous for $\sigma \in (0,1)$ and $\Phi>1$ from Lemma \ref{lem: parameters}. Then, Lemma \ref{lem: wp => polynomial decay} implies that $E$ is also $\alpha$-improving parabolic weakly porous set for $\alpha>0$.  Finally, Corollary \ref{cor: alpha improvement => A1} shows that $\dist_p\big(\cdot,E\big)^{-\beta(n+p)} \in A_1^+(\gamma)$ for some $\beta >0$, which proves that \ref{item: grand theorem 1} $\Rightarrow$ \ref{item: grand theorem 3}.
\end{proof}


\begin{thebibliography}{10}

\bibitem{Aikawa1991}
Hiroaki Aikawa, \emph{Quasiadditivity of {R}iesz capacity}, Math. Scand. \textbf{69} (1991), no.~1, 15--30.

\bibitem{ACDT2014}
Hugo. Aimar, Marilina. Carena, Ricardo Dur\'an, and Marisa Toschi, \emph{Powers of distances to lower dimensional sets as {M}uckenhoupt weights}, Acta Math. Hungar. \textbf{143} (2014), no.~1, 119--137.

\bibitem{AGIM25}
Hugo Aimar, Ivana G\'omez, Ignacio G\'omez~Vargas, and Francisco~Javier Mart\'in-Reyes, \emph{One-sided {M}uckenhoupt weights and one-sided weakly porous sets in {$\mathbb{R}$}}, J. Funct. Anal. \textbf{289} (2025), no.~10, Paper No. 111110, 18.

\bibitem{AimarGomezVargas24}
Hugo Aimar, Ivana G{\'o}mez, and Ignacio~G{\'o}mez Vargas, \emph{Weakly porous sets and $ {A}_1 $ {M}uckenhoupt weights in spaces of homogeneous type}, arXiv preprint arXiv:2406.14369 (2024).

\bibitem{ALMV24}
Theresa~C. Anderson, Juha Lehrb\"ack, Carlos Mudarra, and Antti~V. V\"ah\"akangas, \emph{Weakly porous sets and {M}uckenhoupt {$A_p$} distance functions}, J. Funct. Anal. \textbf{287} (2024), no.~8, Paper No. 110558, 34.

\bibitem{CKYY2026}
Mingming Cao, Weiyi Kong, Dachun Yang, Wen Yuan, and Chenfeng Zhu, \emph{Parabolic extrapolation and its applications to characterizing parabolic {BMO} spaces via parabolic fractional commutators}, J. Geom. Anal. \textbf{36} (2026), no.~2, Paper No. 65, 37.

\bibitem{Mudarra25}
Mudarra Carlos, \emph{Weak porosity on metric measure spaces}, Proc. Roy. Soc. Edinburgh Sect. A (2025), 1--48.

\bibitem{Cruz-Uribe2025Two-weightFunctions}
David Cruz-Uribe and Kim Myyryl{\"{a}}inen, \emph{Two-weight norm inequalities for parabolic fractional maximal functions}, Rev. Mat. Complut. (2025), 1--30.

\bibitem{DuranGarcia2010}
Ricardo~G. Dur\'an and Fernando L\'opez~Garc\'ia, \emph{Solutions of the divergence and analysis of the {S}tokes equations in planar {H}\"older-{$\alpha$} domains}, Math. Models Methods Appl. Sci. \textbf{20} (2010), no.~1, 95--120.

\bibitem{DILLTV2019}
Bart{\l}omiej Dyda, Lizaveta Ihnatsyeva, Juha Lehrb\"ack, Heli Tuominen, and Antti~V. V\"ah\"akangas, \emph{Muckenhoupt {$A_p$}-properties of distance functions and applications to {H}ardy-{S}obolev--type inequalities}, Potential Anal. \textbf{50} (2019), no.~1, 83--105.

\bibitem{GianazzaVespri2006}
Ugo Gianazza and Vincenzo Vespri, \emph{A {H}arnack inequality for solutions of doubly nonlinear parabolic equations}, J. Appl. Funct. Anal. \textbf{1} (2006), no.~3, 271--284.

\bibitem{Vargas2026}
Ignacio G\'omez~Vargas, \emph{New characterizations of {M}uckenhoupt {$A_p$} distance weights for {$p>1$}}, J. Math. Anal. Appl. \textbf{556} (2026), no.~1, Paper No. 130091, 27.

\bibitem{KinnunenKuusi07}
Juha Kinnunen and Tuomo Kuusi, \emph{Local behaviour of solutions to doubly nonlinear parabolic equations}, Math. Ann. \textbf{337} (2007), no.~3, 705--728.

\bibitem{KinnunenMyyry2024}
Juha Kinnunen and Kim Myyryl\"ainen, \emph{Characterizations of parabolic {M}uckenhoupt classes}, Adv. Math. \textbf{444} (2024), Paper No. 109612, 57.

\bibitem{KinnunenMyyry2025}
\bysame, \emph{Characterizations of parabolic reverse {H}\"older classes}, J. Anal. Math. \textbf{155} (2025), no.~2, 609--656.

\bibitem{KinnunenMyyryYang2024}
Juha Kinnunen, Kim Myyryl\"ainen, and Dachun Yang, \emph{John-{N}irenberg inequalities for parabolic {BMO}}, Math. Ann. \textbf{387} (2023), no.~3-4, 1125--1162.

\bibitem{kinnunenSaariMuckenhoupt}
Juha Kinnunen and Olli Saari, \emph{On weights satisfying parabolic {M}uckenhoupt conditions}, Nonlinear Anal. \textbf{131} (2016), 289--299.

\bibitem{kinnunenSaariParabolicWeighted}
\bysame, \emph{Parabolic weighted norm inequalities and partial differential equations}, Anal. PDE \textbf{9} (2016), no.~7, 1711--1736.

\bibitem{KongYangYuan26}
Weiyi Kong, Dachun Yang, and Wen Yuan, \emph{One-sided and parabolic {BLO} spaces with time lag and their applications to {M}uckenhoupt {$A_1$} weights and doubly nonlinear parabolic equations}, arXiv preprint arXiv:2602.09741 (2026).

\bibitem{KongYangYuanZhu2025}
Weiyi Kong, Dachun Yang, Wen Yuan, and Chenfeng Zhu, \emph{Parabolic muckenhoupt weights characterized by parabolic fractional maximal and integral operators with time lag}, Canad. J. Math. (2025), 1--54.

\bibitem{Luukkainen1998}
Jouni Luukkainen, \emph{Assouad dimension: antifractal metrization, porous sets, and homogeneous measures}, J. Korean Math. Soc. \textbf{35} (1998), no.~1, 23--76.

\bibitem{MaHeYan23}
Jiao Ma, Qianjun He, and Dunyan Yan, \emph{Weighted characterization of parabolic fractional maximal operator}, Front. Math. \textbf{18} (2023), no.~1, 185--196.

\bibitem{PasquarielloUriarte-Tuero2025}
Marcus Pasquariello and Ignacio Uriarte-Tuero, \emph{Medians, oscillations, and distance functions}, arXiv preprint arXiv:2507.21020 (2025).

\bibitem{Vasin05}
Andrei~V. Vasin, \emph{The limit set of a {F}uchsian group and the {D}yn'kin lemma}, J. Math. Sci.(N. Y.) \textbf{129} (2005), no.~4, 977--3984.

\bibitem{Vasin2025}
\bysame, \emph{On characterization of weakly porous sets via dyadic coverings}, J. Anal. Math. \textbf{157} (2025), no.~2, 789--802.

\end{thebibliography}
\end{document}